\documentclass[10pt, letterpaper, headings=small]{scrartcl}

\setkomafont{disposition}{\normalfont\rmfamily\bfseries}
\addtokomafont{caption}{\footnotesize}
\addtokomafont{captionlabel}{\footnotesize}
\setcapindent{1em}
\deffootnote[1em]{0em}{1em}{\textsuperscript{\thefootnotemark}}
\usepackage{amsthm, amsmath, amsfonts, amssymb, paralist, addlines, microtype}
\allowdisplaybreaks\newcommand{\pre}[1]{}\newcommand{\pub}[1]{#1}
\newcommand{\changes}[1]{#1}
\numberwithin{equation}{section}
\allowdisplaybreaks
\newtheoremstyle{thm}{3pt}{3pt}{\itshape}{}{\bfseries}{.}{ }{}
\newtheoremstyle{rest}{3pt}{3pt}{}{}{\bfseries}{.}{ }{}
\newtheoremstyle{comment}{3pt}{3pt}{}{}{\itshape}{.}{ }{}
\swapnumbers
\theoremstyle{thm}
\newtheorem{theorem}{Theorem}[section]
\newtheorem{lemma}[theorem]{Lemma}
\newtheorem{corollary}[theorem]{Corollary}

\newtheorem{assumption}[theorem]{Assumption}
\theoremstyle{rest}
\newtheorem{algorithm}[theorem]{Algorithm}

\newtheorem{experiment}[theorem]{Experiment}

\theoremstyle{comment}
\newtheorem*{comments}{Remarks}
\newtheorem*{comment}{Remark}

\DeclareMathOperator{\ev}{\mathrm{E}}
\DeclareMathOperator{\prob}{\mathrm{P}}
\DeclareMathOperator{\mean}{\mathbb{E}}
\DeclareMathOperator{\ep}{\mathbb{G}}
\DeclareMathOperator*{\argmin}{arg\,min}
\DeclareMathOperator*{\pto}{\to}

\DeclareMathOperator{\corr}{Corr}
\DeclareMathOperator{\var}{Var}
\DeclareMathOperator{\diam}{diam}
\DeclareMathOperator{\diag}{diag}
\DeclareMathOperator{\rank}{rank}
\DeclareMathOperator{\se}{se}

\newcommand{\Beta}{\mathit{B}}
\newcommand{\Tau}{\mathit{T}}

\usepackage[pdftex,usenames,dvipsnames]{color}
\usepackage[pdftex]{graphicx}

\usepackage[longnamesfirst]{natbib}
\defcitealias{hendrickskoenker1992}{Hendricks-Koenker}
\defcitealias{wordetal1990}{Word et al.\ (1990)}

\usepackage[bookmarks=false, pdfpagemode=UseNone, pdftex, colorlinks, 
allcolors=black,
pdfauthor=Andreas~Hagemann, pdfstartview=FitH]{hyperref}

\title{\normalsize\uppercase{Cluster-Robust Bootstrap Inference in\\ Quantile Regression Models}}

\pub{\author{\normalsize Andreas Hagemann\thanks{Department of Economics, University of Michigan, 611 Tappan Ave, Ann Arbor, MI 48109, USA. Tel.: +1 (734) 764-2355. Fax: +1 (734) 764-2769. E-mail: \href{mailto:hagem@umich.edu}{\texttt{hagem@umich.edu}}. I would like to thank Roger Koenker for his continued support and encouragement. I would also like to thank the co-editor, an associate editor, two referees, Matias Cattaneo, S\'ilvia Gon\c{c}alves, Carlos Lamarche, Sarah Miller, Stephen Portnoy, Jo\~ao Santos Silva, Ke-Li Xu, and Zhou Zhou for comments and discussions. All errors are my own.}}}
\pre{\author{\normalsize Andreas Hagemann\thanks{Andreas Hagemann is Assistant Professor, Department of Economics, University of Michigan, 611 Tappan Ave, Ann Arbor, MI 48109, USA. Tel.: +1 (734) 764-2355. Fax: +1 (734) 764-2769. E-mail: \href{mailto:hagem@umich.edu}{\texttt{hagem@umich.edu}}. The author would like to thank Roger Koenker for his continued support and encouragement. The author would also like to thank the co-editor, an associate editor, two referees, Matias Cattaneo, S\'ilvia Gon\c{c}alves, Carlos Lamarche, Sarah Miller, Stephen Portnoy, Jo\~ao Santos Silva, Ke-Li Xu, and Zhou Zhou for comments and discussions.}}}
\date{\pub{\normalsize\today}}

\begin{document}
\maketitle
\begin{abstract}\small
In this paper I develop a wild bootstrap procedure for cluster-robust inference in linear quantile regression models. I show that the bootstrap leads to asymptotically valid inference on the entire quantile regression process in a setting with a large number of small, heterogeneous clusters and provides consistent estimates of the asymptotic covariance function of that process. The proposed bootstrap procedure is easy to implement and performs well even when the number of clusters is much smaller than the sample size. An application to Project STAR data is provided.
\\
\\
\pub{\emph{JEL classification}: C01, C15, C21 \\}
\pub{\emph{Keywords}: quantile regression, cluster-robust standard errors, bootstrap}
\pre{\emph{Keywords}: quantile, cluster-robust standard errors, resampling, correlated data, grouped data}
\end{abstract}

\pre{\thispagestyle{empty}\newpage\setcounter{page}{1}}
\pre{\begin{bibunit}[chicago]}

\section{Introduction}
It is common practice in economics and statistics to conduct inference that is robust to within-cluster dependence. Examples of such clusters are households, classrooms, firms, cities, or counties. We have to expect that units within these clusters influence one another or are influenced by the same sociological, technical, political, or environmental shocks. To account for the presence of data clusters, the literature frequently recommends inference using cluster-robust versions of the bootstrap; see, among many others, \citet{bertrandetal2004} and \citet{cameronetal2008} for an overview in the context of linear regression models estimated by least squares. In this paper I develop a bootstrap method for cluster-robust inference in linear quantile regression (QR) models. The method, which I refer to as \emph{wild gradient bootstrap}, is an extension of a wild bootstrap procedure proposed by \citet{chenetal2003}. Despite the fact that it involves resampling the QR gradient process, the wild gradient bootstrap is fast and easy to implement because it does not involve finding zeros of the gradient during each bootstrap iteration. I show that the wild gradient bootstrap allows for the construction of asymptotically valid bootstrap standard errors, hypothesis tests both at individual quantiles or over ranges of quantiles, and confidence bands for the QR coefficient function.

Quantile regression, introduced by \citet{koenkerbassett1978}, has become an important empirical tool because it enables the researcher to quantify the effect of a set of covariates on the entire conditional distribution of the outcome of interest. This is in sharp contrast to conventional mean regression methods, where only the conditional mean can be considered. A disadvantage of QR in comparison to least squares methods is that the asymptotic variance of the QR coefficient function is notoriously difficult to estimate due to its dependence on the unknown conditional density of the response variable. An analytical estimate of the asymptotic variance therefore requires a user-chosen kernel and bandwidth. Hence, two researchers working with the same data can arrive at different conclusions simply because they used different kernels or bandwidths. Furthermore, a common concern in applied work is that analytical estimates of asymptotic variances perform poorly in the cluster context when the number of clusters is small or the within-cluster correlation is high; see, e.g., \citet{bertrandetal2004} and \citet{webb2013} for extensive Monte Carlo evidence in the least squares case. Their overall finding is that true null hypotheses are rejected far too often. Simulations in \citet{mackinnonwebb2014} suggest that similar problems also occur when the cluster sizes differ substantially.

I show that the wild gradient bootstrap is robust to each of these concerns: It performs well even when the number of clusters is small, the within-cluster dependence is high, and the cluster sizes are heterogenous. The wild gradient bootstrap consistently estimates the asymptotic distribution and covariance functions of the QR coefficients without relying on kernels, bandwidths, or other user-chosen parameters. As such, this paper complements recent work by \citet{parentesantossilva2013}, who provide analytical, kernel-based covariance matrix estimates for QR with cluster data. Their estimates have the advantage that they are simpler to compute than the bootstrap procedures presented here. However, as the Monte Carlo study in this paper shows, a drawback is that tests based on their covariance matrix estimates can suffer from severe size distortions in the same situations as those described for the mean regression case above. In addition, \citeauthor{parentesantossilva2013}'s method does not allow for uniform inference across quantiles because the limiting QR process generally has an analytically intractable distribution. In contrast, the bootstrap approximations of the distribution and covariance functions developed here can be combined to perform uniform Wald-type inference about the QR coefficient function. 

The wild bootstrap procedure discussed in this paper was first introduced by \citet{chenetal2003} as a way to construct confidence intervals for QR coefficients at a single quantile. However, they only provide heuristic arguments for the consistency of the bootstrap approximation and note that ``as far as [they] know, there is no analytical proof that the bootstrap method is valid for the general quantile regression model with correlated observations.'' I considerably extend the scope of their method under explicit regularity conditions to allow for inference on the entire QR process and uniformly consistent covariance matrix estimates of that process. Some parts of the proofs below rely on a recent result by \citet{kato2011} regarding the convergence of bootstrap moments. In turn, his results build on a technique developed in \citet{alexander1985} and tail bounds for the empirical process given in \citet{vandervaartwellner1996}. I also utilize empirical process results of \citet{pollard1991b} and \citet{kosorok2003} to address some nonstandard issues arising from the fact that I allow clusters to be heterogeneous both in terms of their size and their distributions.

\changes{Other types of wild bootstrap procedures for QR are given in \citet{fengetal2010} and \citet{davidson2012}. They do not deal with cluster data but their methods are likely to generalize in this direction. Although typically only convergence of the bootstrap distribution is shown, these and other bootstrap methods have been suggested in the literature as ways to construct bootstrap standard errors or, more generally, bootstrap covariance matrix estimates. \citet{hahn1995} and \citet{goncalveswhite2005} explicitly caution against such conclusions because convergence in distribution does not imply convergence of moments without uniform integrability conditions. This paper establishes these conditions for QR estimates in the cluster context. As I show in my Monte Carlo study, the availability of a bootstrap covariance matrix estimate is crucial for bootstrap tests to have good size and power in many empirically relevant situations.}

Cluster-robust inference in linear regression has a long history in economics; see \citet{cameronmiller2014} for a recent survey. \citet{kloek1981} is an early reference. \citet{moulton1990} highlights the importance of correcting inference for within-cluster correlation when covariates do not vary within clusters. However, apart from \citet{chenetal2003} and \citet{parentesantossilva2013}, cluster inference in QR models has not received much attention. Notable exceptions are \citet{wanghe2007} and \citet{wang2009}, who develop methods for cluster-robust quantile rank score inference, \citet{chetverikovetal2013}, who introduce a method for instrumental variables estimation in a QR model with cluster-level treatment, and \citet{yoongalvao2013}, who discuss QR in a panel model where clusters arise from correlation of individual units over time.

The paper is organized as follows: Section \ref{s:qreg} states and discusses several assumptions that are then used to establish the large sample distribution of the QR estimator with cluster data. Section \ref{s:boot} introduces the wild gradient bootstrap procedure and shows how it can be applied to conduct bootstrap inference on the QR process. Section \ref{s:mc} illustrates the finite-sample behavior of the wild gradient bootstrap in three Monte Carlo experiments. Section \ref{s:star} contains a brief application of the proposed bootstrap procedure to Project STAR data. Section \ref{s:conc} concludes. The appendix contains auxiliary results and proofs. 

%\pub{\addlines}

I use the following notation throughout the paper: $|\cdot|$ is Euclidean norm and $1\{\cdot\}$ is the indicator function. Limits are as $n\to\infty$ unless otherwise noted and convergence in distribution is indicated by $\leadsto$.

\section{Quantile Regression with Clustered Data}\label{s:qreg}

This section discusses linear QR with clustered data and outlines the basic assumptions used to justify asymptotic inference. Then I establish weak convergence of the QR estimator in the cluster context.

Regression quantiles in a framework with cluster data express the conditional quantiles of a response variable $Y_{ik}$ in terms of an observed covariate vector $X_{ik}$. Here $i$ indexes clusters and $k$ indexes observations within that cluster. There are $n$ clusters and cluster $i$ has $c_i$ observations. The cluster sizes $c_i$ need not be identical across $i$, but will be taken to be small relative to $n$ for the asymptotic theory. Because there is typically no natural ordering of observations in the same cluster (unless $k$ indexes time), I allow for arbitrary within-cluster dependence of the data. 
\begin{assumption}\label{as:data}
For all $i,j \geq 1$ with $i\neq j$ and all $1\leq k\leq c_i$, $1\leq l\leq c_j$, the random vectors $(Y_{ik}, X_{ik}^\top)^\top$ and $(Y_{jl}, X_{jl}^\top)^\top$ are independent. The cluster sizes $c_i$ are bounded by some $c_{\max} <\infty$ uniformly in $i\geq 1$.
\end{assumption}

The $\tau$th quantile function of $Y_{ik}$ conditional on $X_{ik}=x$ is given by $Q_{{ik}}(\tau\mid x) := \inf\{y: \prob(Y_{ik}\leq y\mid X_{ik} = x)\geq \tau\}$, where $\tau\in(0,1)$. I assume that the linear QR framework is an appropriate model for the data.
\begin{assumption}\label{as:model}
For $\{Y_{ik} : i\geq 1, 1\leq k\leq c_i\}$, the $\tau$th quantile function satisfies $$Q_{{ik}}(\tau\mid x) = x^\top \beta(\tau),\qquad x\in\mathcal{X}_{ik}\subset \mathbb{R}^d,~\tau\in\Tau,$$ where $\mathcal{X}_{ik}$ is the support of $X_{ik}$ and $\Tau$ is a closed subset of $(0,1)$. For all $\tau\in\Tau$, $\beta(\tau)$ is contained in the interior of a compact and convex set $\Beta\subset\mathbb{R}^d$.
\end{assumption}
\begin{comments}
\begin{inparaenum}[(i)] 
\item Because $\tau\mapsto\beta(\tau)$ does not depend on $i$, this assumption implicitly rules out cluster-level ``fixed effects'' as they would lead to incidental parameter problems; see \citet{koenker2004}. It does \emph{not} rule out covariates that vary at the cluster level and, more importantly, fixed effects for levels above the cluster level. For example, the application in Section \ref{s:star} has classroom-level clusters and school-level fixed effects. There are several ways to address the incidental parameters problem when more is known about the dependence structure in the data; see \citet{yoongalvao2013} and the references therein.

\item The assumption of compactness of $\Beta$ has no impact on the estimation of the QR model in practice because $\Beta$ can always be viewed as large. Compactness is, however, essential for the validity of bootstrap moment estimates in the QR context; see the discussion below Theorem \ref{th:bootse} in the next section.
\end{inparaenum}
\end{comments}

Estimates of the unknown QR coefficient function $\tau\mapsto \beta(\tau)$ can be computed with the help of the \citet{koenkerbassett1978} check function $\rho_\tau(z) = (\tau - 1\{z < 0\})z$. For clustered data, the QR problem minimizes
\begin{equation*}
\mathbb{M}_{n}(\beta,\tau) := \frac{1}{n}\sum_{i=1}^n\sum_{k=1}^{c_i}\rho_{\tau}(Y_{ik} - X_{ik}^\top \beta)% = \mean_n m_{\beta,\tau}
\end{equation*}
with respect to $\beta$ for a given $\tau$ so that $\tau \mapsto \beta(\tau)$ is estimated by
\begin{equation*}
\tau\mapsto \hat{\beta}_n(\tau) := \argmin_{\beta\in\Beta}\mathbb{M}_{n}(\beta,\tau), \qquad \tau\in\Tau.
\end{equation*}
The main goal of this paper is to provide a method for cluster-robust bootstrap inference about the QR coefficient function that is valid uniformly on the entire set $\Tau$ and leads to cluster-robust bootstrap covariance matrix estimates for $\tau\mapsto \hat{\beta}_n(\tau)$. The validity of this method relies on the asymptotic normal approximation to the distribution of $\sqrt{n}(\hat{\beta}_n(\tau) - \beta(\tau))$, which I turn to next.

Asymptotic normality of the QR estimates in the cluster context requires straightforward extensions of smoothness and moment conditions familiar from the iid case. The following assumption allows for arbitrary heterogeneity of the clusters as long as the observations within these clusters satisfy mild restrictions on the smoothness of their conditional distributions and the tail behavior of the covariates. This assumption is needed to ensure identification of the QR coefficient function and to justify an approximation argument following immediately below. 
\begin{assumption}\label{as:smooth}
\begin{compactenum}[\upshape(i)]
\item\label{as:smooth:mom} $\ev |X_{ik}|^q< \infty$ uniformly in $i\geq 1$ and $1\leq k\leq c_i$ for some $q > 2$, 
\item $n^{-1} \sum_{i=1}^n \sum_{k=1}^{c_i} \ev X_{ik}X_{ik}^\top$ is positive definite, uniformly in $n$,
\item\label{as:smooth:den} the conditional density $f_{{ik}}(y\mid X_{ik}=x)$ of $Y_{ik}$ and its derivative in $y$ are bounded above uniformly in $y$ and $x\in\mathcal{X}_{ik}$, uniformly in $i\geq 1$ and $1\leq k\leq c_i$, and
\item\label{as:smooth:der} $f_{{ik}}(x^\top\beta \mid X_{ik}=x)$ is bounded away from zero uniformly in $\beta\in\Beta$ and $x\in\mathcal{X}_{ik}$, uniformly in $i\geq 1$ and $1\leq k\leq c_i$.%\marginpar{!}
\end{compactenum}
\end{assumption}
To establish distributional convergence of the QR estimator, I consider the recentered population objective function $\beta\mapsto M_n(\beta,\tau) := \ev (\mathbb{M}_{n}(\beta,\tau) - \mathbb{M}_{n}(\beta(\tau),\tau))$. The recentering ensures that $M_n$ is well defined without moment conditions on the response variable. Provided Assumptions \ref{as:model} and \ref{as:smooth} hold, the map $\beta\mapsto M_n(\beta,\tau)$ is differentiable with derivative $M_n'(\beta, \tau) := \partial M_n(\beta,\tau) / \partial \beta^\top$ and achieves its minimum at $\beta(\tau)$ by convexity. I show in the appendix that under Assumptions \ref{as:data}-\ref{as:smooth} we can write
\begin{equation}\label{eq:taylorex}
0 = \sqrt{n} M_n'(\beta(\tau), \tau) =  \sqrt{n} M_n'(\hat{\beta}_n(\tau), \tau) - J_n(\tau)\sqrt{n}\bigl(\hat{\beta}_n(\tau) - \beta(\tau)\bigr) + o_{\prob} (1)
\end{equation}
uniformly in $\tau\in\Tau$ by a mean value expansion about $\hat{\beta}_n(\tau)$. Here 
$J_n(\tau)$ is the Jacobian matrix of the expansion evaluated at $\beta(\tau)$,
$$ J_n(\tau) = \frac{1}{n} \sum_{i=1}^n \sum_{k=1}^{c_i} \ev f_{ik}\bigl( X_{ik}^\top\beta(\tau) \mid X_{ik}\bigr) X_{ik}X_{ik}^\top.$$

%\pub{\addlines}

After rearranging \eqref{eq:taylorex}, a stochastic equicontinuity argument (see the appendix for details) can be used to show that $\sqrt{n} M_n'(\hat{\beta}_n(\tau), \tau)$ is, uniformly in $\tau\in\Tau$, within $o_{\prob} (1)$ of the first term on the right of
\begin{equation} 
J_n(\tau)  \sqrt{n}\bigl(\hat{\beta}_n(\tau) - \beta(\tau)\bigr) = \frac{1}{\sqrt{n}} \sum_{i=1}^n \sum_{k=1}^{c_i} \psi_\tau\bigl(Y_{ik}-X_{ik}^\top\beta(\tau)\bigr) X_{ik} + o_{\prob}(1),\label{eq:bahadur}
\end{equation}
where $\psi_\tau(z) = \tau - 1\{z < 0\}$. The outer sum on the right-hand side can be viewed as an empirical process evaluated at functions of the form $\sum_{k=1}^{c_i} \psi_\tau(Y_{ik}-X_{ik}^\top\beta(\tau)) X_{ik}$ indexed by $\tau\in\Tau$.\footnote{In view of Assumption \ref{as:data}, we can always take $(Y_{ik},X_{ik}^\top)^\top=0$ for $c_i<k\leq c_{\max}$ whenever $c_i < c_{\max}$ to make this a well-defined class of functions from $\mathbb{R}^{c_{\max}}\times\mathbb{R}^{d\times c_{\max}}$ to $\mathbb{R}^d$.} This empirical process has covariances
$$\Sigma_n(\tau,\tau') := \frac{1}{n} \sum_{i=1}^n \sum_{k=1}^{c_i}\sum_{l=1}^{c_i} \ev \psi_\tau\bigl(Y_{ik}-X_{ik}^\top\beta(\tau)\bigr) \psi_{\tau'}\bigl(Y_{il}-X_{il}^\top\beta(\tau')\bigr)X_{ik}X_{il}^\top, \qquad \tau,\tau'\in\Tau.$$
\phantomsection%
\label{rev:liangzeger}%
\changes{As a referee points out, similar covariances arise in the generalized estimating equations framework of \citet{liangzeger1986}.}

In the absence of clusters (i.e., $c_i\equiv 1$), $\Sigma_n(\tau,\tau')$ reduces to the familiar form of $n^{-1}\sum_{i=1}^n \ev X_{i1}X_{i1}^\top$ times the covariance function $(\min\{\tau,\tau'\}-\tau\tau')I_d$ of the standard $d$-dimensional Brownian bridge, where $I_d$ is the identity matrix of size $d$. Because of the within-cluster dependence, the structure of $\Sigma_n(\tau,\tau')$ is now significantly more involved. This does not change in the limit as $n\to\infty$, which is assumed to exist along with the limit of the Jacobian.
\begin{assumption}\label{as:gstat}
$J(\tau) = \lim_{n\to\infty} J_n(\tau)$, $\Sigma(\tau,\tau') = \lim_{n\to\infty} \Sigma_n(\tau,\tau')$ exist for $\tau,\tau'\in\Tau$.
\end{assumption}
\begin{comment}
\noindent I show in the appendix that under Assumptions \ref{as:data}-\ref{as:smooth} the pointwise convergence in Assumption \ref{as:gstat} already implies uniform convergence of $J_n$ and $\Sigma_n$.
\end{comment}

The matrix limit $J(\tau)$ is positive definite by Assumption \ref{as:smooth}. Hence, the asymptotic distribution of $\sqrt{n}(\hat{\beta}_n(\tau) - \beta(\tau))$ can be determined via the continuous mapping theorem and an application of a central limit theorem to the right-hand side of \eqref{eq:bahadur}. The following theorem confirms that this even remains valid when $\tau \mapsto \sqrt{n}(\hat{\beta}_n(\tau) - \beta(\tau))$ is viewed as a random process indexed by $\Tau$. The distributional convergence then occurs relative to $\ell^\infty(T)^d$, the class of uniformly bounded functions on $\Tau$ with values in $\mathbb{R}^d$.
\begin{theorem}\label{th:clt}
Suppose Assumptions \ref{as:data}-\ref{as:gstat} hold. Then $\{\sqrt{n}(\hat{\beta}_n(\tau) - \beta(\tau)): \tau\in\Tau\}$ converges in distribution to a mean-zero Gaussian process $\{\mathbb{Z}(\tau):\tau\in\Tau\}$ with covariance function $\ev\mathbb{Z}(\tau)\mathbb{Z}(\tau')^\top= J^{-1}(\tau)\Sigma(\tau,\tau')J^{-1}(\tau')$, $\tau,\tau'\in\Tau$.
\end{theorem}
\begin{comment}
\begin{inparaenum}[(i)]
\item The proof of the theorem proceeds via empirical process arguments similar to those used in \citet{angristetal2006}. Their results do not carry over to the present case because heterogeneous cluster data is not covered by their iid assumptions. This prevents the use of standard Donsker theorems and leads to measurability issues typically not encountered in such proofs. Both problems are taken care of through results of \citet{pollard1991b} and \citet{kosorok2003}.

\item The theorem implies joint asymptotic normality of $\sqrt{n}(\hat{\beta}_n(\tau_j) - \beta(\tau_j))$ for every finite set of quantile indices $\tau_j\in\Tau$, $j=1, 2, \dots$; see, e.g., Theorem 18.14 of \citet{vandervaart1998}. The asymptotic covariance at $\tau_{j}$ and $\tau_{j'}$ is $J^{-1}(\tau_j)\Sigma(\tau_j,\tau_{j'})J^{-1}(\tau_{j'})$. If only this finite dimensional convergence is needed, then Assumptions \ref{as:smooth}\eqref{as:smooth:den} and \eqref{as:smooth:der} can be relaxed considerably using the approach of \citet{knight1998}.
\end{inparaenum}
\end{comment}

The non-iid structure of the data causes the asymptotic covariance function of $\sqrt{n}(\hat{\beta}_n(\tau) - \beta(\tau))$ to take on the sandwich form $J^{-1}(\tau)\Sigma(\tau,\tau')J^{-1}(\tau')$. Estimates of these covariances are needed for Wald-type inference. However, in addition to the usual problem of having to control the nuisance quantities $f_{ik}( X_{ik}^\top\beta(\tau) \mid X_{ik})$ contained in the Jacobian $J(\tau)$, the matrix $\Sigma(\tau,\tau')$ now also contains products of quantile crossing indicators $\psi_\tau(Y_{ik}-X_{ik}^\top\beta(\tau)) \psi_{\tau'}(Y_{il}-X_{il}^\top\beta(\tau'))$. For standard plug-in inference, the crossing indicators can be estimated by replacing $\tau\mapsto\beta(\tau)$ with $\tau\mapsto\hat{\beta}_n(\tau)$. The Jacobian is not directly affected by the within-cluster dependence and can be estimated using the bandwidth-driven methods of \citet{hendrickskoenker1992} and \citet{powell1986}. 

\citet{parentesantossilva2013} propose such a plug-in estimator based on \citeauthor{powell1986}'s method and show that it leads to asymptotically valid covariance matrix estimates at individual quantiles in a setting with  iid (necessarily equal-sized) clusters.\footnote{Their method is likely to generalize to allow for pointwise inference in the presence of clusters with unequal sizes.} However, both the \citetalias{hendrickskoenker1992} and \citeauthor{powell1986} estimators are sensitive to the choice of bandwidth. \citet[pp.\ 5-6]{parentesantossilva2013} give practical suggestions on how to select this bandwidth in the cluster context, but also note that some standard bandwidth rules derived for iid data seem to not perform well in some contexts. To my knowledge, bandwidth rules for QR that explicitly deal with cluster data are currently not available. Moreover, \citeauthor{parentesantossilva2013}'s method does not extend to uniform inference over ranges of quantiles because the limiting Gaussian process $\{\mathbb{Z}(\tau):\tau\in\Tau\}$ from Theorem \ref{th:clt} is nuisance parameter dependent and cannot be normalized to be free of these parameters. Critical values for inference based on $\mathbb{Z}$ therefore cannot be tabulated. 

In the next section I present a bootstrap method that is able to approximate the distribution of the limiting process, consistently estimates the covariance function of that process, and avoids the issue of choosing a bandwidth (and kernel) altogether. 

\section{Bootstrap Algorithms for Cluster-Robust Inference}\label{s:boot}
In this section I describe and establish the validity of procedures for cluster-robust bootstrap inference (Algorithm \ref{al:binf} below) and cluster-robust confidence bands (Algorithm \ref{al:bcbands}) in QR models. Recall from the discussion above equation \eqref{eq:taylorex} that the population first-order condition of the QR objective function  can be written as
\begin{equation}\label{eq:foc}
\sqrt{n}M'_{n}(\beta(\tau), \tau) =- \frac{1}{\sqrt{n}}\sum_{i=1}^n\sum_{k=1}^{c_i}\ev\psi_{\tau}\bigl(Y_{ik} - X_{ik}^\top\beta(\tau)\bigr)X_{ik} = 0,% = \ev(\mean_n m_{\beta,\tau})
\end{equation}
where $\psi_\tau(z) = \tau - 1\{z < 0\}$ as before. The sample analogue of this condition,
\begin{equation}\label{eq:grad}
\frac{1}{\sqrt{n}}\sum_{i=1}^n\sum_{k=1}^{c_i}\psi_{\tau}(Y_{ik} - X_{ik}^\top\beta)X_{ik} = 0,% = \ev(\mean_n m_{\beta,\tau})
\end{equation}
can be thought of as being nearly solved by the QR estimate $\beta=\hat{\beta}_n(\tau)$.

The idea is now to bootstrap by repeatedly computing solutions to perturbations of \eqref{eq:grad}. To account for the possible heterogeneity in the data, I use \citeauthor{chenetal2003}'s (\citeyear{chenetal2003}) modification of the wild bootstrap \citep{wu1986,liu1988,mammen1992} for QR with correlated data. \citeauthor{chenetal2003}\ use their method to obtain confidence intervals for QR estimators at a single quantile. Here, I considerably extend the scope of their method to allow for inference on the entire QR process and uniformly consistent covariance matrix estimates of that process via the bootstrap; confidence intervals at individual quantiles $\tau_j\in\Tau$, $j=1, 2, \dots$, then follow as a special case. Because \citeauthor{chenetal2003}\ do not give explicit regularity conditions for the validity of their method, this paper also serves as a theoretical justification for their pointwise confidence intervals.

%\pub{\addlines}

To ensure that the bootstrap versions of the QR estimate accurately reflect the within-cluster dependence, the resampling scheme perturbs the gradient condition \eqref{eq:grad} at the cluster level. Let $W_1,\dots,W_n$ be iid copies of a random variable $W$ with $\ev W = 0$,  $\var W=1$, and $\ev|W|^q < \infty$, where $q >2$ as in Assumption \ref{as:smooth}\eqref{as:smooth:mom}. Here $W$ is independent of the data. Define the bootstrap gradient process as
\begin{equation*}
\mathbb{W}_n(\beta,\tau) = \frac{1}{\sqrt{n}}\sum_{i=1}^n W_i \sum_{k=1}^{c_i}\psi_{\tau}(Y_{ik} - X_{ik}^\top\beta)X_{ik}.
\end{equation*}
An obvious strategy for bootstrap resampling would now be to repeatedly solve $\mathbb{W}_n(\beta,\tau) =0$ for $\beta$ with different draws of 
$W_1,\dots,W_n$. However, this type of resampling is impractical because zeros of $\beta\mapsto\mathbb{W}_n(\beta,\tau)$ are difficult to compute %; see also the remarks below Algorithm \ref{al:bse}. 
due to the fact that $\mathbb{W}_n(\beta,\tau) =0$ is not a first-order condition of a convex optimization problem.

Instead, I use the bootstrap gradient process $\mathbb{W}_n(\tau) := \mathbb{W}_n(\hat{\beta}_n(\tau) ,\tau)$ evaluated at the original QR estimate to construct the new objective function
\begin{equation}\label{eq:bootmin}
\beta\mapsto\mathbb{M}^*_{n}(\beta,\tau) = \mathbb{M}_{n}(\beta,\tau) + \mathbb{W}_n(\tau)^\top \beta/\sqrt{n}
\end{equation}
and define the process $\tau\mapsto \hat{\beta}^*_n(\tau)$ as any solution to $\min_{\beta\in\Beta}\mathbb{M}^*_{n}(\beta,\tau)$. Then 
$\hat{\beta}^*_n(\tau)$ can be interpreted as the $\beta$ that nearly solves the corresponding ``first-order condition'' 
\begin{equation*}
\frac{1}{\sqrt{n}}\sum_{i=1}^n\sum_{k=1}^{c_i}\psi_{\tau}(Y_{ik} - X_{ik}^\top\beta)X_{ik} = \mathbb{W}_n(\tau).% = \ev(\mean_n m_{\beta,\tau})
\end{equation*}

This bootstrap, which I refer to as \emph{wild gradient bootstrap}, essentially perturbs the right-hand side of \eqref{eq:grad} instead of the left. Because $\mathbb{W}_n(\tau)$ mimics the original gradient process $n^{-1/2}\sum_{i=1}^n\sum_{k=1}^{c_i}\psi_{\tau}(Y_{ik} - X_{ik}^\top\hat{\beta}_n(\tau))X_{ik}$ just like the original gradient process mimics the population first-order condition \eqref{eq:foc}, choosing $\hat{\beta}^*_n(\tau)$ in such a way induces the left-hand side of the preceding display to match the behavior of $\mathbb{W}_n(\tau)$. The distributions of $\sqrt{n}(\hat{\beta}^*_n(\tau)-\hat{\beta}_n(\tau))$ and $\sqrt{n}(\hat{\beta}_n(\tau) - \beta(\tau))$ can then be expected to be similar. Theorem \ref{th:bootclt} ahead confirms that this is indeed the case, uniformly in $\tau\in\Tau$. 

The distributional convergence occurs both in the standard sense and with probability approaching one, conditional on the sample data $D_n := \{ (Y_{ik}, X_{ik}^\top)^\top : 1\leq k\leq c_i, 1\leq i\leq n\}$. The latter concept is the standard measure of consistency for bootstrap distributions; see, e.g., \citet[p.\ 332]{vandervaart1998}. Let $\mathrm{BL}_1(\ell^\infty(\Tau)^d)$ be the set of functions on $\ell^\infty(\Tau)^d$ with values in $[-1,1]$ that are uniformly Lipschitz and define $\ev^*(\cdot) := \ev(\cdot \mid D_n)$.
\begin{theorem}\label{th:bootclt}
If Assumptions \ref{as:data}-\ref{as:gstat} hold, then $\{\sqrt{n}(\hat{\beta}_n^*(\tau) - \hat{\beta}_n(\tau)):\tau\in\Tau\}$ converges in distribution to the Gaussian process $\{\mathbb{Z}(\tau): \tau\in\Tau\}$ described in Theorem \ref{th:clt}. The convergence also holds conditional on the data in the sense that
$$\sup_{h\in\mathrm{BL}_1(\ell^\infty(\Tau)^d)} \bigl| \ev^* h\bigl(\sqrt{n}\bigl(\hat{\beta}_n^*(\tau) - \hat{\beta}_n(\tau)\bigr)\bigr) - \ev h(\mathbb{Z}(\tau))\bigr| 
%\overset{\prob}{\to} 
\pto^{\prob}
0.$$
\end{theorem}

Minimizing the bootstrap objective function \eqref{eq:bootmin} is %, as opposed to finding zeros of $\beta\mapsto\mathbb{W}_n(\beta,\tau)$, 
a standard convex optimization problem. In fact, as the following algorithm shows, the problem can be implemented in statistical software as a linear QR with one additional observation. The idea is to pick a large enough $Y^*$ to ensure $Y^* > X^{*\top}\hat{\beta}^*_n(\tau)$ for all $\tau\in\Tau$, where $X^* = -\sqrt{n}\mathbb{W}_{n}(\tau)/\tau$. Then $\sqrt{n}\mathbb{W}_n(\tau)^\top \beta = \rho_{\tau}(Y^* - X^{*\top} \beta) - \tau Y^*$ and $- \tau Y^*$ can be ignored because $\hat{\beta}^*_n(\tau)$ not only minimizes \eqref{eq:bootmin}, but also $\beta \mapsto n\mathbb{M}^*_{n}(\beta,\tau)- \tau Y^*$.
\begin{algorithm}[Wild gradient bootstrap]\label{al:bse}
\begin{compactenum}
\item\label{al:bse:1} Run a QR of $Y_{ik}$ on $X_{ik}$ and save $\tau\mapsto \hat{\beta}_n(\tau)$. Compute $Y^* = n \max_{1\leq i \leq n} c_i \max_{1\leq i \leq n, 1\leq k\leq c_i} |Y_{ik}|$.
\item\label{al:bse:2} Draw iid copies $W_1,\dots,W_n$ of $W$ and compute $\mathbb{W}_n(\tau) := \mathbb{W}_n(\hat{\beta}_n(\tau) ,\tau)$ for that draw. Generate $X^* = -\sqrt{n}\mathbb{W}_{n}(\tau)/\tau$ and rerun the QR from Step \ref{al:bse:1} with the additional observation $(Y^*, X^{*\top})^\top$ to obtain
$\tau\mapsto \hat{\beta}^*_{n}(\tau) = \argmin_{\beta}\sum_{i=1}^n\sum_{k=1}^{c_i}\rho_{\tau}(Y_{ik} - X_{ik}^\top \beta) + \rho_{\tau}(Y^* - X^{*\top} \beta).$
\item\label{al:bse:3} Repeat Step \ref{al:bse:2} $m$ times, each with a new realization of $W_1,\dots,W_n$. 

\item Approximate the distribution of $\{\sqrt{n}(\hat{\beta}_n(\tau) - \beta(\tau)) : \tau\in\Tau \}$ by the empirical distribution of the $m$ observations of $\{\sqrt{n}(\hat{\beta}_n^*(\tau) - \hat{\beta}_n(\tau)) : \tau\in\Tau\}$.
\end{compactenum}
\end{algorithm}
%
%\begin{algorithm}[Bootstrap standard errors]\label{al:bse}
%\begin{compactenum}
%\item\label{al:bse:1} Run a quantile regression of $Y_{ik}$ on $X_{ik}$ and save $\tau\mapsto \hat{\beta}_n(\tau)$.
%\item Compute $Y_{\max} = n \max_i c_i \max_{i,k} |Y_{ik}|$ and, for each $b=1,\dots,B$,
%\begin{compactenum}[(i)]
%\item draw a set of iid copies $\{W_1,\dots,W_n\}_b$ of $W$ and compute $\mathbb{W}_{nb}(\tau) := \mathbb{W}_n(\hat{\beta}_n(\tau) ,\tau)$ for that draw,
%\item generate $X_{b} = -\sqrt{n}\mathbb{W}_{nb}(\tau)/\tau$ and rerun the quantile regression from Step \ref{al:bse:1} with the additional observation $(Y_{\max}, X_b^\top)^\top$ to obtain
%\begin{equation*}
%\tau\mapsto \hat{\beta}^*_{nb}(\tau) := \argmin_{\beta\in\Beta}\sum_{i=1}^n\sum_{k=1}^{c_i}\rho_{\tau}(Y_{ik} - X_{ik}^\top \beta) + \rho_{\tau}(Y_{\max} - X_{b}^\top \beta).
%\end{equation*}
% \eqref{eq:bootmin} because $Y_{\max} > X_{b}^\top \hat{\beta}^*_{nb}(\tau)$ for all $\tau\in\Tau$.
%\end{compactenum}
%\item Let $\bar{\beta}^*_{nB}(\tau) = B^{-1}\sum_{b=1}^B\hat{\beta}^*_{nb}(\tau)$ and estimate the covariance function of $\tau\mapsto \hat{\beta}_n(\tau)$ by 
%\begin{equation*}
%(\tau,\tau')\mapsto\frac{1}{B}\sum_{b=1}^B\bigl(\hat{\beta}^*_{nb}(\tau) - \bar{\beta}^*_{nB}(\tau)\bigr)\bigl(\hat{\beta}^*_{nb}(\tau') - \bar{\beta}^*_{nB}(\tau')\bigr)^\top
%\end{equation*}
%\end{compactenum}
%\end{algorithm}

\begin{comments}
\begin{inparaenum}[(i)]
\item The idea of representing a perturbed QR problem as linear QR with one additional observation is due to \citet{parzenetal1994}. The value of $Y^*$ given in the first step of the algorithm is similar to the one suggested by \citet[Algorithm A.4]{bellonietal2011}

\phantomsection%
\label{rev:mammenrec}%
\changes{\item The Monte Carlo experiments in the next section suggest that in practice $W$ should be drawn from the \citet{mammen1992} 2-point distribution that takes on the value $-(\sqrt{5}-1)/2$ with probability $(\sqrt{5}+1)/(2\sqrt{5})$ and the value $(\sqrt{5}+1)/2$ with probability $(\sqrt{5}-1)/(2\sqrt{5})$. Other distributions such as the Rademacher or \citet{webb2013} distributions can be used, but there is no evidence that this would lead to better inference.}
\end{inparaenum}
\end{comments}

By choosing the number of bootstrap simulations $m$ in Algorithm \ref{al:bse} large enough,\footnote{\citet{andrewsbuchinsky2000} and \citet{davidsonmackinnon2000} provide methods for determining an appropriate number of bootstrap simulations $m$ in practice.} the distribution of $\sqrt{n}(\hat{\beta}_n^*(\tau) - \hat{\beta}_n(\tau))$ or functionals thereof can be approximated with arbitrary precision. I therefore let $m\to\infty$ in the following and define the bootstrap estimate of the asymptotic covariance function $V(\tau,\tau') := J^{-1}(\tau)\Sigma(\tau,\tau')J^{-1}(\tau')$ directly as
\begin{equation*}
\hat{V}_n^*(\tau,\tau') = \ev^*n\bigl(\hat{\beta}_n^*(\tau) - \hat{\beta}_n(\tau)\bigr)\bigl(\hat{\beta}_n^*(\tau') - \hat{\beta}_n(\tau')\bigr)^\top, \qquad \tau,\tau'\in\Tau.
\end{equation*}
In practice one simply computes $\hat{V}_n^*(\tau,\tau')$ as the sample covariance of the $m$ bootstrap observations of $\sqrt{n}(\hat{\beta}_n^*(\tau) - \hat{\beta}_n(\tau))$ and $\sqrt{n}(\hat{\beta}_n^*(\tau') - \hat{\beta}_n(\tau'))$. Cluster-robust standard errors of $\hat{\beta}_n(\tau)$ are the square-roots of the diagonal elements of $\hat{V}_n^*(\tau,\tau)/n$.

\pub{\addlines[2]}

Availability of a consistent estimate of the covariance function of the limiting process is not strictly required for valid bootstrap inference on the QR process. Algorithm \ref{al:binf} ahead shows how this is done. However, especially in the presence of data clusters, applied researchers frequently emphasize the importance of bootstrap covariance matrix estimates for Wald-type inference in mean regression models; see, among others, \citet{bertrandetal2004} and \citet{cameronetal2008}. As the Monte Carlo results in the next section show, reweighting by the bootstrap covariance matrix is equally important for cluster-robust inference in the QR context. Still, because convergence in distribution does not imply convergence in moments, consistency of $\hat{V}_n^*(\tau,\tau')$ does not immediately follow from Theorem \ref{th:bootclt}. 

Fortunately, the wild gradient bootstrap is able to consistently approximate the asymptotic variance of $\sqrt{n}(\hat{\beta}_n(\tau) - \beta_n(\tau))$. If the covariates have moments of high enough order, then the approximation of the asymptotic covariance function $V(\tau,\tau')$ through its bootstrap counterpart $\hat{V}_n^*(\tau,\tau')$ is in fact uniform in $\tau,\tau'\in\Tau$. 
\begin{theorem}\label{th:bootse}
Suppose Assumptions \ref{as:data}-\ref{as:gstat} hold. Then,
\begin{compactenum}[\upshape(i)]
\item\label{th:bootse_pw} for all $\tau,\tau'\in\Tau$, $\hat{V}_n^*(\tau,\tau') \pto^{\prob} V(\tau,\tau')$, and

\item\label{th:bootse_uni} if $q>4$ in Assumption \ref{as:smooth}, then $\sup_{\tau,\tau'\in\Tau} | \hat{V}_n^*(\tau,\tau') - V(\tau,\tau') | \pto^{\prob} 0$.
\end{compactenum}
\end{theorem}

\begin{comments}
%\begin{inparaenum}[(i)]
%\item 
(i)~For the proof of this theorem I extend ideas developed by \citet{kato2011}, who in turn relies to some extent on the strategy used in the proof of Theorem 3.2.5 of \citet{vandervaartwellner1996} and \citeauthor{alexander1985}'s (\citeyear{alexander1985}) ``peeling device.'' \citeauthor{kato2011}'s results do not apply to the present case because he works with a single quantile, iid data, and a different bootstrap method. For the proof I develop new tail bounds on the QR gradient process and differences of such processes. They yield $\ev\sup_{\tau\in\Tau}|\sqrt{n}( \hat{\beta}^*_n(\tau) - \hat{\beta}_n(\tau))|^{p}< \infty$ and $\ev\sup_{\tau\in\Tau}|\sqrt{n}( \hat{\beta}_n(\tau) - \beta(\tau))|^{p}< \infty$ uniformly in $n$ for $p < q$. The first part of the theorem then follows from Theorem \ref{th:bootclt} and a bootstrap version of a standard uniform integrability result. The proof of the second part is considerably more involved, but relies on the same tail bounds.

%\item %Denote by $\{\mathbb{Z}(\tau):\tau\in\Tau\}$ the Gaussian limit process described in Theorems \ref{th:clt} and \ref{th:bootclt}. 
(ii)~A byproduct of the proof of the theorem is the result that the wild gradient bootstrap correctly approximates other (possibly fractional) order moments of $\mathbb{Z}(\tau)$ if the covariates have moments of slightly higher order: As long as $p<q$, the results in the appendix immediately give $\ev^* |\sqrt{n}(\hat{\beta}_n^*(\tau) - \hat{\beta}_n(\tau)) |^p \pto^{\prob} \ev |\mathbb{Z}(\tau)|^p$.

%\item 
(iii)~\citet{ghoshetal1984} show that the bootstrap variance estimate of an unconditional quantile can be inconsistent if the bootstrap observations are too likely to take on extreme values. This problem is generic and does not depend on the specific type of bootstrap. The boundedness of the parameter space imposed in Assumption \ref{as:model} prevents such behavior in the bootstrap estimates obtained from the perturbed QR problem \eqref{eq:bootmin}. As \citet{kato2011} points out, a possible (although not particularly desirable) alternative would be to restrict the moments on the response variable.

%\item 
(iv)~Similar but somewhat simpler arguments can be used to prove analogues of Theorems \ref{th:clt}, \ref{th:bootclt}, and \ref{th:bootse} for the bootstrap method presented in \citet{parzenetal1994} for QR with independent data. For iid data, such analogues of Theorems \ref{th:clt} and \ref{th:bootclt} are essentially contained in the results of \citet{bellonietal2011} as special cases.
%\end{inparaenum}
\end{comments}

%\pub{\addlines}

I now turn to inference with the wild gradient bootstrap. Let $\tau \mapsto R(\tau)$ be a continuous, $(h\times d)$-matrix-valued function with $h\leq d$ and let $r\colon\Tau\to\mathbb{R}^d$. Suppose $R(\tau)$ has full rank for every $\tau\in\Tau$. I consider testing general pairs of hypotheses of the form
\begin{align*}
\mathrm{H_0}\colon R(\tau)\beta(\tau)=r(\tau)~\text{for all $\tau\in\Tau$}, \qquad \mathrm{H_1}\colon R(\tau)\beta(\tau)\neq r(\tau)~\text{for some $\tau\in\Tau$}.
\end{align*}
Many empirically relevant hypotheses can be tested with this framework. For example, a standard hypothesis in practice is that a single QR coefficient is zero for all $\tau\in\Tau$. If the coefficient of interest is the first entry of $\beta(\tau)$, then $R(\tau)\equiv (1,0,\dots,0)$ and $r(\tau)\equiv 0$.

For inference I use generalized Kolmogorov-Smirnov statistics. Cram\'er-von-Mises versions of these statistics can be used as well, but are not discussed here to conserve space. For a positive definite weight matrix function $\tau\mapsto\Omega(\tau)$ with positive square root $\Omega^{1/2}(\tau)$, define the test statistic
\begin{equation}\label{eq:ksstat}
%K_n(\Omega,T) = \sup_{\tau\in\Tau} n\bigl(R(\tau)\hat{\beta}_n(\tau)-r(\tau)\bigr)^\top \Omega^{-1}(\tau)\bigl(R(\tau)\hat{\beta}_n(\tau)-r(\tau)\bigr).
K_n(\Omega,T) = \sup_{\tau\in\Tau} \bigl\vert\Omega^{-1/2}(\tau)\sqrt{n}\bigl(R(\tau)\hat{\beta}_n(\tau)-r(\tau)\bigr)\bigr\vert.
\end{equation}
I focus on two versions of the statistic: (i)~an unweighted version with $\Omega(\tau)\equiv I_d$ and (ii)~a Wald-type statistic with $\Omega(\tau)$ equal to $$ \hat{\Omega}_n^*(\tau) := R(\tau)\hat{V}_n^*(\tau,\tau)R(\tau)^\top. $$ Other choices are clearly possible. For example, $\hat{V}_n^*(\tau,\tau)$ can be replaced by any other uniformly consistent estimate of $V(\tau,\tau)$. However, the Monte Carlo study in the next section suggests that option (ii) leads to tests with better finite-sample size and power than tests based on (i) or analytical estimates of $V(\tau,\tau)$.

\pub{\addlines}

In the absence of within-cluster correlation, the process inside the Euclidean norm in \eqref{eq:ksstat} with $\Omega = \hat{\Omega}_n^*$ would converge weakly to a standard vector Brownian bridge. Consequently, $\mathit{K}_n(\hat{\Omega}_n^*,T)$ would converge in distribution to the supremum of a standardized, tied-down Bessel process whose critical values can be simulated or computed exactly; see \citet{koenkermachado1999} for details. In the presence of data clusters, the limiting Gaussian process of the quantity inside the Euclidean norm is no longer a Brownian bridge for any choice of weight matrix. Both $\mathit{K}_n(\hat{\Omega}_n^*,T)$ and $K_n(I_d,\Tau)$ are then, in general, asymptotically non-pivotal statistics.
%The Wald-type weight tends to lead to better size and power in finite samples (see Experiment \ref{ex:KS} in the next section for Monte Carlo evidence), but the within-cluster dependence will generally render both statistics asymptotically non-pivotal. 
Bootstrap tests based on $\mathit{K}_n(\hat{\Omega}_n^*,T)$ therefore do not necessarily outperform tests based on $\mathit{K}_n(I_d,\Tau)$ because of asymptotic refinements; see, e.g., \citet{hall1992}. 
However, as I will show below, $\mathit{K}_n(\hat{\Omega}_n^*,T)$ still has the advantage that its square converges to a chi-square distribution if  $\Tau$ consists of only a single quantile.
%; $K_n(I_d,\Tau)$ is simpler to compute because no weight matrix has to be estimated and inverted.\marginpar{!}

The following algorithm describes how to conduct inference and how to test restrictions on the QR process uniformly over the entire set $\Tau$. This includes, for example, individual quantiles, finite sets of quantiles, closed intervals, and disjoint unions of closed intervals.
\begin{algorithm}[Wild gradient bootstrap inference]\label{al:binf}
\begin{compactenum}
\item\label{al:binf:1} Do Steps \ref{al:bse:1}-\ref{al:bse:3} of Algorithm \ref{al:bse}.

\item\label{al:binf:2} If $\Omega(\tau) = \hat{\Omega}_n^*(\tau)$, compute $\hat{V}_n^*(\tau,\tau)$ as the sample variance of the $m$ bootstrap observations of $\sqrt{n}(\hat{\beta}_n^*(\tau) - \hat{\beta}_n(\tau))$ from Step \ref{al:binf:1}.

\item For each of the $m$ bootstrap observations from Step \ref{al:binf:1}, calculate
\begin{equation}\label{eq:ksstarstat}
K^*_n(\Omega,T) := \sup_{\tau\in\Tau} \bigl\vert \Omega^{-1/2}(\tau)\sqrt{n}R(\tau)\bigl(\hat{\beta}^*_n(\tau)-\hat{\beta}_n(\tau)\bigr)\bigr\vert.
\end{equation}

\item Reject $\mathrm{H}_0$ in favor of $\mathrm{H}_1$ if $\mathit{K}_n(\Omega,\Tau)$ is larger than $q_{n,1-\alpha}(\Omega,\Tau)$, the $1-\alpha$ empirical quantile of the $m$ bootstrap statistics $\mathit{K}^*_n(\Omega,\Tau)$.
\end{compactenum}
\end{algorithm}

As before, I take the number of bootstrap simulations $m$ as large and view the bootstrap quantile $q=q_{n,1-\alpha}(\Omega,\Tau)$ directly as the minimizer of $$\ev^* \Bigl( \rho_{1-\alpha}\bigr(K^*_n(\Omega,T) - q\bigr) - \rho_{1-\alpha}\bigl(K^*_n(\Omega,T)\bigr)\Bigr).$$ Subtracting the second term here again ensures that this expression is necessarily finite without further conditions on the underlying variables.
%For fixed $q$, the expression inside the ($\prob^*$-outer) expectation in the preceding display into a bounded, Lipschitz continuous function of the bootstrap statistic $K^*_n(\hat{\Omega}_n^*,T)$. This works well in conjunction with the bounded Lipschitz metric introduced in Theorem \ref{th:bootclt}.

To prove consistency of Algorithm \ref{al:binf} for the Wald-type weight $\hat{\Omega}_n^*$, we also need to guarantee that $\hat{\Omega}_n^*$ is non-singular with probability approaching one as $n\to\infty$. This requires the eigenvalues of $\Sigma(\tau,\tau)$ in $V(\tau,\tau) = J^{-1}(\tau)\Sigma(\tau,\tau)J^{-1}(\tau)$ to be bounded away from zero, uniformly in $\tau\in\Tau$. In the absence of clusters, such a property would automatically follow from non-singularity of $n^{-1}\sum_{i=1}^n \ev X_{i1}X_{i1}^\top$. (Recall the discussion above Assumption \ref{as:gstat}.) In the cluster context, it is a separate restriction that rules out some scenarios where several clusters have similar forms of extreme within-cluster dependence.
\begin{assumption}\label{as:posdef}
For all non-zero $a\in\mathbb{R}^d$, $\inf_{\tau\in\Tau}a^\top\Sigma(\tau,\tau)a > 0$.
\end{assumption}

%[1,]    1    1    0    0    0
%[2,]    1    1    0    0    0
%[3,]    0    0    1    1    0
%[4,]    0    0    1    1    0
%[5,]    0    0    0    0    1
%
%[1,]    1    1    1
%[2,]    1   -1   -1
%[3,]    1    1    1
%[4,]    1   -1   -1
%[5,]    1    2    3

The next result shows that Algorithm \ref{al:binf} is indeed a consistent test of the null hypothesis $R(\tau)\beta(\tau)=r(\tau)$ for all $\tau\in\Tau$ against the alternative that $R(\tau)\beta(\tau)\neq r(\tau)$ for some $\tau\in\Tau$.% The theorem applies to both the Euclidean and the maximum norm versions of the statistic $\mathit{K}_n(\Omega,\Tau)$.
\begin{theorem}\label{c:bootse}
Suppose Assumptions \ref{as:data}-\ref{as:gstat} and \ref{as:posdef} hold. For $\alpha\in(0,1)$, we have
\begin{compactenum}[\upshape(i)]
\item under $\mathrm{H_0}$, $\prob(\mathit{K}_n(\hat{\Omega}_n^*,\Tau) > q_{n,1-\alpha}(\hat{\Omega}_n^*,\Tau))\to\alpha$ and
\item under $\mathrm{H_1}$, $\prob(\mathit{K}_n(\hat{\Omega}_n^*,\Tau) > q_{n,1-\alpha}(\hat{\Omega}_n^*,\Tau))\to 1$.
\end{compactenum}
Both results also hold without Assumption \ref{as:posdef} if $I_d$ is used instead of $\hat{\Omega}_n^*$ in all instances.
\end{theorem}

\begin{comment}
The theorem in fact remains valid if the Euclidean norm in the definition of $\mathit{K}_n(\Omega,\Tau)$ in \eqref{eq:ksstat} is replaced by any other norm on $\mathbb{R}^d$ as long as the same norm is also employed in the bootstrap statistic $\mathit{K}^*_n(\Omega,\Tau)$ in \eqref{eq:ksstarstat}. A natural choice other than the Euclidean norm is the maximum norm $|x|_{\max} = \max\{|x_1|,\dots,|x_d|\}$, i.e., the maximum absolute entry of a vector $x = (x_1,\dots,x_d)$. I will use this norm below to construct bootstrap confidence bands for the QR coefficient functions. 
\end{comment}

I now discuss three useful corollaries of Theorems \ref{th:bootse} and \ref{c:bootse} regarding (i)~chi-square inference with the bootstrap covariance matrix, (ii)~bootstrap confidence bands, and (iii)~computation of the supremum in the Kolmogorov-Smirnov statistics. First, if $\Tau$ consists of only a single quantile $\tau_0$, then the square of $K_n(\hat{\Omega}_n^*,\tau_0)$ is simply the ordinary Wald statistic 
$$n\bigl(R(\tau_0)\hat{\beta}_n(\tau_0)-r(\tau_0)\bigr)^\top \hat{\Omega}_n^{*-1}(\tau_0)\bigl(R(\tau_0)\hat{\beta}_n(\tau_0)-r(\tau_0)\bigr).$$ 
Because $\sqrt{n}(R(\tau_0)\hat{\beta}_n(\tau_0)-r(\tau_0))$ is asymptotically multivariate normal under the null hypothesis and $\hat{\Omega}_n^*(\tau_0)$ is consistent for the variance of that multivariate normal distribution, the statistic in the preceding display has an asymptotic chi-square distribution. Hence, chi-square critical values can be used instead of bootstrap critical values for the test decision. The following corollary makes this precise.
\begin{corollary}\label{c:bootpw}
Suppose we are in the situation of Theorem \ref{c:bootse} with $\Tau = \{\tau_0\}$ for some $\tau_0\in(0,1)$. Then
\begin{compactenum}[\upshape(i)]
\item under $\mathrm{H_0}$, $K_n(\hat{\Omega}_n^*,\tau_0)^2 \leadsto \chi^2_{\rank R(\tau_0)}$ and
\item under $\mathrm{H_1}$, $\prob(\mathit{K}_n(\hat{\Omega}_n^*,\tau_0)^2 > q)\to 1$ for every $q\in\mathbb{R}$.
\end{compactenum}
\end{corollary}
\begin{comments}
\begin{inparaenum}[(i)]
\item From this result it also follows immediately that a single QR coefficient at a single quantile can be studentized with its bootstrap standard error and compared to a standard normal critical value.

\phantomsection%
\label{rev:smallsample}%
\changes{\item The Monte Carlo study below suggests that asymptotic inference using the bootstrap covariance matrix generally performs well and is only slightly worse in terms of finite-sample size than bootstrapping both the covariance matrix and the critical values. Still, when there are very few clusters, asymptotic inference with bootstrap standard errors tends to over-reject while simultaneously having significantly lower power than the test with bootstrap critical values. The over-rejection could, in principle, be avoided by replacing standard normal and chi-square critical values with larger critical values from the Student $t_{n-1}$ and similarly scaled $F$ distributions \citep{donaldlang2007, besteretal2011}. However, such small-sample adjustments would decrease the power of the test even further. It is therefore recommended to bootstrap the critical values when only few clusters are available.} %As the Monte Carlo study shows, similar issues occur for certain bootstrap weight distributions when $\tau_0$ is near $0$ or $1$. The problem with this approach in the present context is that the power of tests based on standard normal and chi-square critical values is at its lowest when the number of clusters is small, the within-cluster correlation is high, or $\tau_0$ is near $0$ or $1$. Enlarging the critical values would make these issues worse. It is therefore recommended to also bootstrap the critical values in such circumstances.
\end{inparaenum}
\end{comments}

%Let $\tau \mapsto \mathbb{X}_h(\tau)$ be a standard $h$-dimensional Brownian bridge so that $\tau \mapsto \mathbb{Q}_h(\tau) := |\mathbb{X}_h(\tau)|/\sqrt{\tau(1-\tau)}$ is a standardized tied-down Bessel process of order $h$. Hence, for every fixed $\tau\in\Tau$, $\mathbb{Q}^2_h(\tau)$ has a chi-square distribution with $h$ degrees of freedom.

%\begin{corollary}\label{c:bootseinid}
%In the situation of Theorem \ref{th:bootse} and $c_i\equiv 1$, we have
%\begin{compactenum}[\upshape(i)]
%\item under $\mathrm{H_0}$, $\mathit{K}_n(\Tau) \leadsto \sup_{\tau\in\Tau} \mathbb{Q}^2_h(\tau)$ and
%\item under $\mathrm{H_1}$, $\lim_{n\to\infty}\prob(\mathit{K}_n(\Tau) > c)= 1$ for every $c\in\mathbb{R}$.
%\end{compactenum}
%\end{corollary}

Next, the results in Theorems \ref{th:bootse} and \ref{c:bootse} allow for the construction of bootstrap confidence bands (uniform in $\tau\in\Tau$) for the QR coefficient function. These bands can be computed jointly for the entire $d$-dimensional function or only a subset $\Delta\subset\{1,\dots,d\}$ of coefficients. As before, a positive definite weight matrix function, denoted here by $\tau\mapsto\Lambda(\tau)$, can be specified to improve the finite-sample performance. An obvious choice is $\Lambda(\tau) = \hat{V}_n^*(\tau,\tau)$. In the following algorithm and in the corollary immediately below, I write $a_j$ for the $j$th entry of a $d$-vector $a$ and $A_{jj}$ for the $j$th diagonal element of a $d\times d$ square matrix $A$.

\begin{algorithm}[Wild gradient bootstrap confidence bands]\label{al:bcbands}
\begin{compactenum}
\item Do Steps \ref{al:bse:1}-\ref{al:bse:3} of Algorithm \ref{al:bse} and, if $\Lambda(\tau) = \hat{V}_n^*(\tau,\tau)$, compute $\hat{V}_n^*(\tau,\tau)$ as in Step \ref{al:binf:2} of Algorithm \ref{al:binf}.

\item For each of the $m$ bootstrap observations, calculate
\begin{equation*}
K^*_n(\Lambda,\Tau,\Delta) := \sup_{\tau\in\Tau} \max_{j\in\Delta} \biggl\vert\frac{\hat{\beta}^*_n(\tau)_j-\hat{\beta}_n(\tau)_j}{\sqrt{\Lambda(\tau)_{jj}/n}}\biggr\vert
\end{equation*}
and $q_{n,1-\alpha}(\Lambda,\Tau,\Delta)$, the $1-\alpha$ empirical quantile of $K^*_n(\Lambda,\Tau,\Delta)$.

\item For each $\tau\in\Tau$ and $j\in\Delta$, compute the interval $$\Bigl[\hat{\beta}_n(\tau)_j \pm q_{n,1-\alpha}(\Lambda,\Tau,\Delta)\sqrt{\smash{\Lambda(\tau)_{jj}}/n} \Bigr].$$
\end{compactenum}
\end{algorithm}
The confidence band given in the last step of the algorithm has asymptotic coverage probability $1-\alpha$. The proof of this result is based on the fact that, as long as the maximum norm is used in \eqref{eq:ksstarstat} instead of the Euclidean norm, $K^*_n(\Lambda,\Tau,\Delta)$ is nothing but the bootstrap statistic $K^*_n(\Omega,\Tau)$ with a diagonal weight matrix and a matrix of restrictions $R(\tau)\equiv R$ that selects the coefficients given in $\Delta$. %For example, if $\Delta = \{1,3\}$ then the first and third diagonal element of $R$ would be equal to one and all other elements would be zero.
\begin{corollary}\label{c:bootband}
Suppose we are in the situation of Theorem \ref{c:bootse}. For every $\Delta\subset\{1,\dots,d\}$ 
$$ \prob\biggl(\beta(\tau)_j \in \biggl[\hat{\beta}_n(\tau)_j \pm q_{n,1-\alpha}(\hat{V}_n^*,\Tau,\Delta)\sqrt{\hat{V}_n^*(\tau,\tau)_{jj}/n} \biggr]~\text{for all $\tau\in\Tau$, all $j\in\Delta$}\biggr) $$ converges to $1-\alpha$ as $n\to\infty$. This continues to hold without Assumption \ref{as:posdef} if all instances of $\hat{V}_n^*$ are replaced by $I_d$.
\end{corollary}

Finally, if $\Tau$ is not a finite set, computing $K_n(\Omega,\Tau)$ and the confidence bands is generally infeasible in practice due to the presence of a supremum in their definitions. This can be circumvented by replacing the supremum with a maximum over a finite grid $\Tau_n\subset\Tau$ that becomes finer as the sample size increases. For example, if $\Tau$ is a closed interval, we can take $\Tau_n = \{ j/n : j=0,1,\dots,n \}\cap \Tau$. For any $\tau$ in the interior of $\Tau$ and $n$ large enough, we can then find $\tau_n,\tau_n'\in\Tau_n$ that differ by $1/n$ and satisfy $\tau_n\leq \tau < \tau_n'$. This gives $0\leq \tau - \tau_n < 1/n$. Furthermore, the endpoints of $\Tau_n$ are less than $1/n$ away from the respective endpoints of $\Tau$. Hence, every $\tau\in\Tau$ is the limit of a sequence $\tau_n\in\Tau_n$. This turns out to be the property needed to ensure that the approximation of $\Tau$ by a finite set has no influence on the asymptotic behavior of the bootstrap test.
\begin{corollary}\label{c:bootgrid}
Suppose we are in the situation of Theorem \ref{c:bootse} and there exist sets $\Tau_n\subset\Tau$ such that for every $\tau\in\Tau$ there is a sequence $\tau_n\in\Tau_n$ such that $\tau_n\to\tau$ as $n\to\infty$. Then Theorem \ref{c:bootse} and Corollary \ref{c:bootband} continue to hold when $\Tau_n$ is used instead of $\Tau$.
\end{corollary}

The next section illustrates the finite-sample behavior of the wild gradient bootstrap in a brief Monte Carlo exercise. Section \ref{s:star} then provides an application of the wild gradient bootstrap to Project STAR data.

\section{Monte Carlo Experiments}\label{s:mc}
This section presents several Monte Carlo experiments to investigate the small-sample properties of the wild gradient bootstrap in comparison to other methods of inference. I discuss significance tests at a single quantile (Experiment \ref{ex:signi}), inference about the QR coefficient function (Experiment \ref{ex:KS}), and confidence bands (Experiment \ref{ex:bands}).

%\pub{\addlines[2]}

The data generating process (DGP) for the following experiments is
\begin{align*}
Y_{ik} = 0.1 U_{ik} + X_{ik} + X_{ik}^2 U_{ik},
\end{align*}
where $X_{ik} = \sqrt{\varrho}Z_{i} + \sqrt{1-\varrho}\varepsilon_{ik}$ with $\varrho \in [0, 1)$;  $Z_i$ and $\varepsilon_{ik}$ are standard normal, independent of each other, and independent across their indices. This guarantees that the $X_{ik}$ are standard normal and, within each cluster, any two observations $X_{ik}$ and $X_{il}$ have a correlation coefficient of $\varrho$.  The $U_{ik}$ are distributed as $\mathrm{N}(0,1/3)$ and drawn independently of $X_{ik}$ to ensure that the $X_{ik}^2 U_{ik}$ have mean zero and variance one. The correlation structure of $U_{ik}$ is chosen such that the within-cluster correlation coefficient of $X_{ik}^2 U_{ik}$ is also approximately $\varrho$.\footnote{By construction, the correlation coefficient of $X_{ik}^2 U_{ik}$ and $X_{il}^2 U_{il}$ is $\corr(U_{ik}, U_{il})(2\varrho^2+1)/3$. I generate data such that $\corr(U_{ik}, U_{il}) = \min\{1, 3\varrho/(2\varrho^2 + 1)\}$. The within-cluster correlation coefficient of $X_{ik}^2 U_{ik}$ is then exactly $\varrho$ for $\varrho \in [0, 0.5]$ and has a value slightly below $\varrho$ for $\varrho\in(0.5,1)$. This choice for $\corr(U_{ik}, U_{il})$ ensures that the other restrictions on the DGP hold for all values of $\varrho$ used in the experiments.} Both $X_{ik}$ and $U_{ik}$ are independent across clusters. 

The DGP in the preceding display corresponds to the quadratic QR model
\begin{align}\label{eq:mc}
Q_{ik}(\tau\mid X_{ik}) = \beta_0(\tau) + \beta_1(\tau) X_{ik} + \beta_2(\tau)X_{ik}^2
\end{align}
with $\beta_0(\tau) = \Phi^{-1}(\tau)/\sqrt{300}$, $\beta_1(\tau) \equiv 1$, and $\beta_2(\tau) = \Phi^{-1}(\tau)/\sqrt{3}$, where $\Phi$ is the standard normal distribution function. I denote the QR estimates of the two slope parameters $\beta_1(\tau)$ and $\beta_2(\tau)$ by $\hat{\beta}_{1,n}(\tau)$ and $\hat{\beta}_{2,n}(\tau)$. Their bootstrap versions are $\hat{\beta}^*_{1,n}(\tau)$ and $\hat{\beta}^*_{2,n}(\tau)$. As before, I refer to the square roots of the diagonal elements of $\smash{\hat{V}_n^*}(\tau, \tau)/n$ as bootstrap standard errors and, for simplicity, now denote the bootstrap standard error of $\hat{\beta}_{2,n}(\tau)$ by $\se^*(\hat{\beta}^*_{2,n}(\tau))$.

In the following experiments, I consider inference about $\tau\mapsto \beta_1(\tau)$ and $\tau\mapsto \beta_2(\tau)$ for different values of the number of clusters $n$, the within-cluster correlation $\varrho$, and the variance of the cluster size $\var(c_i)$. In all experiments below, the smallest possible cluster size is $5$ and $c_i$ is distributed uniformly on $\{5, 6, \dots, c_{\max}\}$. Unless otherwise noted, the bootstrap weights are drawn from the Mammen distribution as defined in the remarks below Algorithm \ref{al:bse}.
\begin{experiment}[Significance tests at a single quantile]\label{ex:signi}
This Monte Carlo experiment illustrates the small-sample size and power of different methods for testing whether a single QR coefficient equals zero at a given quantile. To test the correct null hypothesis $\beta_2(.5) = 0$ in \eqref{eq:mc} against the alternative $\beta_2(.5) \neq 0$, I consider (i) wild gradient bootstrap inference as in Algorithm \ref{al:binf}, (ii) standard inference with bootstrap standard errors as in Corollary \ref{c:bootpw}, (iii) cluster-robust inference based on analytically estimating the standard errors, (iv) standard inference without cluster correction, (v) cluster-robust Rao score inference, and (vi) wild bootstrap inference without cluster correction.

For (i), note that $R\equiv (0, 0, 1)$ and $r\equiv 0$. Hence, Algorithm \ref{al:binf} is equivalent to testing whether $|\hat{\beta}_{2,n}(.5)|$ exceeds the empirical $1-\alpha$ quantile of the $m$ observations of $|\hat{\beta}^*_{2,n}(.5)-\smash{\hat{\beta}_{2,n}}(.5)|$ conditional on $\smash{\hat{\beta}_{2,n}}(.5)$. No weight matrix is needed because the test decision is independent of $\Omega(\tau)$ whenever $R(\tau)V(\tau,\tau)R(\tau)^\top$ is a scalar. Similarly, for (ii), the test decision in Corollary \ref{c:bootpw} is equivalent to simply comparing $|\hat{\beta}_{2,n}(.5)|/\se^*(\hat{\beta}^*_{2,n}(.5))$ to $\Phi^{-1}(1-\alpha/2)$. For (iii), I obtain standard errors by estimating $V(\tau,\tau) = J^{-1}(\tau)\Sigma(\tau,\tau)J^{-1}(\tau)$ analytically as suggested by \citet{parentesantossilva2013}. They propose the plug-in estimate 
%\begin{equation}
$$%\tilde{\Sigma}_n(\tau, \tau) = 
\frac{1}{n} \sum_{i=1}^n \sum_{k=1}^{c_i}\sum_{l=1}^{c_i}\psi_\tau\bigl(Y_{ik}-X_{ik}^\top\hat{\beta}_n(\tau)\bigr) \psi_{\tau}\bigl(Y_{il}-X_{il}^\top\hat{\beta}_n(\tau)\bigr)X_{ik}X_{il}^\top$$
%\end{equation}
for $\Sigma(\tau,\tau)$ and replace $J(\tau)$ by a \citet{powell1986} kernel estimate. The kernel estimate requires a bandwidth choice. The results here are based on the standard implementation in the \texttt{quantreg} package in \texttt{R} with the Hall-Sheather rule; see \citet[pp.\ 80-81]{koenker2005} and \citet{koenker2013}.\footnote{This bandwidth choice required a robust estimate of scale. \citet{koenker2013} uses the minimum of the standard deviation of the QR residuals and their normalized interquartile range. \citet{parentesantossilva2013} suggest the median absolute deviation of the QR residuals with a scaling constant of $1$. I chose \citeauthor{koenker2013}'s implementation because it yielded better results in nearly all cases.} For (iv), I use the regular version of the \citeauthor{powell1986} sandwich estimator described in \citet{koenker2005}. It employs the same kernel estimate of $J(\tau)$ as in (iii), but replaces $\Sigma(\tau,\tau)$ by $n^{-1} \tau(1-\tau) \sum_{i=1}^n \sum_{k=1}^{c_i}X_{ik}X_{ik}^\top$ and is therefore not designed to account for within-cluster correlation. For (v), I apply the $\mathrm{QRS}_0$ test of \citet{wanghe2007}, a cluster-robust version of the QR rank score test \citep{gutenbrunneretal1993}. \citeauthor{wanghe2007} derive their test statistic under homoscedasticity assumptions; the DGP considered here is highly heteroscedastic. For (vi), I compute critical values from the \texttt{quantreg} implementation of the \citeauthor{fengetal2010} (\citeyear{fengetal2010}, FHH hereafter) wild bootstrap for QR models. Their method perturbs the QR residuals via a carefully chosen weight distribution but presumes independent observations. An alternative wild bootstrap procedure due to \citet{davidson2012} had size properties similar to those of the FHH method but had lower power in nearly all of my experiments; results for this bootstrap are therefore omitted.

\begin{figure}[t]
\centering
\includegraphics[width=0.9\textwidth]{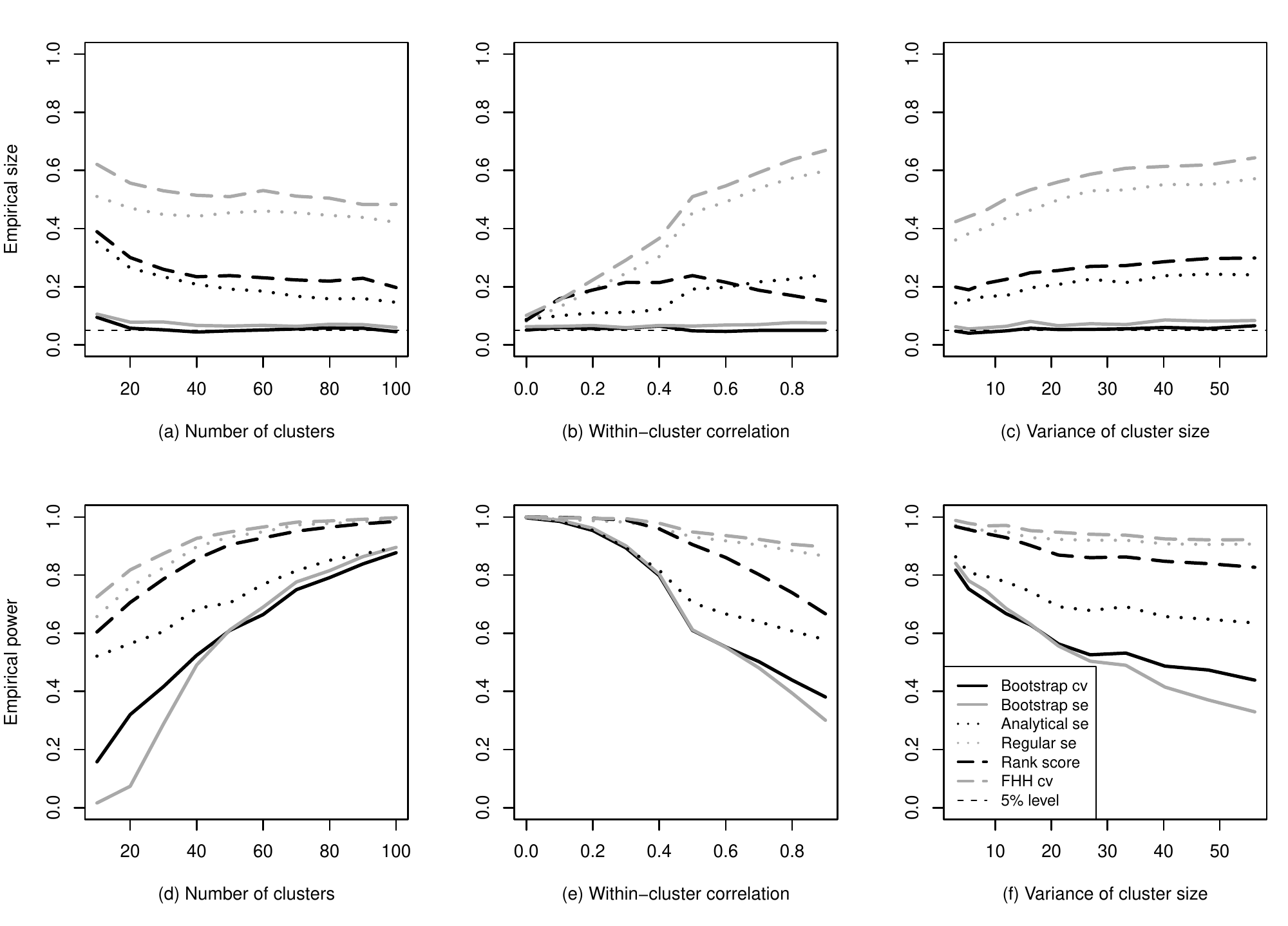}
\caption{Empirical rejection frequencies of a correct hypothesis $\mathrm{H}_0\colon \beta_2(.5) = 0$ (panels (a)-(c)) and the incorrect hypothesis $\mathrm{H}_0\colon \beta_2(.75) = 0$ (panels (d)-(f)) using wild gradient bootstrap critical values (solid black lines), bootstrap standard errors (solid grey), analytical cluster-robust standard errors (dotted black), regular standard errors without cluster correction (dotted grey), cluster-robust rank score inference (long-dashed black), and FHH wild bootstrap without cluster correction (long-dashed grey) at the 5\% level (short-dashed) as a function of the (a) number of clusters, (b) within-cluster correlation, and (c) maximal cluster size.} \label{fig:mc_ex11}
\end{figure}

Panels (a)-(c) in Figure \ref{fig:mc_ex11} show empirical rejection frequencies of a correct hypothesis $\mathrm{H}_0\colon \beta_2(.5) = 0$ for methods (i)-(vi) at the 5\% level (short-dashed line) as a function of (a) the number of clusters $n$, (b) the within-cluster correlation $\varrho$, and (c) the variance of the cluster size $\var(c_i)$. Each horizontal coordinate was computed from 2,000 simulations and all six methods were faced with the same data. The three bootstrap tests used $m=299$ bootstrap repetitions. The wild gradient bootstrap had Mammen weights. Results for other weight distributions are discussed below. 

For panel (a), I set $\varrho = .5$, $\var(c_i) = 10$ (i.e., $c_{\max} = 15$), and considered $n\in\{10, 20, \dots, 100\}$. As can be seen, the wild gradient bootstrap critical values (solid black lines) and bootstrap standard errors (solid grey) provided tests that were essentially at the nominal level with as few as 20 clusters, with the bootstrap critical values performing slightly better. Tests based on analytical cluster-robust standard errors (dotted black) and cluster-robust rank scores (long-dashed black) over-rejected significantly, although this property became less pronounced for larger numbers of clusters. Regular standard errors without cluster correction (dotted grey) and wild bootstrap without cluster correction (long-dashed grey) led to severe over-rejection in all cases. For (b), I chose $n = 50$, $\var(c_i) = 10$, and varied $\varrho\in\{0, .1, \dots, .9\}$. At $\varrho = 0$, all tests except the rank score test apply and had almost correct size. For larger values of the within-cluster correlation, the three analytical tests and the FHH bootstrap test over-rejected considerably, although the rank score test improved for very high correlations. The test based on $\se^*(\hat{\beta}^*_{2,n}(\tau))$ over-rejected mildly. The bootstrap test given in Algorithm \ref{al:binf} was nearly at the nominal level in most cases. For (c), I fixed $\varrho = .5$ and changed $c_{\max}\in\{9, 11, \dots, 29\}$ so that $\var(c_i)$ increased from $2$ to $52$ over this range. I simultaneously decreased $n$ in order to keep the average total number of observations constant at approximately 250; this resulted in numbers of clusters between 36 and 15. The test based on the bootstrap standard error again over-rejected slightly but far less than the ones based on the analytical cluster-robust standard error and the cluster-robust rank score. Wild gradient bootstrap critical values again provided a test with almost proper size, while regular standard errors and the wild bootstrap for independent observations were not useful at any value of $\var(c_i)$.

Panels (d)-(f) show empirical rejection frequencies of the incorrect hypothesis $\mathrm{H}_0\colon \beta_2(.75) = 0$ for the same data.  Rejection frequencies of the three analytical methods and the FHH wild bootstrap are only reported for completeness and, because of their size distortion, should not be interpreted as estimates of their power. The wild gradient bootstrap critical values tended to lead to a more powerful test than inference with bootstrap standard errors. This was, in particular, the case in small samples, at high within-cluster correlations, and for large variances of the cluster size. The rejection frequencies of all tests were increasing in the number of clusters, decreasing in the within-cluster correlations, and decreasing in the variance of the cluster size.

\phantomsection%
\label{rev:differentrho}%
Following \citet{mackinnonwebb2014}, I also experimented (not shown) with cases where I varied the within-cluster correlation of $X$ and $U$ in the Monte Carlo DGP independently. For the wild gradient bootstrap, I found that for any within-cluster correlation of $X$, the degree of correlation in $U$ had little impact, whereas increases in the within-cluster correlation in $X$ led to mild size distortions similar to the ones found in Figure \ref{fig:mc_ex11}. In contrast, increases in the within-cluster correlation of $U$ led to severe over-rejection in tests based on analytical cluster-robust standard errors; higher correlation in $X$ also induced over-rejection, but the impact was considerably less pronounced.

In light of the findings so far it should be noted that the small-sample results for the analytically estimated standard errors reported here do not contradict the ones reported by \citet{parentesantossilva2013}, who find a much better performance of their method in terms of finite-sample size. In comparison to their experiments, I consider data with smaller numbers of clusters, different correlation structures, and much stronger cluster heterogeneity in terms of cluster sizes. Computing the standard errors analytically worked well when the number of clusters was large, the within-cluster correlation was low, and the clusters were small. 
\phantomsection%
\label{rev:wanghedging}%
Similarly, the rank score test of \citet{wanghe2007} is designed for homoscedastic models and performed much better in such settings. For heteroscedastic models, \citet{wang2009} shows that reweighting their test statistic can significantly improve inference when more is known about the specific form of heteroscedasticity; her reweighting schemes do not apply to the DGP in the present example and are therefore not discussed.

\begin{table}[thp]
\caption{Empirical size and power as in Figure \ref{fig:mc_ex11} for different bootstrap weights}\label{tab:weights}
\centering
{\small
\resizebox{\columnwidth}{!}{%
\begin{tabular}{cp{0cm}ccp{0cm}ccp{0cm}ccp{0cm}ccp{0cm}ccp{0cm}cc}%
\hline
&
& \multicolumn{8}{c}{$\mathrm{H}_0\colon \beta_2(.5) = 0$ (size)}
&
& \multicolumn{8}{c}{$\mathrm{H}_0\colon \beta_2(.75) = 0$ (power)}
\\
\cline{3-10}\cline{12-19}
& 
& \multicolumn{2}{c}{Mammen}
&
& \multicolumn{2}{c}{Rademacher}
&
& \multicolumn{2}{c}{Webb}
& 
& \multicolumn{2}{c}{Mammen}
&
& \multicolumn{2}{c}{Rademacher}
&
& \multicolumn{2}{c}{Webb}
\\ 
\cline{3-4}\cline{6-7}\cline{9-10}\cline{12-13}\cline{15-16}\cline{18-19}
$n$
&
& cv & se 
&
& cv & se 
&
& cv & se 
&
& cv & se 
&
& cv & se
&
& cv & se \\ 
\hline
 10 & & .098 & .114 & & .131 & .166 & & .128 & .146 & & .155 & .019 & & .104 & .016 & & .094 & .011 \\ 
 20 & & .068 & .088 & & .076 & .102 & & .071 & .098 & & .328 & .086 & & .302 & .036 & & .270 & .026 \\ 
100 & & .054 & .068 & & .059 & .070 & & .054 & .070 & & .876 & .896 & & .864 & .886 & & .868 & .890 \\ 
\rule{0pt}{3ex}$\varrho$ \\
\hline
.1  & & .061 & .069 & & .063 & .071 & & .062 & .068 & & .998 & 1    & & .998 & 1    & & .999 & 1 \\ 
.5  & & .055 & .071 & & .059 & .074 & & .061 & .076 & & .602 & .613 & & .590 & .544 & & .589 & .502 \\ 
.9  & & .057 & .078 & & .067 & .088 & & .065 & .091 & & .378 & .308 & & .376 & .183 & & .371 & .156 \\ 
\rule{0pt}{3ex}$\var(c_i)$ \\
\hline
 2  & & .056 & .070 & & .057 & .072 & & .056 & .072 & & .820 & .840 & & .808 & .830 & & .803 & .831 \\ 
24  & & .054 & .076 & & .062 & .078 & & .060 & .074 & & .580 & .578 & & .578 & .479 & & .570 & .446 \\ 
52  & & .056 & .082 & & .062 & .085 & & .066 & .086 & & .456 & .364 & & .459 & .235 & & .446 & .183 \\ 
\hline
\end{tabular}%
}}
\end{table}

%As a referee points out, the choice of bootstrap weight distribution is important in many applications. This is also true in the present case. 
Bootstrap weight distributions other than the Mammen distribution are often found to work well in regression settings. 
\phantomsection%
\label{rev:bootweights}%
These include the standard normal distribution, the recentered $\mathrm{Exponential}(1)$ distribution, the Rademacher distribution, which takes on the values $-1$ and $1$ with equal probability, and the \citet{webb2013} 6-point distribution, which takes on $-\sqrt{1.5}, -1, -\sqrt{0.5}, \sqrt{0.5}, 1,$ and $\sqrt{1.5}$ with equal probability. In my experiments, the standard normal had size properties very similar to those of the Rademacher and Webb distributions, but lower power. I therefore do not present detailed results for this distribution. The same holds for the recentered $\mathrm{Exponential}(1)$, which behaved almost like the Mammen distribution in terms of size, but also had lower power. Comparisons of the other distributions are shown in Table \ref{tab:weights}. The experimental setup and data were the same as in Figure \ref{fig:mc_ex11}. The left-hand side of the table measures finite-sample size for different numbers of clusters, within-cluster correlations, and variances of the cluster size as in panels (a), (b), and (c) of Figure \ref{fig:mc_ex11}; the right-hand side corresponds to the power estimates in panels (d)-(f). As can be seen, the Mammen distribution had slightly better size and power, in particular when the number of clusters was small, the within-cluster correlation was high, and the variance of the cluster size was large.

\begin{figure}[t]
\centering
\includegraphics[width=.55\textwidth]{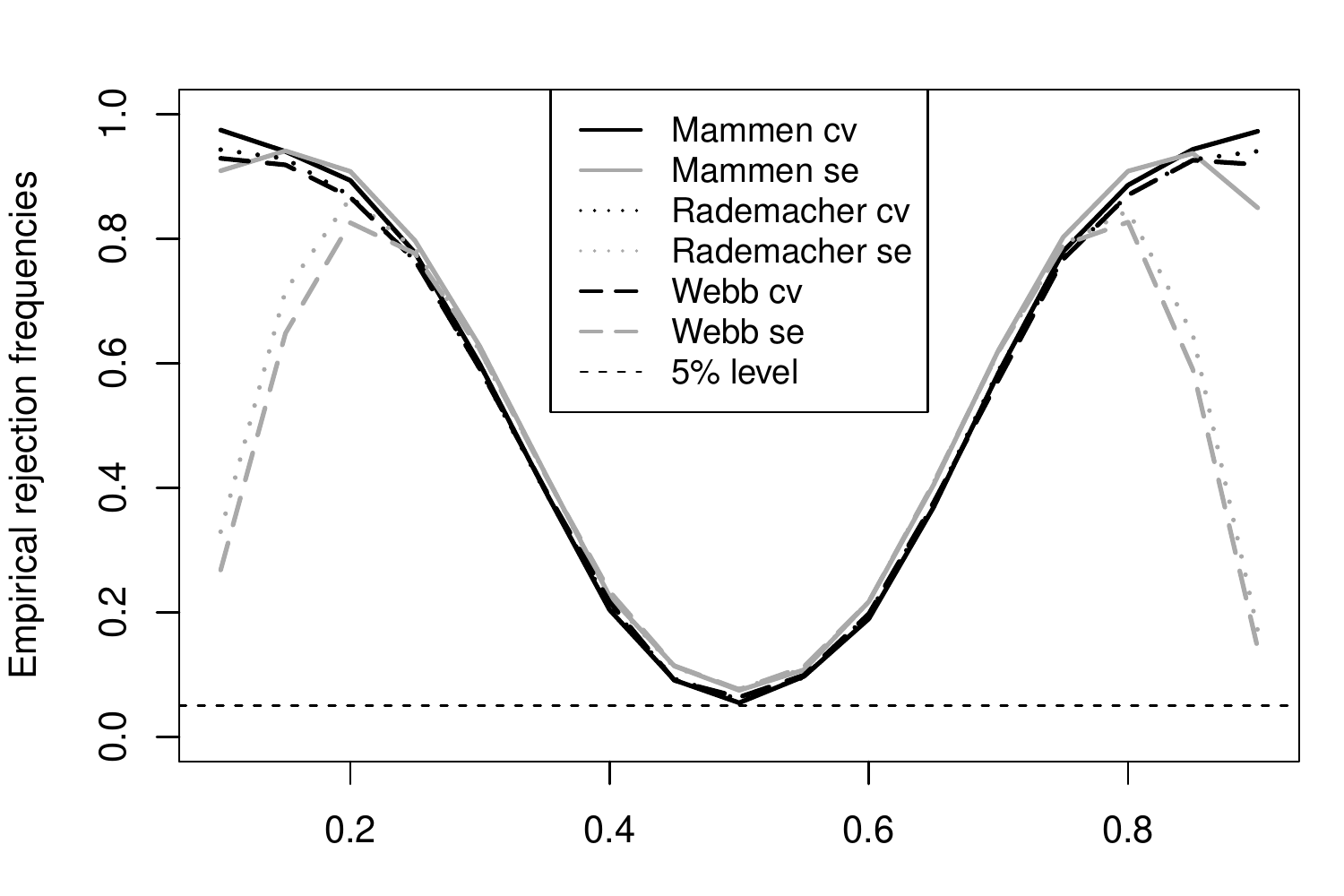}
\caption{Rejection frequencies of $\mathrm{H}_0\colon \beta_2(\tau) = 0$ for different values of $\tau$ using the same methods as in Table \ref{tab:weights}. $\mathrm{H}_0$ is only true at $\tau = .5$.}
\label{fig:mc_ex12}
\end{figure}

%\pub{\addlines}

To further investigate finite-sample power of the three bootstrap weight distributions, I plot in Figure~\ref{fig:mc_ex12} their rejection frequencies of $\mathrm{H}_0\colon\beta_2(\tau) = 0$ at 17 separate quantile indices $\tau\in\{.1, .15, \dots, .9\}$ for $n=75$. I again chose $\varrho=.5$, $\var(c_i) = 10$, $m=299$, and 2,000 Monte Carlo simulations. For this experiment all 17 possible null hypotheses were tested with the same data. Only $\mathrm{H}_0\colon\beta_2(.5)=0$ is true. As the plot shows, wild gradient bootstrap critical values led to tests with good size and power at all quantiles and for all weight distribution. Size and power of the tests based on bootstrap standard errors were similar for $\tau\in[.2,.8]$. However, for quantile indices outside this interval the tests with bootstrap standard errors from the Rademacher (dotted grey) and Webb (long-dashed grey) distributions showed a sharp decline in power; the Mammen distribution (solid grey) did not have this issue. I experimented with the parameters of the DGP and found that the power of the Rademacher and Webb distributions for small and large $\tau$ increased quickly when I increased the number of clusters or decreased the within-cluster correlation. For example, for the Rademacher distribution the rejection frequency at $\tau=.1$ and $.9$ was about 80\% when I either set $\varrho$ to $.3$ or $n$ to $100$. 

\phantomsection%
\label{rev:skewdiscussion}%
The reason for the large differences in finite-sample power between the distributions appears to be the extreme skewness of the distribution of the summands in the gradient process for small and large $\tau$. The asymmetry in the Mammen distribution seems to mimic this property particularly well. The standard errors also improved when I used a recentered $\mathrm{Exponential}(1)$ or other asymmetric distributions, but the Mammen distribution provided the best results. \hfill$\square$
\end{experiment}

\begin{experiment}[Uniform inference on the QR process]\label{ex:KS}
This experiment illustrates the finite-sample performance of Algorithm \ref{al:binf} for inference on the entire QR process. I tested the true hypothesis $\mathrm{H}_0\colon \beta_1(\tau) = 1$ for all $\tau\in\Tau$ and the false hypothesis $\mathrm{H}_0\colon \beta_1(\tau)= 0$ for all $\tau\in\Tau$ at the 5\% level. I chose $\var(c_i)=10$, $m=199$ bootstrap simulations with the Mammen distribution, and, in view of Corollary \ref{c:bootgrid}, I approximated $\Tau = [.1, .9]$ by $\{.1, .2, \dots, .9\}$. The test statistics  $K_n(\Omega, \Tau)$ were either (i)~weighted by the bootstrap estimate $\hat{\Omega}_n^*$, (ii)~weighted by the analytical estimate of $\tau\mapsto R(\tau)V(\tau,\tau)R(\tau)^\top$ described in the preceding Monte Carlo exercise, or (iii)~unweighted ($\Omega = I$). All three methods were faced with the same data. %The performance of the algorithm was similar to that of the bootstrap critical value test described in Experiment \ref{ex:signi}. I therefore went with two fairly extreme scenarios in which $n=25$ and $c_{\max} \in \{ 15, 50 \}$ to explore the limits of the wild gradient bootstrap. In particular for $c_{\max} = 50$, the data is far from the underlying assumptions and the number of clusters is much smaller than the sample size.

\begin{table}[thp]
\caption{Empirical size of Algorithm \ref{al:binf} at the 5\% level}\label{tab:ks}
\centering
{\small
\begin{tabular}{cp{0cm}cccp{0cm}cccp{0cm}ccc}%
\hline
& 
& \multicolumn{3}{c}{$n=20$}
&
& \multicolumn{3}{c}{$n=30$}
& 
& \multicolumn{3}{c}{$n=50$}
\\ 
\cline{3-5}\cline{7-9}\cline{11-13}
$\varrho$
&
& boot. & ana. & unw. 
&
& boot. & ana. & unw. 
&
& boot. & ana. & unw.  
\\ 
\hline
 0 & & .029 & .026 & .012 & & .034 & .014 & .011 & & .038 & .026 & .012 \\ 
.1 & & .030 & .022 & .004 & & .038 & .029 & .013 & & .029 & .029 & .011 \\ 
.2 & & .020 & .016 & .006 & & .025 & .026 & .010 & & .035 & .034 & .015 \\ 
.3 & & .020 & .017 & .002 & & .032 & .028 & .009 & & .042 & .039 & .017 \\ 
.4 & & .010 & .006 & .001 & & .030 & .026 & .009 & & .048 & .039 & .011 \\ 
.5 & & .008 & .001 & .000 & & .016 & .011 & .001 & & .036 & .025 & .006 \\ 
.6 & & .005 & .000 & .000 & & .015 & .006 & .000 & & .041 & .029 & .004 \\ 
.7 & & .011 & .001 & .000 & & .015 & .010 & .000 & & .039 & .030 & .001 \\ 
.8 & & .012 & .000 & .000 & & .008 & .005 & .000 & & .037 & .027 & .003 \\ 
.9 & & .012 & .000 & .000 & & .006 & .002 & .000 & & .032 & .020 & .002 \\  
\hline
\end{tabular}%
}
\end{table}

%\pub{\addlines}

Table \ref{tab:ks} reports the empirical rejection frequencies of the true null hypothesis for methods (i)-(iii) from an experiment with 1,000 Monte Carlo simulations for each $\varrho\in\{0, .1, \dots, .9\}$ and each $n\in\{20, 30, 50\}$. At $n = 20$, the test based on the bootstrapped Wald weight (``boot.'')\ was quite conservative for all degrees of within-cluster correlation but not overly so for $\varrho$ smaller than $.4$. The performance of the bootstrap test with analytical weights (``ana.'')\ was slightly worse, especially for higher within-cluster correlations. However, both of these tests under-rejected considerably less for larger $n$ so that at $n=50$ the size of bootstrap-weighted test was above .3 for all but one $\varrho$. In contrast, the unweighted version was very conservative for all within-cluster correlations and all numbers of clusters.

\begin{table}[thp]
\caption{Empirical power of Algorithm \ref{al:binf} at the 5\% level}\label{tab:kspow}
\centering
{\small
\begin{tabular}{cp{0cm}cccp{0cm}cccp{0cm}ccc}%
\hline
& 
& \multicolumn{3}{c}{$n=10$}
&
& \multicolumn{3}{c}{$n=15$}
& 
& \multicolumn{3}{c}{$n=20$}
\\ 
\cline{3-5}\cline{7-9}\cline{11-13}
$\varrho$
&
& boot. & ana. & unw. 
&
& boot. & ana. & unw. 
&
& boot. & ana. & unw.  
\\ 
\hline
 0 & & 1    & .735 & .692 & & 1    & .986 & .983 & & 1    & 1     & 1     \\ 
.1 & & .994 & .431 & .373 & & 1    & .924 & .904 & & 1    & .994 & .991 \\ 
.2 & & .944 & .173 & .133 & & .993 & .704 & .665 & & 1    & .962 & .945 \\ 
.3 & & .828 & .043 & .027 & & .964 & .352 & .320 & & .999 & .787 & .747 \\ 
.4 & & .665 & .013 & .001 & & .849 & .085 & .062 & & .963 & .417 & .362 \\ 
.5 & & .438 & .007 & .000 & & .594 & .005 & .001 & & .790 & .063 & .041 \\ 
.6 & & .412 & .003 & .000 & & .531 & .004 & .000 & & .703 & .034 & .016 \\ 
.7 & & .409 & .006 & .000 & & .474 & .001 & .000 & & .624 & .024 & .008 \\ 
.8 & & .388 & .006 & .000 & & .423 & .003 & .001 & & .527 & .014 & .002 \\ 
.9 & & .380 & .006 & .001 & & .338 & .006 & .001 & & .428 & .014 & .001 \\  
\hline
\end{tabular}%
}
\end{table}

%The conservative behavior did not lead to low power for the bootstrap-weighted Wald test. 
Table \ref{tab:kspow} shows empirical rejection frequencies of the false null hypothesis $\mathrm{H}_0\colon \beta_1(\tau)= 0$ for all $\tau\in\Tau$ at $n\in\{10, 15, 20\}$ in the same experimental setup as above.  For $n=10$, the Wald test with bootstrap weights had substantial power even for high within-cluster correlations. In sharp contrast, the unweighted and analytically weighted tests rejected considerably fewer false null hypotheses and exhibited a total loss of power starting from about $\varrho = .5$. Increases in the number of clusters translated into significant gains in the power of all tests, but the test based on the bootstrap weight matrix far outperformed the other two tests at all sample sizes. %All tests rejected increasingly fewer false null hypotheses for higher within-cluster correlations.
%Both tests had good power for small values of $\varrho$, but rejected increasingly fewer false null hypotheses for higher within-cluster correlations. The unweighted test exhibited a total loss of power starting from about $\varrho = .5$. At $c_{\max}=50$, the increased sample size translated into significant increases in the power of the two bootstrap tests. However, the Wald test now over-rejected quite severely for low within-cluster correlations but still performed well when $\varrho$ was larger than $.3$. The unweighted test consistently under-rejected, but this behavior was less pronounced.
%
%I also experimented with the analytical estimate of $R(\tau)V(\tau,\tau)R(\tau)^\top$ described in the preceding Monte Carlo exercise (not shown in the table). The test was conducted as in Algorithm \ref{al:binf}, but only the QR coefficient functions were bootstrapped. For $c_{\max}=15$, the size varied between $.068$ (at $\varrho = .1$) and $.002$ ($\varrho = .8$); power was between $.970$ ($\varrho = 0$) and $.003$ ($\varrho = .5$). For $c_{\max}=50$, the size ranged from $.102$ ($\varrho = 0$) to $.002$ ($\varrho = .5$); power was between $1$ ($\varrho = 0$) and $.002$ ($\varrho = .5$).
\hfill $\square$
\end{experiment}

\begin{experiment}[Confidence bands]\label{ex:bands}
%The performance of the algorithm was similar to that of the bootstrap critical value test described in Experiment \ref{ex:signi}. I therefore went with two fairly extreme scenarios in which $n=25$ and $c_{\max} \in \{ 15, 50 \}$ to explore the limits of the wild gradient bootstrap. In particular for $c_{\max} = 50$, the data is far from the underlying assumptions and the number of clusters is much smaller than the sample size.
In this experiment I investigate the finite-sample properties of Algorithm \ref{al:bcbands}. The setup is as in the preceding experiment. The empirical coverage of $\tau\mapsto \beta_1(\tau)$ with a 95\% wild gradient bootstrap confidence band is, by construction, identical to $1$ minus the empirical size of the bootstrap test in Table \ref{tab:ks} and therefore not shown here. I instead consider a more complex scenario where I report the empirical coverage of a joint 95\% confidence band for the two slope functions $\tau\mapsto \beta_1(\tau)$ and $\tau\mapsto \beta_2(\tau)$ for $n \in \{ 10, 15 , 20\}$  and $\varrho\in\{0, .1, \dots, .9\}$. %This is a more challenging scenario than Experiment \ref{ex:KS} because the function of interest $\tau\mapsto (\beta_1, \beta_2)(\tau)$ varies considerably with the quantile index $\tau$. 
Table \ref{tab:cb} contains the results.

\begin{table}[htp]
\caption{Empirical coverage of $\tau\mapsto (\beta_1, \beta_2)(\tau)$ by 95\% confidence band} \label{tab:cb}
\centering
{\small
\begin{tabular}{cp{0cm}cccp{0cm}cccp{0cm}ccc}%
\hline
& 
& \multicolumn{3}{c}{$n=10$}
&
& \multicolumn{3}{c}{$n=15$}
& 
& \multicolumn{3}{c}{$n=20$}
\\ 
\cline{3-5}\cline{7-9}\cline{11-13}
$\varrho$
&
& boot. & ana. & unw. 
&
& boot. & ana. & unw. 
&
& boot. & ana. & unw.  
\\ 
\hline
 0 & & .939 & .953 & .993 & & .938 & .949 & .993 & & .945 & .958 & .995 \\ 
.1 & & .949 & .977 & .997 & & .923 & .938 & .992 & & .938 & .942 & .992 \\ 
.2 & & .955 & .981 & .991 & & .920 & .946 & .988 & & .930 & .934 & .990 \\ 
.3 & & .960 & .994 & .997 & & .945 & .975 & .993 & & .934 & .947 & .994 \\ 
.4 & & .954 & .998 & 1    & & .966 & .993 & 1    & & .948 & .972 & .997 \\ 
.5 & & .962 & 1    & 1    & & .981 & .999 & 1    & & .974 & .998 & 1    \\ 
.6 & & .950 & .999 & 1    & & .958 & 1    & 1    & & .972 & .998 & 1    \\ 
.7 & & .937 & .999 & 1    & & .960 & .998 & 1    & & .969 & 1    & 1    \\ 
.8 & & .919 & .999 & 1    & & .949 & .997 & .999 & & .965 & .998 & 1    \\ 
.9 & & .900 & .998 & 1    & & .941 & .997 & .999 & & .960 & .996 & 1    \\
\hline
\end{tabular}%
}
\end{table}
As before, the procedure based on the bootstrapped Wald weight showed the most balanced performance with confidence bands that were close to 95\% in most cases. The only exceptions occurred at $n=10$ for very high within-cluster correlations, where the confidence bands were too thin. The unweighted confidence bands were consistently too wide. For analytical weights, the empirical coverage was near 95\% for small $\varrho$. However, at values of $\varrho$ larger than $.4$ the coverage was essentially 100\% even for $n=20$. Further increases in $n$ (not shown) yielded improvements for all versions of the confidence band but even the bootstrap-weighted confidence band needed a large number of clusters for the coverage to be fully balanced across $\varrho$. \hfill $\square$
\end{experiment}

In summary, the wild gradient bootstrap performs well even in fairly extreme (but empirically relevant) situations where the number of clusters is small, the within-cluster correlation is high, and the clusters are very heterogeneous. Here, reweighting the test statistic by the bootstrap covariance matrix is crucial for tests to have good size and power in finite samples. Analytical weights or no weights can be used when the number of clusters is large; otherwise they tend to lead to tests that are less reliable than those based on the bootstrapped Wald weight. For inference at a single quantile, testing with bootstrap standard errors and normal/chi-square critical values provides a simpler alternative to testing with bootstrap critical values that is, with some exceptions, nearly as good. These findings are also confirmed by an additional experiment in the next section, where I implement placebo interventions in the Project STAR data. 

\section{Application: Project STAR}\label{s:star}
This section applies the wild gradient bootstrap to investigate the effects of a class size reduction experiment on the conditional quantiles of student performance on a standardized test. The data come from the first year of the Tennessee \emph{Student/Teacher Achievement Ratio} experiment, known as Project STAR.

I start by briefly describing Project STAR; the discussion closely follows \citetalias{wordetal1990} and \citet{graham2008}, where more details can be found. At the beginning of the 1985-1986 school year, incoming kindergarten students who enrolled in one of the 79 project schools in Tennessee were randomly assigned to one of three class types within their school: a small class (13-17 students), a regular-size class (22-25 students), or a regular-size class (22-25 students) with a full-time teacher's aide. Teachers were then randomly assigned to one of these class types. Each of the project schools was required to have at least one of each kindergarten class type. During the 1985-1986 school year, a total of 6,325 students in 325 different classrooms across Tennessee participated in the project. Classroom identifiers are not available, but \citeauthor{graham2008}'s (\citeyear{graham2008}) matching algorithm is able to uniquely identify 317 of these classrooms in the data. 5,727 students in these classrooms have the complete set of characteristics available that I use in the QR model below. I restrict the analysis to only these kindergarten students.

The outcome of interest is student performance on the \emph{Stanford Achievement Test} in mathematics and reading administered at the end of the 1985-1986 school year. I standardized the raw test scores as in \citet{krueger1999}: First, I computed the empirical distribution functions of the math and reading scores for the pooled sample of regular (with and without teacher's aide) students. Next, I transformed the math and reading scores for students in all three class types into percentiles using the math and reading empirical distribution functions, respectively, obtained in the first step. Finally, to summarize overall performance, I computed the average of the two percentiles for each student. I use this percentile score as the dependent variable in the following analysis. The idea behind \citeauthor{krueger1999}'s normalization is that in the absence of a class size effect, the transformed subject scores for both small and regular class types would have an approximately uniform distribution. 

%The dependent variable in the following is the percentile score variable $\mathit{pscore}$.

The two main covariates of interest are the treatment dummy $\mathit{small}$ indicating whether the student was assigned to a small class and the treatment dummy $\mathit{regaide}$ indicating whether the student was assigned to a regular class with an aide. I consider the following model for the conditional quantiles of the transformed scores:
\begin{align}\label{eq:star1}
Q_{ik}(\tau\mid X_{ik}) = \beta_0(\tau) + \beta_1(\tau) \mathit{small}_{ik} + \beta_2(\tau) \mathit{regaide}_{ik} + \beta_3(\tau)^\top Z_{ik}. 
\end{align}
This specification is similar to the mean regression given in \citeauthor{krueger1999}'s (\citeyear{krueger1999}) Table V.4. The covariate vector $Z_{ik}$ contains a dummy indicating if the student identifies as $\mathit{black}$,\footnote{The sample also contains a large number of students who identify as white and a very small number of students who identify as Hispanic, Asian, American Indian, or other.} a student gender dummy, a dummy indicating whether the student is $\mathit{poor}$ as measured by their access to free school lunch, a dummy indicating if the teacher identifies as black ($\mathit{tblack}$, the other teachers in the sample identify as white), the teacher's years of teaching experience ($\mathit{texp}$), a dummy indicating whether the teacher has at least a master's degree ($\mathit{tmasters}$), and additive school ``fixed effects.'' Because of possible peer effects and unobserved teacher characteristics, I cluster at the classroom level. 

\begin{figure}[htp]
\begin{center}
\includegraphics[width=0.9\textwidth]{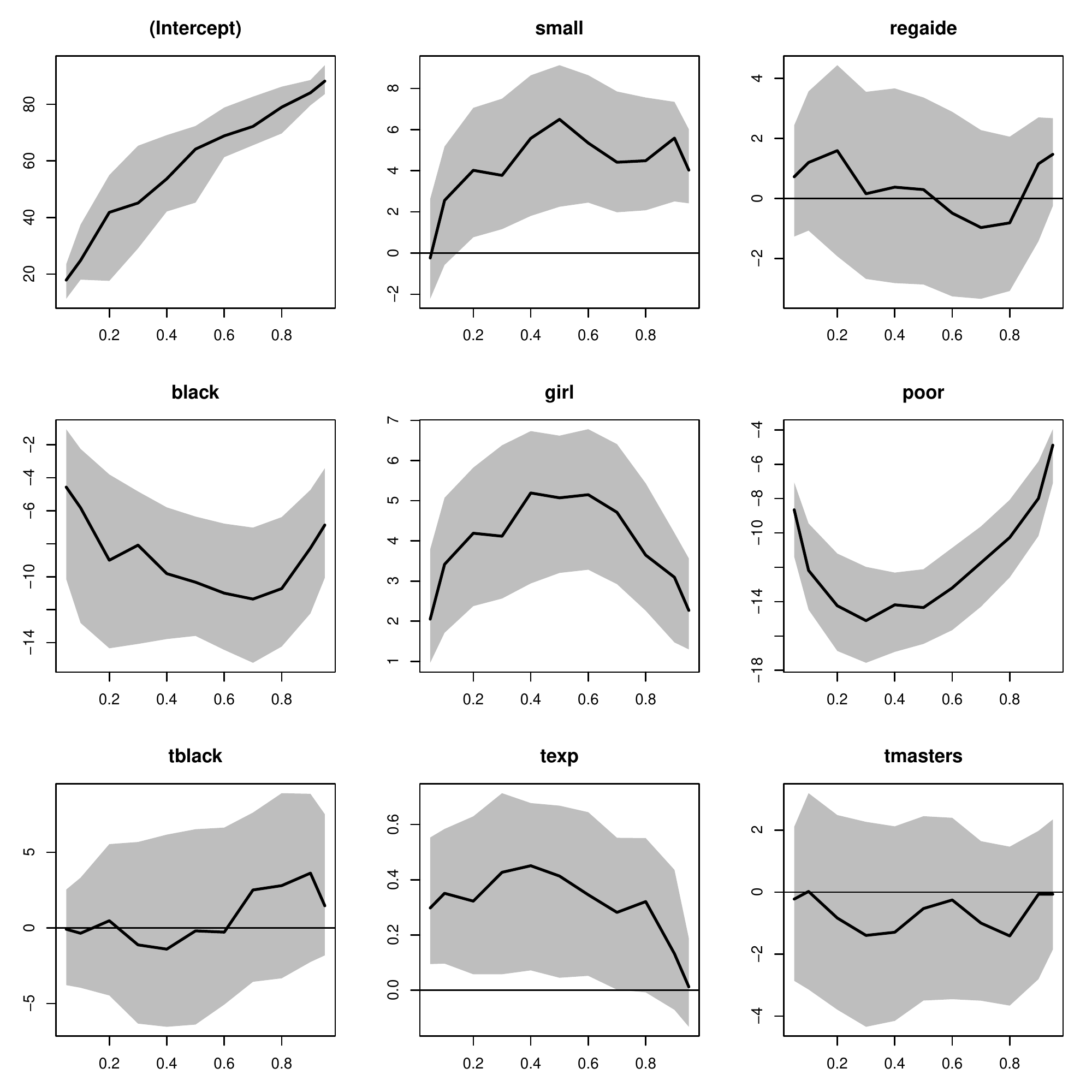}
\caption{QR coefficient estimates $\tau\mapsto\hat{\beta}_n(\tau)$ (solid black lines) of model \eqref{eq:star1}. The regression includes additive school ``fixed effects'' (not shown). The grey areas are pointwise 95\% wild gradient bootstrap confidence intervals based on bootstrap quantiles clustered at the classroom level.} \label{fig:star1}
\end{center}
\end{figure}

The results are shown in Figure \ref{fig:star1}. The solid black lines in each panel plot a coefficient estimate corresponding to a coefficient in \eqref{eq:star1} as a function of $\tau$. The vertical scale is the average percentile score. The grey bands are pointwise 95\% wild gradient bootstrap confidence intervals based on bootstrap quantiles computed from $m=999$ bootstrap simulations with Mammen weights. 
Students assigned to small classes mostly perform better than students assigned to regular classes (with or without aide), although the effect varies across the distribution. For scores above the .2 quantile of the score distribution, the difference is about five percentage points. This is in accordance with \citeauthor{krueger1999}'s (\citeyear{krueger1999}) findings. However, the benefits for students below the .2 quantile are much smaller and become insignificant at the .1 quantile. The impact of a smaller class on students at the very bottom of the score distribution is essentially zero. %The class size reduction therefore did not seem to make a difference for the weakest students. 
In addition, as in \citeauthor{krueger1999}'s mean regression analysis, the effect of being assigned a full-time aide is insignificant.

I now briefly discuss the other covariates. Black students perform worse than non-black students with otherwise identical characteristics; this is particularly pronounced between the first and third quartiles of the conditional score distribution, where black students' scores are about 10 percentage points lower. Girls generally score higher than boys, although the gap is quite small near the tails of the conditional score distribution. Poor students score up to 15 percentage points lower than otherwise identical students; however, this difference is much smaller near the top of the conditional distribution. As \citet{krueger1999} and earlier studies have found, teacher characteristics seem to matter little: their race and education (measured by whether they have a master's degree) have no significant impact. Another year of teaching experience has a small, positive effect for all but the very best students. 

As a referee points out, an issue with Monte Carlo studies such as those in the preceding section is that the data sets used in simulations are likely to be quite different from real data sets. I therefore also evaluate the performance of the wild gradient bootstrap and the alternative methods introduced in Experiment \ref{ex:signi} above through placebo interventions in the Project STAR data. 
\begin{experiment}[Placebo interventions]\label{ex:placebo}
%For this experiment, I removed all small classes and all schools with fewer than two regular-size classes without teacher's aide from the sample. This reduced the data set from 79 schools with a total of 317 classes to 18 schools with a total of 38 regular-size classes and 26 regular-size classes with aide; no school had more than three classrooms of the same type.

For this experiment, I removed all small classes from the sample so that only 194 regular-size classes (with and without teacher's aide) in the 79 project schools remained. Of these schools, 16 had two regular-size classes without aide and 2 had three such classes.

\begin{table}[htp]
\caption{Rejection frequencies of $\mathrm{H}_0\colon \beta_1(.5) = 0$ in placebo interventions for different values of $\beta_1(.5)$}\label{tab:placebo}
\centering
{\small
\resizebox{\columnwidth}{!}{%
\begin{tabular}{lp{0cm}ccp{0cm}ccp{0cm}ccp{0cm}cccccccc}%
\hline
& 
& \multicolumn{2}{c}{Mammen}
&
& \multicolumn{2}{c}{Rademacher}
&
& \multicolumn{2}{c}{Webb}
& 
& \multicolumn{2}{c}{Ana.}
%&
& \multicolumn{2}{c}{Reg.}
%&
& \multicolumn{2}{c}{Rank}
& \multicolumn{2}{c}{FHH}
\\ 
\cline{3-4}\cline{6-7}\cline{9-10}
&
& cv & se 
&
& cv & se 
&
& cv & se 
&
& \multicolumn{2}{c}{se} 
& \multicolumn{2}{c}{se}
& \multicolumn{2}{c}{score}
& \multicolumn{2}{c}{cv} \\ 
\hline
$\beta_1(.5) = 0$ (size)  & & .072 & .084 & & .077 & .093 & & .085 & .095 & & 
\multicolumn{2}{c}{.277} & \multicolumn{2}{c}{.284} & \multicolumn{2}{c}{.098} &\multicolumn{2}{c}{.311} \\ 
$\beta_1(.5) = 5$ (power) & & .413 & .427 & & .449 & .446 & & .450 & .450 & & 
\multicolumn{2}{c}{.701} & \multicolumn{2}{c}{.714} & \multicolumn{2}{c}{.516} &\multicolumn{2}{c}{.727} \\ 
\hline
\end{tabular}%
}}
\end{table}

In each of these 18 schools, I then randomly assigned one of the regular-size classes without aide the treatment indicator $\mathit{small}=1$. This mimics the random assignment of class sizes within schools in the original sample, even though in this case no student actually attended a small class. Next, I reran the QR in \eqref{eq:star1} and tested, at the 5\% level, the correct null hypothesis that the coefficient on $\mathit{small}$ is zero at the median, $\mathrm{H}_0\colon \beta_1(.5) = 0$, using the same methods as in Experiment \ref{ex:signi}. The rejection frequencies in the first line of results in Table \ref{tab:placebo} show the outcome of repeating this process 1,000 times. The bootstraps were again based on $m=299$ simulations. As can be seen, the wild gradient bootstrap test from Algorithm \ref{al:binf} with the Mammen distribution outperformed all other methods of inference, some by a very large margin. Still, the test over-rejected slightly. This can be attributed to the fact that the treatment effect is now identified off of comparisons within only 18 instead of 79 schools, which makes the estimation problem much more challenging than in the actual data. The size of the tests in the placebo experiment can, in that sense, be viewed as an upper bound for the size of the tests in the original sample. 

I also investigated power by increasing the percentile scores of all students in the randomly drawn small classes of the placebo experiment by 5. This increase is of the same order of magnitude as the estimated treatment effect at the median in the actual sample. Then I repeatedly tested the incorrect hypothesis $\mathrm{H}_0\colon \beta_1(.5) = 0$ (the correct value is $\beta_1(.5) = 5$) with the same experimental setup as before. The results are shown in the second line of Table \ref{tab:placebo}. Despite the now much smaller sample, the wild gradient bootstrap was able to reject the null in a large number of cases. The other methods rejected more often, but this was likely driven by their size distortion. Notable here is the high power of the \citet{wanghe2007} rank score test despite its relatively mild over-rejection under the null.
\hfill $\square$
\end{experiment}

%\pub{\addlines}

\phantomsection%
\label{rev:intraclass}%
The large differences in the finite-sample size of the methods of inference considered in the preceding experiment can be attributed to the within-cluster dependence in the data. This is also supported by a back-of-the-envelope comparison of the results here to the Monte Carlo experiments in Section \ref{s:mc}. For the Monte Carlo DGP \eqref{eq:mc}, the within-cluster correlation coefficient of the outcome variable can be shown to be approximately $\varrho$. For the Project STAR data, the \citet{karlinetal1981} intraclass correlation coefficient $$ \hat{\varrho}_n := \frac{\sum_{i=1}^n\sum_{k=1}^{c_i}\sum_{l\neq k} (Y_{ik} - \bar{Y}_n)(Y_{il} - \bar{Y}_n)/(c_i -1)}{\sum_{i=1}^n\sum_{k=1}^{c_i} (Y_{ik} - \bar{Y}_n)^2},\quad\text{where}\quad\bar{Y}_n = \frac{\sum_{i=1}^n\sum_{k=1}^{c_i}Y_{ik}}{\sum_{i=1}^n c_i},$$ of the percentile score is $.319$. This is a consistent estimate of the within-cluster correlation coefficient of the percentile score as long as both its mean and within-cluster covariance structure are identical across clusters. (Neither of these conditions is needed for any of the theoretical results in this paper.) At $\varrho = \hat{\varrho}_n$, the results of Experiments \ref{ex:signi} and \ref{ex:placebo} are quite similar, with the exception that the rank score test performed much better in Experiment \ref{ex:placebo} than the test based on analytical cluster-robust standard errors.

%This suggests that the data under consideration falls well within the region of within-cluster correlations where the wild gradient bootstrap test had good size and power.

Finally, before concluding this section, Figure \ref{fig:star2} illustrates the difference between a 95\% pointwise confidence interval based on a Powell sandwich estimator (as described in Experiment \ref{ex:signi}) that does not control for within-cluster correlation (dotted lines), the wild gradient bootstrap confidence interval shown in Figure \ref{fig:star1} (grey), and a 95\% wild bootstrap confidence band for the entire coefficient function of $\mathit{small}$ weighted by the bootstrap covariance matrix (dashed). As can be seen from the size of the grey area, not accounting for the possibility of peer effects and unobserved teacher characteristics via cluster-robust inference appears to give a false sense of precision at most quantiles. However, as the confidence band shows, we can conclude that the effect of the small class size is significantly positive over a large part of the support of the score distribution.

\begin{figure}[thp]
\begin{center}
\includegraphics[width=0.4\textwidth]{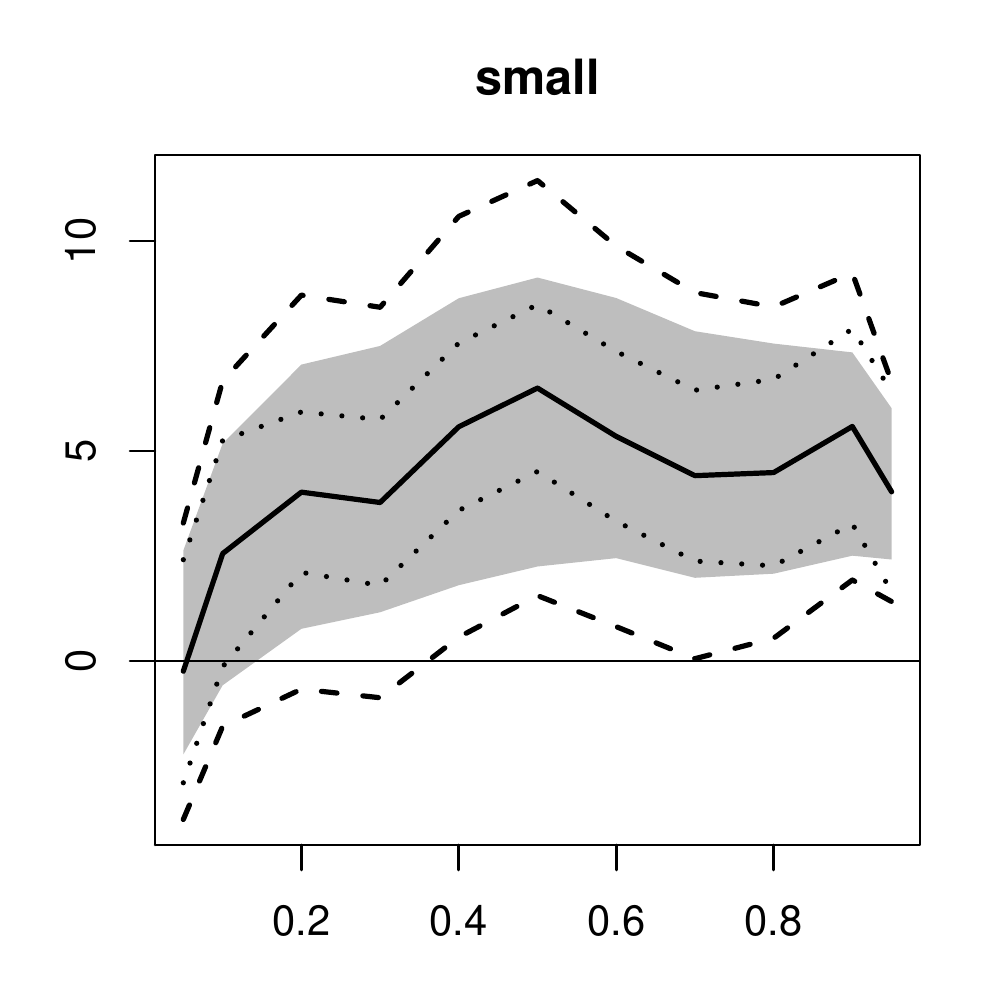}
\caption{QR coefficient estimate and pointwise confidence interval for $\mathit{small}$ from Figure \ref{fig:star1}, a 95\% pointwise confidence interval not robust to within-cluster correlation (dotted lines), and a 95\% wild bootstrap confidence band for the entire coefficient function (dashed).} \label{fig:star2}
\end{center}
\end{figure}

\section{Conclusion}\label{s:conc}
In this paper I develop a wild bootstrap procedure for cluster-robust inference in linear QR models. I show that the bootstrap leads to asymptotically valid inference on the entire QR process in a setting with a large number of small, heterogeneous clusters and provides consistent estimates of the asymptotic covariance function of that process. The proposed bootstrap procedure is easy to implement and performs well even when the number of clusters is much smaller than the sample size. A brief application to Project STAR data is provided. It is still an open question how cluster-level fixed effects that correspond to the intuitive notion of identifying parameters from within-cluster variation can fit into the present framework; this is currently under investigation by the author. Another question is if the jackknife can improve on the bootstrap in the current context; recent results by \citet{portnoy2014} for censored regression quantiles suggest this possibility.

\pre{
\section*{Supplementary Materials}

\begin{description}
\item[\rmfamily Supplementary appendix:] Technical appendix containing auxiliary results and proofs.
\end{description}

\phantomsection
\addcontentsline{toc}{section}{References}
\putbib[qspec]
\end{bibunit}
\newpage\setcounter{page}{1}
}

\appendix

\pre{\begin{bibunit}[chicago]}

\pub{\section*{Appendix}}

\pre{
\begin{center}
\textbf{\normalsize\uppercase{Supplementary Appendix to\\ ``Cluster-Robust Bootstrap Inference in\\ Quantile Regression Models''}}
\end{center}

\noindent Appendix \ref{sec:aux} introduces some notation and definitions that are used throughout the remainder of the paper. Then I state some auxiliary results. All proofs can be found in Appendix \ref{sec:proofs}.\addlines[2]
}

\section{Auxiliary Definitions and Results}\label{sec:aux}

\pub{
I first introduce some notation and definitions that are used throughout the remainder of the paper. Then I state some auxiliary results. All proofs can be found in the next section.
}
\begin{proof}[Notation]
For vectors $a$ and $b$, I will occasionally write $(a,b)$ instead of $(a^\top,b^\top)^\top$ if the dimensions are not essential. Take $(Y_{ik},X_{ik})=0$ for $c_i<k\leq c_{\max}$ whenever $c_i < c_{\max}$ and let $Y_i = (Y_{i1},\dots,Y_{ic_{\max}})$ and $X_i = (X_{i1}^\top,\dots,X_{ic_{\max}}^\top)^\top$. Let $\ep_n f = n^{-1/2}\sum_{i=1}^n ( f(Y_i, X_i) - \ev f(Y_i, X_i))$ be the empirical process evaluated at some function $f$ and let $\mean_n f = n^{-1}\sum_{i=1}^n f(Y_i,X_i)$ be the empirical average at $f$. I will frequently use the notation $|\ep_n f|_\mathcal{F} := \sup_{f\in\mathcal{F}}|\ep_n f|$ for functional classes and $|f_\theta|_\Theta := \sup_{\theta\in\Theta} |f_\theta|$ for functions indexed by parameters.  Define 
\begin{align*}
&m_{\beta,\tau}(Y_i,X_i) = \sum_{k=1}^{c_{\max}}\rho_{\tau}(Y_{ik} - X_{ik}^\top\beta), \quad z_{\beta,\tau}(Y_i,X_i) = \sum_{k=1}^{c_{\max}}\psi_{\tau}(Y_{ik} - X_{ik}^\top\beta)X_{ik},\\
&g_{\beta_1,\tau_1,\beta_2,\tau_2}(Y_i,X_i) = \sum_{k=1}^{c_{\max}}\sum_{l=1}^{c_{\max}} \psi_{\tau_1}(Y_{ik}-X_{ik}^\top\beta_1) \psi_{\tau_2}(Y_{il}-X_{il}^\top\beta_2)X_{ik}X_{il}^\top
\end{align*}
and the corresponding classes
\begin{align*}
&\mathcal{M}_\delta = \{ m_{\beta,\tau} - m_{\beta(\tau),\tau} : |\beta-\beta(\tau)|\leq \delta, \beta\in\Beta, \tau\in\Tau \}, \quad 
\mathcal{Z} = \{z_{\beta,\tau} : \beta\in\Beta, \tau\in\Tau \},\\ 
&\mathcal{G} = \{g_{\beta_1,\tau_1,\beta_2,\tau_2} : (\beta_1,\tau_1,\beta_2,\tau_2)\in\Beta\times\Tau\times\Beta\times\Tau\}.
\end{align*}
Write the $j$th coordinate projection as $x = (x_1,\dots, x_j ,\dots ,x_d) \mapsto \pi_j(x) = x_j$. Define a pseudometric $\varrho$ on $\mathcal{Z}$ by $$ \varrho\bigl( z_{\beta,\tau}, z_{\beta',\tau'} \bigr) = \max_{1\leq j\leq d} \sup_{n\geq 1} \bigl( \mean_n \ev(\pi_j\circ z_{\beta,\tau}-\pi_j\circ z_{\beta',\tau'})^2 \bigr)^{1/2}. $$
Denote by $\pi_{jh}$ the function that picks out the entry in the $j$th row and $h$th column of a matrix. For a matrix $A$, denote the Frobenius norm by $|A|=\mathrm{trace}^{1/2}(AA^\top)$;  if $A$ is a vector, this is the Euclidean norm. Let $\lambda_{\min}(A)$ be the smallest eigenvalue of a symmetric matrix $A$. Let $H_n(\beta) = n^{-1} \sum_{i=1}^n \sum_{k=1}^{c_i} \ev f_{ik}( X_{ik}^\top\beta\mid X_{ik}) X_{ik}X_{ik}^\top$ and note that $J_n(\tau) = H_n(\beta(\tau))$. Define the mean value $$I_n(\beta,\tau) = \int_0^1 H_n\Bigl(\beta(\tau) + t\bigl( \beta -\beta(\tau)\bigr)\Bigr)\, dt,$$ where the integral is taken componentwise. For scalars $a$ and $b$, the notation $a\lesssim b$ means $a$ is bounded by an absolute constant times $b$. 
\end{proof}
Some expressions above and below may be non-measurable; probability and expectation of these expressions are understood in terms of outer probability and outer expectation \citep[see, e.g.][p.\ 6]{vandervaartwellner1996}. Application of Fubini's theorem to such expectations requires additional care. A measurability condition that restores the Fubini theorem for independent non-identically distributed (inid) data is the ``almost measurable Suslin'' condition of \citet{kosorok2003}. It is satisfied in all applications below.
\begin{lemma}\label{l:suslin}
If Assumptions \ref{as:data}-\ref{as:smooth} are satisfied, then $\{ z_{\beta(\tau),\tau} : \tau\in\Tau\}$, $\{ z_{\beta,\tau} : \beta\in\Beta, \tau\in\Tau\}$ and $\{ m_{\beta,\tau}-m_{\beta(\tau),\tau} : \beta\in\Beta, \tau\in\Tau\}$ are almost measurable Suslin.
\end{lemma}

The following lemmas are used in the proofs of the results stated in the main text.

\begin{lemma}\label{l:maxineq}
Suppose Assumptions \ref{as:data}-\ref{as:smooth} hold. Then,
\begin{compactenum}[\upshape(i)]
\item\label{l:maxineq_i} $\sup_{n\geq 1}\ev |\mathbb{G}_n m|^q_{\mathcal{M}_\delta} \lesssim \delta^q$,
%\item\label{l:maxineq_iii} $\sup_{n\geq 1}\ev |\mathbb{G}_n m|^q_{\mathcal{M}} < \infty$,\marginpar{!}
\item\label{l:maxineq_ii} $\sup_{n\geq 1}\ev |\mathbb{W}_n(\beta,\tau)|^q_{\Beta\times\Tau} < \infty$, and
\item\label{l:maxineq_iv} $\sup_{n\geq 1}\ev |\mathbb{G}_n z|^q_{\mathcal{Z}} < \infty$.
\end{compactenum}
\end{lemma}

\begin{lemma}\label{l:lowerbound}
If Assumptions \ref{as:data}-\ref{as:gstat} are satisfied, then $|\beta - \beta(\tau)|^2\lesssim {M}_{n}(\beta,\tau) - {M}_{n}(\beta(\tau),\tau)$ for all $\beta\in\Beta$ and all $\tau\in\Tau$.
\end{lemma}

\begin{lemma}\label{l:uniint} Suppose Assumptions \ref{as:data}-\ref{as:gstat} are true. Then, for all $0 < p < q$,
\begin{compactenum}[\upshape(i)]
\item\label{l:uniint:bootsam} $\sup_{n\geq 1}\ev|\sqrt{n}( \hat{\beta}^*_n(\tau) - \hat{\beta}_n(\tau))|^{p}_\Tau< \infty$,
\item\label{l:uniint:bootpop} $\sup_{n\geq 1}\ev|\sqrt{n}( \hat{\beta}^*_n(\tau) - \beta(\tau))|^{p}_\Tau< \infty$, and
\item\label{l:uniint:sampop} $\sup_{n\geq 1}\ev|\sqrt{n}( \hat{\beta}_n(\tau) - \beta(\tau))|^{p}_\Tau< \infty$.
\end{compactenum}
\end{lemma}

\begin{lemma}\label{l:jbound}
Suppose Assumptions \ref{as:data}-\ref{as:smooth} hold with $q\geq 3$. Then $I_{n}(\beta,\tau)$ is invertible and, uniformly in $n$, $|I_{n}^{-1}(\beta,\tau) - J_{n}^{-1}(\tau)| \lesssim |\beta-\beta(\tau)|$ for all $\beta\in\Beta$ and $\tau\in\Tau$.
\end{lemma}

\begin{lemma}\label{l:junif}
If Assumptions \ref{as:data}-\ref{as:gstat} hold with $q\geq 3$, then  $|J_{n}(\tau) - J(\tau)|_{\Tau}$ and $|J_{n}^{-1}(\tau) - J^{-1}(\tau)|_{\Tau}$ converge to zero.
\end{lemma}

%\begin{lemma}\label{l:mconv}
%Let $\{Z_{n}(\tau):\tau\in\Tau\}$ be a sequence of arbitrary stochastic processes such that $|Z_n(\tau)|_{\Tau}\pto^{\prob} 0$ and $\sup_{n\geq 1}\ev |Z_{n}(\tau)|^r_{\Tau}<\infty$ for some $r>1$. Then $\ev^*|Z_n(\tau)|_{\Tau}\pto^{\prob} 0$.
%\end{lemma}

%In the following, I will simply write $\varrho(\beta,\tau)(\beta',\tau')$ instead of $\varrho( (\beta,\tau),(\beta',\tau'))$.
\begin{lemma}\label{l:sequi}
If Assumptions \ref{as:data}-\ref{as:smooth} hold, then
\begin{compactenum}[\upshape(i)]
\item\label{l:sequi:totbound} $\varrho( z_{\beta,\tau},z_{\beta',\tau'} ) \lesssim |((\beta-\beta')^\top,\tau-\tau')|^{(q-2)/(2q)}$ for all $\beta,\beta'\in\Beta$ and $\tau,\tau'\in\Tau$, and

%\item \label{l:sequi:totmetric}  $(\Beta\times\Tau, \varrho)$ is a totally bounded metric space, and

\item\label{l:sequi:asy} $z\mapsto \ep_n z$ is $\varrho$-stochastically equicontinuous on $\mathcal{Z}$, i.e., for all $\varepsilon,\eta >0$, there is a $\delta > 0$ such that $$ \limsup_{n\to\infty} \prob\biggl(\sup_{\varrho( z_{\beta,\tau}, z_{\beta',\tau'} )<\delta} |\ep_n z_{\beta,\tau}-\ep_n z_{\beta',\tau'} | > \eta\biggr) < \varepsilon. $$
\end{compactenum}
\end{lemma}

%Members of this class can again be understood as maps with domain $(\mathbb{R}^{c_{\max}}\times\mathbb{R}^{d\times c_{\max}})^2$ by filling up smaller clusters with zeros. %Define a pseudometric $\sigma$ by $$ \sigma(\beta_1,\tau_1,\beta_2,\tau_2)(\beta_1',\tau_1',\beta_2',\tau_2') = \max_{j,h\in\{1,\dots, d\}} \sup_{n\geq 1} \mean_n \ev|\pi_{jh}\circ g_{\beta_1,\tau_1,\beta_2,\tau_2}-\pi_{jh}\circ g_{\beta_1',\tau_1',\beta_2',\tau_2'}|. $$
\begin{lemma}\label{l:sigmaequi}
If Assumptions \ref{as:data}-\ref{as:smooth} hold, then $|\mean_n (g - \ev g)|_\mathcal{G}\pto^{\prob} 0$.
%\begin{compactenum}[\upshape(i)]
%\item\label{l:sigmaequi:totbound} for all $\tau_1,\tau_2,\tau_1',\tau_2'\in\Tau$ and $\beta_1,\beta_2,\beta_1',\beta_2'\in\Beta$, $$\sigma(\beta_1,\tau_1,\beta_2,\tau_2)(\beta_1',\tau_1',\beta_2',\tau_2')  \lesssim \bigl|\bigl((\beta_1-\beta_1')^\top, \tau_1-\tau_1', (\beta_2-\beta_2')^\top, \tau_2-\tau_2'\bigr)\bigr|^{(q-2)/(2q)},$$

%\item\label{l:sigmaequi:asy} 
%is $\sigma$-stochastically equicontinuous on $\mathcal{G}$.%, i.e., for all $\varepsilon,\eta >0$, there is a $\delta > 0$ such that $$ \limsup_{n\to\infty} \prob\biggl(\sup_{\sigma(\beta_1,\tau_1,\beta_2,\tau_2)(\beta_1',\tau_1',\beta_2',\tau_2')<\delta} |\mean_n g_{\beta_1,\tau_1,\beta_2,\tau_2}-  \mean_n g_{\beta_1',\tau_1',\beta_2',\tau_2'} | > \eta\biggr) < \varepsilon. $$
%\end{compactenum}
\end{lemma}

\section{Proofs}\label{sec:proofs}

\begin{proof}[Proof of Lemma \ref{l:suslin}]
To verify the almost measurable Suslin property for $m_{\beta(\tau),\tau}$, I start by establishing two useful facts: First, because $|1\{a< b\} - 1\{a< c\}|\leq 1\{|a-b|< |b-c|\}$ for $a,b,c\in\mathbb{R}$,  each realization of $Y_{ik}$ and $X_{ik}$ satisfies
$$ | 1\{ Y_{ik} < X_{ik}^\top \beta(\tau) \} - 1\{ Y_{ik} < X_{ik}^\top \beta(\tau') \} | \leq 1\{ |Y_{ik} - X_{ik}^\top \beta(\tau)| < |X_{ik}||\beta(\tau)-\beta(\tau')| \}.$$
Second, the eigenvalues of $J_n(\tau)$ are bounded away from zero uniformly in $\tau$ and $n$ by Assumption \ref{as:smooth}. The same assumption and the inverse function theorem applied to $\sum_{i=1}^n\sum_{k=1}^{c_i} \ev (\tau - 1\{Y_{ik} < X_{ik}^\top \beta(\tau)\})X_{ik} = 0$ then give $d\beta(\tau)/d\tau = J_n^{-1}(\tau)n^{-1}\sum_{i=1}^n\sum_{k=1}^{c_i}\ev X_{ik}$. Because the Frobenius norm is also the 2-Schatten norm, we have $|J_{n}^{-1}(\tau)| \leq \sqrt{d\lambda_{\min}(J_{n}(\tau))^{-1}}$ and hence $d\beta(\tau)/d\tau$ is bounded uniformly in $\tau\in\Tau$ by some $C > 0$. The mean-value theorem yields
\begin{equation}\label{eq:betalip}
|\beta(\tau)-\beta(\tau')|\leq  C |\tau-\tau'|, \qquad \tau,\tau'\in\Tau.
\end{equation}

For each $\tau\in\Tau$, combine the preceding two displays with the Lo\`eve $c_r$ inequality to obtain ($\mathbb{Q}$ here is the set of rationals)
\begin{align*}
&\inf_{\tau'\in\Tau\cap\mathbb{Q}} \mean_n( \pi_j \circ z_{\beta(\tau),\tau} - \pi_j \circ z_{\beta(\tau'),\tau'} )^2\\ &\qquad \lesssim \inf_{\tau'\in\Tau\cap\mathbb{Q}} \frac{1}{n} \sum_{i=1}^n\sum_{k=1}^{c_i} \bigl((\tau-\tau')^2 + 1\{ |Y_{ik} - X_{ik}^\top \beta(\tau)| < |X_{ik}|C |\tau-\tau'| \}\bigr) \pi_j(X_{ik})^2.
\end{align*}
The infimum on the right must be smaller than $n^{-1}\sum_{i=1}^n\sum_{k=1}^{c_i} (\delta^2 + 1\{ |Y_{ik} - X_{ik}^\top \beta(\tau)| < |X_{ik}|C\delta \})\pi_j(X_{ik})^2$ for every $\delta > 0$. Conclude that the infimum is $n^{-1}\sum_{i=1}^n\sum_{k=1}^{c_i}1\{ |Y_{ik} - X_{ik}^\top \beta(\tau)| < 0 \}\pi_j(X_{ik})^2 = 0$. This does not change if we take suprema over $\tau\in\Tau$ on both sides of the display. If follows that
$$ \prob \biggl( \sup_{\tau\in\Tau} \inf_{\tau'\in\Tau\cap\mathbb{Q}} \mean_n( \pi_j \circ z_{\beta(\tau),\tau} - \pi_j \circ z_{\beta(\tau'),\tau'} )^2 >0 \biggr) = 0, \qquad 1\leq j\leq d,$$ which makes $\{ z_{\beta(\tau),\tau} : \tau\in\Tau\}$ almost measurable Suslin by Lemma 2 of \citet{kosorok2003}. Nearly identical calculations verify the same property for $\{ z_{\beta,\tau} : \beta\in\Beta, \tau\in\Tau\}$. Because
\begin{align*}
&\rho_\tau(y-x^\top\beta) - \rho_\tau\bigl(y-x^\top\beta(\tau)\bigr) - \bigl(\rho_{\tau'}(y-x^\top\beta') - \rho_{\tau'}\bigl(y-x^\top\beta(\tau')\bigr)\bigr)\\
&\qquad \leq |x|\bigl(|\tau-\tau'|\bigl(\beta - \beta(\tau)\bigr) + |\beta - \beta'| + |\beta(\tau) - \beta(\tau')|\bigr),
\end{align*}
the process $\{ m_{\beta,\tau}-m_{\beta(\tau),\tau} : \beta\in\Beta, \tau\in\Tau\}$ is almost measurable Suslin as well.
\end{proof}

\begin{proof}[Proof of Lemma \ref{l:maxineq}]
\begin{inparaenum}
\item[\eqref{l:maxineq_i}]  The subgraph of a real-valued function $f$ is defined as $\{ (x,t): f(x) > t \}$. Write the subgraph of $\rho(a - b) - \rho(a)$ as
\begin{align*}
&(\{ a \geq 0 \}\cap \{a\geq b\}\cap\{ b < -t/\tau \})\cup (\{a \geq 0 \}\cap \{ a < b \}\cap\{ (1-\tau)b - a > t \})\\
&\quad \cup(\{ a < 0 \}\cap \{a < b\}\cap\{ a-b < t/(\tau-1) \})\cup (\{a < 0\}\cap \{ a \geq b \}\cap\{ a-\tau b > t \}).
\end{align*}
Now take $a = y - x^\top\beta(\tau)$ and $b  = x^\top (\beta - \beta(\tau))$ so that, e.g., $\{ a > 0\} = \{ (y,x) :  y - x^\top\beta(\tau) > 0\}$. 
Hence,
\begin{align*}
&\{ (1-\tau)b - a > t \} = \bigl\{ (y,x) : y - x^\top \bigl(\tau\beta(\tau) + (1-\tau)\beta\bigr) < - t \bigr\}\quad\text{and}\\
&\{ a-\tau b > t \} = \bigl\{ (y,x) : y - x^\top \bigl(\tau\beta + (1-\tau)\beta(\tau)\bigr) > t \bigr\}.
\end{align*}
By convexity of $\Beta$, the two collections of sets in the display, indexed by $\beta\in\Beta$, $\tau\in\Tau$, and $t\in\mathbb{R}$, are contained in the collection of sets $\{(y,x) : y-x^\top\beta < -t \}$ and $\{(y,x) : y-x^\top\beta > t \}$, respectively, indexed by $\beta\in\Beta$ and $t\in\mathbb{R}$.

The collection of sets $\{ v\in\mathbb{R}^{d+2} : v^\top \lambda \leq 0 \}$ indexed by $\lambda\in\mathbb{R}^{d+2}$ is a Vapnik--\u{C}hervonenkis (VC) class of sets \citep[see][Problem 2.6.14]{vandervaartwellner1996}; the same holds for the collection $\{ v\in\mathbb{R}^{d+2} : v^\top \lambda > 0 \}$ by Lemma 2.6.17(i) of \citeauthor{vandervaartwellner1996}. Because $\Beta\subset \mathbb{R}^d$, each individual set in the subgraph above indexed by $\beta\in\Beta$, $\tau\in\Tau$, and $t\in\mathbb{R}$ is contained in one of these two VC classes. Subclasses of VC classes are VC classes themselves. Conclude from \citeauthor{vandervaartwellner1996}'s Lemma 2.6.17(ii) and (iii) that the subgraph above is a VC class. Therefore the map
\begin{align*}
\rho_\tau(y-x^\top\beta) - \rho_\tau\bigl(y-x^\top\beta(\tau)\bigr)
\end{align*}
indexed by $\beta$ and $\tau$ is a VC subgraph class. %As such, it satisfies \citeauthor{pollard1982}'s (\citeyear{pollard1982}) uniform entropy condition; see also \citet[p.~239, with the condition being $J(1,\mathcal{F}) < \infty$ in their notation]{vandervaartwellner1996}. 
Sums of functions from VC subgraph classes do not necessarily form VC subgraph classes, but their uniform entropy numbers behave like those of VC subgraph classes if the envelopes are increased accordingly \citep[p.\ 157]{kosorok2008}. Because the absolute value of the preceding display is bounded above by $|x^\top(\beta-\beta(\tau))|$, we can use $\sum_{k=1}^{c_{\max}}|X_{ik}|\delta$ as an envelope for $m(Y_i,X_i)$ with $m\in\mathcal{M}_\delta$. The class $\mathcal{M}_\delta$ then has a finite uniform entropy integral in the sense of \citet[condition (2.5.1), p.\ 127]{vandervaartwellner1996}.

Use \citeauthor{vandervaartwellner1996}'s (\citeyear{vandervaartwellner1996}) Theorem 2.14.1 (which applies to inid observations if the reference to their Lemma 2.3.1 is replaced by a reference to their Lemma 2.3.6 and my Lemma \ref{l:suslin} is used to justify their symmetrization argument) to deduce that $$ \ev |\ep_n m|^q_{\mathcal{M}_\delta} \lesssim \frac{1}{n}\sum_{i=1}^n\ev\biggl|\sum_{k=1}^{c_{\max}} | X_{ik}|\delta \biggr|^q  \lesssim \delta^q \sup_{i,k} \ev|X_{ik}|^q.$$
The right-hand side is finite by assumption, which completes the proof.

%\item[\eqref{l:maxineq_iii}] This follows from \eqref{l:maxineq_i} with $\delta = \diam B $.

\item[\eqref{l:maxineq_ii}] By the Lo\`eve $c_r$ inequality, it suffices to show that each of the $d$ elements of $\mathbb{W}_n$ has the desired property.  Arguments similar to the ones given in the first part of the proof establish that the collection of functions $(y,x,w) \mapsto w\pi_j(x) \psi_{\tau}(y-x^\top \beta)$ indexed by $(\beta,\tau)\in\Beta\times\Tau$ is a VC subgraph class. As such, it satisfies \citeauthor{pollard1982}'s  (\citeyear{pollard1982}) uniform entropy condition. By Theorem 3 of \citet[p.\ 2273]{andrews1994}, the class of functions with finite uniform entropy is stable under addition as long as the envelope function is increased accordingly. An appropriate envelope for the set of functions $\sum_{k=1}^{c_{\max}}W_i\pi_j(X_{ik})\psi_{\tau}(Y_{ik} - X_{ik}^\top\beta)$ indexed by $(\beta,\tau)\in{\Beta\times\Tau}$ is
%\begin{align}
$F_j(W_i,X_{i}) = 2|W_i|\sum_{k=1}^{c_{\max}}|\pi_j(X_{ik})|.$
%\end{align}
In addition, the components of $\mathbb{W}_n(\beta,\tau)$ are almost measurable Suslin by Lemma \ref{l:suslin} and the bound
\begin{align*}
&\frac{1}{\sqrt{n}}\sum_{i=1}^nW_i^2 \biggl(\sum_{k=1}^{c_{\max}}\bigl(\psi_{\tau}(Y_{ik} - X_{ik}^\top\beta) - \psi_{\tau'}(Y_{ik} - X_{ik}^\top\beta')\bigr)\pi_j(X_{ik})\biggr)^2\\
&\qquad \leq \sqrt{n} \max_{1\leq i\leq n} W_i^2 \mean_n( \pi_j \circ z_{\beta,\tau} - \pi_j \circ z_{\beta',\tau'} )^2.
\end{align*}
Conclude from \citeauthor{vandervaartwellner1996}'s (\citeyear{vandervaartwellner1996}) Theorem 2.14.1 (see part \eqref{l:maxineq_i} of the proof above) and independence of the bootstrap weights from the data that
\begin{align*}
\ev |\pi_j \circ \mathbb{W}_n(\beta,\tau)|^q_{\Beta\times\Tau} \lesssim \ev \mean_n F_j^q \lesssim \ev|W|^q \sup_{i,k}\ev|X_{ik}|^q,
\end{align*}
which is finite by assumption.

\item[\eqref{l:maxineq_iv}] This follows from \eqref{l:maxineq_ii} with $W \equiv 1$.
\end{inparaenum}
\end{proof}

\begin{proof}[Proof of Lemma \ref{l:lowerbound}]
By a Taylor expansion about $\beta(\tau)$,
$$ M_n(\beta,\tau) - M_n(\beta(\tau),\tau) = M_n'(\beta(\tau),\tau)\bigl( \beta-\beta(\tau) \bigr) + \bigl( \beta-\beta(\tau) \bigr)^\top I_n(\beta,\tau)\bigl( \beta-\beta(\tau) \bigr)/2. $$ The first term on the right is zero because $\beta = \beta(\tau)$ minimizes $M_n(\beta,\tau)$. By the properties of Rayleigh quotients, the second term on the right is at least as large as $\lambda_{\min}(I_n(\beta,\tau))|\beta-\beta(\tau)|^2/2$. Assumption \ref{as:smooth} implies that $\lambda_{\min}(H_n(\beta))$ is bounded away from zero uniformly in $\beta$ and $n$. For every non-zero $a\in\mathbb{R}^d$, we must have $a^\top I_n(\beta,\tau)a \geq \inf_{\beta,n} a^\top H_n(\beta) a$ and therefore $\lambda_{\min}(I_n(\beta,\tau))$ is bounded away from zero as well, uniformly in $\beta$, $\tau$, and $n$.
\end{proof}

\begin{proof}[Proof of Lemma \ref{l:uniint}]
This proof uses a (non-trivial) modification of the strategy of proof used by \citet{kato2011}. Without loss of generality, take $p\geq 2$. Because $\Beta$ is bounded, for every $\varepsilon > 0$ there exists some ${\tau}_n^*\in\Tau$ such that $|\sqrt{n}( \hat{\beta}^*_n(\tau) - \beta(\tau))|_\Tau$ and $|\sqrt{n}( \hat{\beta}^*_n(\tau^*_n) - \beta(\tau^*_n))|$ differ at most by $\varepsilon$. Choose $\varepsilon < 2$. For every nonnegative integer $N$, the inequality $1\{a + b > c\}\leq 1\{2a > c\}+1\{2b > c\}$ with $a,b,c\in\mathbb{R}$ then yields
\begin{align*}
\prob \Bigl( \bigl|\sqrt{n}\bigl( \hat{\beta}^*_n(\tau) - \beta(\tau)\bigr)\bigr|_\Tau > 2^{N+1} \Bigr) &\leq \prob \Bigl( \varepsilon + \bigl|\sqrt{n}\bigl( \hat{\beta}^*_n(\tau^*_n) - \beta(\tau^*_n)\bigr)\bigr| > 2^{N+1} \Bigr)\\
&\leq \prob \Bigl( \bigl|\sqrt{n}\bigl( \hat{\beta}^*_n(\tau^*_n) - \beta(\tau^*_n)\bigr)\bigr| > 2^N \Bigr).
\end{align*}

Define shells $S_{jn} = \{ (\beta,\tau) \in \Beta\times\Tau : 2^{j-1} < \sqrt{n}|\beta - \beta(\tau)|\leq 2^j\}$ for integer $j\geq 1$. If the event in the second line of the preceding display occurs, then there exists some $j\geq N$ such that $(\hat{\beta}^*_n(\tau^*_n), \tau^*_n)\in S_{jn}$. Because $\hat{\beta}^*_n(\tau)$ minimizes $\mathbb{M}^*_{n}(\beta,\tau)$ for every $\tau\in\Tau$, including ${\tau}_n^*$, this implies
$\inf_{(\beta,\tau)\in S_{jn}} \mathbb{M}^*_{n}(\beta,\tau) - \mathbb{M}^*_{n}(\beta(\tau),\tau) \leq 0$. The union bound then gives
\begin{align}\label{eq:minbound}
\begin{split}
&\prob \Bigl( \bigl|\sqrt{n}\bigl( \hat{\beta}^*_n(\tau^*_n) - \beta(\tau^*_n)\bigr)\bigr| > 2^N \Bigr)\\ &\qquad\leq \sum_{j\geq N} \prob\biggl( \inf_{(\beta,\tau)\in S_{jn}} \mathbb{M}^*_{n}(\beta,\tau) - \mathbb{M}^*_{n}(\beta(\tau),\tau) \leq 0 \biggr).
\end{split}
\end{align}

Add and subtract to decompose $\mathbb{M}^*_{n}(\beta,\tau) - \mathbb{M}^*_{n}(\beta(\tau),\tau)$ into
\begin{align*}
{M}_{n}(\beta,\tau) - {M}_{n}(\beta(\tau),\tau) + \mathbb{G}_n(m_{\beta,\tau}-m_{\beta(\tau),\tau})/\sqrt{n} + \mathbb{W}_n(\tau)^\top\bigl(\beta-\beta(\tau)\bigr)\sqrt{n}.
%\\
%&\quad+\mathbb{M}_{n}(\beta,\tau) - \mathbb{M}_{n}(\beta(\tau),\tau) - \bigl({M}_{n}(\beta,\tau) - {M}_{n}(\beta(\tau),\tau)\bigr)\\ 
%&\quad+\mathbb{M}^*_{n}(\beta,\tau) - \mathbb{M}^*_{n}(\beta(\tau),\tau) - (\mathbb{M}_{n}(\beta,\tau) - \mathbb{M}_{n}(\beta(\tau),\tau))
\end{align*}
By Lemma \ref{l:lowerbound}, ${M}_{n}(\beta,\tau) - {M}_{n}(\beta(\tau),\tau)\geq c|\beta - \beta(\tau)|^2\geq c 2^{2j-2}/n$ on $S_{jn}$ for some $c>0$. For each $j$, %the events on the right-hand side of \eqref{eq:minbound} are therefore included in the events
we therefore have the inclusion
\begin{align*}
&\biggl\{ \inf_{(\beta,\tau)\in S_{jn}} \mathbb{M}^*_{n}(\beta,\tau) - \mathbb{M}^*_{n}(\beta(\tau),\tau) \leq 0 \biggr\}\\
&\qquad \subset \Bigl\{\bigl|\mathbb{G}_n(m_{\beta,\tau}-m_{\beta(\tau),\tau})\bigr|_{S_{jn}} + \bigl|\mathbb{W}_n(\tau)^\top\bigl(\beta-\beta(\tau)\bigr)\bigr|_{S_{jn}} \geq c 2^{2j-2}/\sqrt{n}\Bigr\}.
\end{align*}
Similarly, $|\mathbb{W}_n(\tau)^\top(\beta-\beta(\tau))|\leq |\mathbb{W}_n(\tau)|2^j/\sqrt{n}$ on $S_{jn}$ by the Cauchy--Schwarz inequality. This can be bounded further by the supremum of $|\mathbb{W}_n(\beta,\tau)|2^j/\sqrt{n}$ over $\Beta\times\Tau$. After slightly decreasing $c$, conclude that the right-hand side of \eqref{eq:minbound} is at most
\begin{equation}\label{eq:emprobound}
\sum_{j\geq N} \prob\Bigl( \bigl|\mathbb{G}_n(m_{\beta,\tau}-m_{\beta(\tau),\tau})\bigr|_{S_{jn}} \geq c 2^{2j}/\sqrt{n} \Bigr) + \sum_{j\geq N} \prob\Bigl(|\mathbb{W}_n(\beta,\tau)|_{\Beta\times\Tau} \geq c 2^{j} \Bigr).
\end{equation}
%The goal is now to bound these two terms via the Markov inequality and tail bounds for empirical processes. 

Consider the supremum inside the first term. For $\delta_{jn} = 2^j/\sqrt{n}$, the supremum over $S_{jn}$ does not exceed the supremum over $\mathcal{M}_{\delta_{jn}}$. Together with Lemma \ref{l:maxineq}\eqref{l:maxineq_i} this yields
\begin{align*}
\ev |\mathbb{G}_n(m_{\beta,\tau}-m_{\beta(\tau),\tau})|^q_{S_{jn}} \lesssim (2^{j}/\sqrt{n})^q.
\end{align*}
The supremum inside the second term in \eqref{eq:emprobound} satisfies $\sup_{n\geq 1}\ev |\mathbb{W}_n(\beta,\tau)|^q_{\Beta\times\Tau} < \infty$ by Lemma \ref{l:maxineq}\eqref{l:maxineq_ii}. Combine these results with the Markov inequality to see that \eqref{eq:emprobound} is bounded by a constant multiple of $\sum_{j\geq N} 2^{-qj}\lesssim 2^{-qN}$, uniformly in $n$. Take $N = \lfloor\log_2 t\rfloor$ for $t\geq 2$ and conclude from the bounds developed so far that
\begin{align*}
\prob \Bigl( \bigl|\sqrt{n}\bigl( \hat{\beta}^*_n(\tau) - \beta(\tau)\bigr)\bigr|_\Tau > 2t \Bigr) \lesssim t^{-q}.
\end{align*}
The case $0<t<2$ can be absorbed into a constant. The preceding display is then valid for all $t>0$. Tonelli's theorem and Lemma 1.2.2 of \citet{vandervaartwellner1996} now give %the representation
\begin{align*}
\ev\bigl|\sqrt{n}\bigl( \hat{\beta}^*_n(\tau) - {\beta}(\tau)\bigr)\bigr|^{p}_\Tau = 2^pp \int_{0}^\infty t^{p-1} \prob \Bigl( \bigl|\sqrt{n}\bigl( \hat{\beta}^*_n(\tau) - \beta(\tau)\bigr)\bigr|_\Tau > 2t \Bigr) \, dt,
\end{align*}
which is finite as long as $p<q$. 

A simpler, nearly identical argument establishes $\ev|\sqrt{n}( \hat{\beta}_n(\tau) - {\beta}(\tau))|^{p}_\Tau < \infty$ uniformly in $n$. The Lo\`eve $c_r$ inequality completes the proof.
%By construction, $\ev( \mathbb{M}^*_{n}(\beta,\tau) \mid \mathcal{D}_n) = \mathbb{M}_{n}(\beta,\tau)$ because $\mathbb{W}_n(\tau)$ has mean $0$ conditional on the data $\mathcal{D}_n$. 
\end{proof}

\begin{proof}[Proof of Lemma \ref{l:jbound}]
Use sub-additivity of the Frobenius norm, Assumptions \ref{as:smooth}\eqref{as:smooth:den} and \eqref{as:smooth:der}, and the mean-value theorem to write $$ |I_{n}(\beta,\tau) - J_{n}(\tau)| \lesssim \frac{1}{n} \sum_{i=1}^{n} \sum_{k=1}^{c_i} \int_0^1 t\, dt\ev \bigl|X_{ik}^\top\bigl( \beta -\beta(\tau)\bigr)\bigr||X_{ik}|^2 \lesssim |\beta-\beta(\tau)|\sup_{i,k} \ev|X_{ik}|^3 . $$ To transform this into a bound on the difference of inverses, note that the eigenvalues of $I_{n}(\beta,\tau)$ and $J_{n}(\tau)$ are bounded away from zero uniformly in $\beta$, $\tau$, and $n$ by Assumption \ref{as:smooth}.  Since $A^{-1} - B^{-1} = B^{-1}(B-A)A^{-1}$ for any two nonsingular matrices $A$ and $B$, conclude from the preceding display, sub-multiplicativity of the Frobenius norm, and Assumption \ref{as:smooth}\eqref{as:smooth:mom} (with $q\geq 3$) that $|I_{n}^{-1}(\beta,\tau) - J_{n}^{-1}(\tau)| \lesssim |I_{n}^{-1}(\beta,\tau)| |J_{n}^{-1}(\tau)| |\beta-\beta(\tau)|$. The right-hand side is finite because $|I_{n}^{-1}(\beta,\tau)| \leq \sqrt{d\lambda_{\min}(I_{n}(\beta,\tau))^{-1}}$ and $|J_{n}^{-1}(\tau)| \leq \sqrt{d\lambda_{\min}(J_{n}(\tau))^{-1}}$ due to the fact that the Frobenius norm is also the 2-Schatten norm.
\end{proof}

\begin{proof}[Proof of Lemma \ref{l:junif}]
The metric space $(\Tau,|\cdot|)$ is totally bounded because $T$ is bounded. As in the proof of Lemma \ref{l:jbound}, $|J_n(\tau)-J_n(\tau')| \lesssim |\beta(\tau)-\beta(\tau')|$. Conclude from \eqref{eq:betalip} that $J_n(\tau)$ is asymptotically uniformly equicontinuous. The pointwise convergence given in Assumption \ref{as:gstat} is therefore uniform by a version of the Arzel\`a-Ascoli theorem \citep[see, e.g.,][Theorem 21.7, p.\ 335]{davidson1994}. The result for the difference of inverses is deduced as in Lemma \ref{l:jbound}.
\end{proof}

%\begin{proof}[Proof of Lemma \ref{l:mconv}] Because $|Z_n(\tau)|_{\Tau}\pto^{\prob} 0$, we must also have $|Z_n(\tau)|_{\Tau}\leadsto 0$. Hence, $\ev \min\{|Z_n(\tau)|_{\Tau},\Delta\}\to 0$ due to boundedness and continuity of $z\mapsto \min\{z,\Delta\}$ with nonnegative $z$ and $\Delta$. By the version of the Fubini theorem given in Lemma 1.2.6 of \citet[p.\ 11]{vandervaartwellner1996}, $\ev \ev^* \min\{|Z_n(\tau)|_{\Tau},\Delta\}\leq \ev \min\{|Z_n(\tau)|_{\Tau},\Delta\}$. Conclude from the Markov inequality that the second term on the right of
%$$ \ev^*|Z_n(\tau)|_{\Tau} = \bigl(\ev^*|Z_n(\tau)|_{\Tau} - \ev^*\min\{|Z_n(\tau)|_{\Tau},\Delta\}\bigr) + \ev^*\min\{|Z_n(\tau)|_{\Tau},\Delta\}.$$
%converges to zero. The first term is bounded above by $\ev^*|Z_n(\tau)|_{\Tau}1\{ |Z_n(\tau)|_{\Tau} > \Delta \}$.
%\end{proof}

\begin{proof}[Proof of Lemma \ref{l:sequi}]
\begin{inparaenum}
\item[\eqref{l:sequi:totbound}] Use the Jensen inequality, the Lo\`eve $c_r$ inequality and $\psi_{\tau'} = \tau' - \tau + \psi_{\tau}$, Assumption \ref{as:smooth}\eqref{as:smooth:mom} and the H\"older inequality with exponents $q/2$ and $q/(q-2)$, the fact that $|1\{a\leq b\} - 1\{a\leq c\}|\leq 1\{|a-b|\leq |b-c|\}$ for $a,b,c\in\mathbb{R}$, and finally Assumption \ref{as:smooth}\eqref{as:smooth:den} and the mean-value theorem to see
\begin{align*}
&\bigl( \mean_n  \ev(\pi_j\circ z_{\beta,\tau}-\pi_j\circ z_{\beta',\tau'})^2 \bigr)^{1/2}\\ 
&\quad\lesssim \mean_n \max_{1\leq k\leq c_i} \Bigl(\ev\pi_j(X_{ik})^2\bigl( \psi_\tau(Y_{ik}-X_{ik}^\top \beta)-\psi_{\tau'}(Y_{ik}-X_{ik}^\top \beta')\bigr)^2 \Bigr)^{1/2}\\
%&\quad\lesssim |\tau-\tau'| + \mean_n \max_{1\leq k\leq c_i} \Bigl(\ev\pi_j(X_{ik})^2\bigl( \psi_\tau(Y_{ik}-X_{ik}^\top \beta)-\psi_{\tau}(Y_{ik}-X_{ik}^\top \beta')\bigr)^2 \Bigr)^{1/2}\\
&\quad\lesssim |\tau-\tau'| + \mean_n \max_{1\leq k\leq c_i} \Bigl(\ev\bigl| \psi_\tau(Y_{ik}-X_{ik}^\top \beta)-\psi_{\tau}(Y_{ik}-X_{ik}^\top \beta')\bigr|^{2q/(q-2)} \Bigr)^{(q-2)/(2q)}\\
&\quad \lesssim |\tau-\tau'| + \mean_n \max_{1\leq k\leq c_i} \bigl(\ev 1\{ |Y_{ik}-X_{ik}^\top \beta| \leq |X_{ik}||\beta-\beta'| \}\bigr)^{(q-2)/(2q)}\\
%&\quad= \mean_n \max_{1\leq k\leq c_i} \Bigl(\ev F_{ik}\bigl(X_{ik}^\top \beta + |X_{ik}||\beta-\beta'|\mid X_{ik}\bigr) - \ev F_{ik}\bigl(X_{ik}^\top \beta - |X_{ik}||\beta-\beta'|\mid X_{ik}\bigr)\Bigr)^{\varepsilon/[2(1+\varepsilon)]}
&\quad\lesssim |\tau-\tau'| + \mean_n \max_{1\leq k\leq c_i} \bigl(|\beta-\beta'|\ev |X_{ik}|\bigr)^{(q-2)/(2q)}.
\end{align*}
Because $|\tau-\tau'| < 1$, we have $|\tau-\tau'|\leq |\tau-\tau'|^{(q-2)/(2q)}$. Conclude from the H\"older inequality that the extreme right-hand side of the display does not exceed a constant multiple of $|((\beta-\beta')^\top,\tau-\tau')|^{(q-2)/(2q)}$. Now take suprema over $n$ and then maxima over $1\leq j\leq d$.

%\item[\eqref{l:sequi:totmetric}] For every $\varepsilon > 0$, at most $(\diam(\Beta\times\Tau)/\varepsilon)^{d+1}$ cubes with side length $\varepsilon$ are needed to cover $\Beta\times\Tau$. Equally many $|\cdot|$-balls of radius $\varepsilon\sqrt{d+1}$ then also cover $\Beta\times\Tau$. Hence, in view of \eqref{l:sequi:totbound}, there exists some absolute constant $\Delta$ such that for every given radius $\delta$, we can pick $\varepsilon = \delta^{2q/(q-2)}/\Delta$ such that the number of $\varrho$-balls of radius $\delta$ needed to cover $\Beta\times\Tau$ is at most $(\diam(\Beta\times\Tau)/\varepsilon)^{d+1}$.

\item[\eqref{l:sequi:asy}] In view of the proof of Lemma \ref{l:maxineq}\eqref{l:maxineq_ii}, the collection of functions $(y,x) \mapsto \pi_j(x) \psi_{\tau}(y-x^\top \beta)$ indexed by $(\beta,\tau)\in\Beta\times\Tau$ is a VC subgraph class. %Hence, this must also be true for the subclass $(y,x^\top)^\top \mapsto \pi_j(x) \psi_{\tau}(y-x^\top \beta(\tau))$ indexed by $\tau\in\Tau$. 
Sums of functions with finite uniform entropy still have finite uniform entropy if the envelope is increased accordingly. An appropriate envelope for $\pi_{j}\circ z_{\beta,\tau}$ is $2\sum_{k=1}^{c_{\max}}|\pi_{j}(X_{ik})|$. Because this envelope is $L_q$-integrable, stochastic equicontinuity follows from \citeauthor{andrews1994}' (\citeyear{andrews1994}, Theorem 1, p.\ 2269) modification of \citeauthor{pollard1991b}'s (\citeyear{pollard1991b}, p. 53) functional central limit theorem; see also \citet[Theorem 1]{kosorok2003}. The process is suitably measurable because $\{ z_{\beta,\tau} : \tau\in\Tau\}$ is almost measurable Suslin by Lemma \ref{l:suslin}.
\end{inparaenum}
\end{proof}

\begin{proof}[Proof of Lemma \ref{l:sigmaequi}]
As noted in the proof of Lemma \ref{l:maxineq}\eqref{l:maxineq_ii}, the collection of functions $(y,x) \mapsto \pi_j(x) \psi_{\tau}(y-x^\top \beta)$ indexed by $\beta\in\Beta$ and $\tau\in\Tau$ is a VC subgraph class. %Hence, this must also be true for the subclass $(y,x) \mapsto \pi_j(x) \psi_{\tau}(y-x^\top \beta(\tau))$ indexed by $\tau\in\Tau$.  
Finite sums of products of such functions need not be VC subgraph, but their uniform entropy numbers behave like those of VC subgraph classes as long as the envelope is increased accordingly; see \citet[pp.\ 157-158]{kosorok2008}. An appropriate envelope for $\pi_{jh}\circ g$ is $$4\sum_{k=1}^{c_{\max}}\sum_{l=1}^{c_{\max}} |\pi_{j}(X_{ik})\pi_h(X_{il})|, \qquad j,h\in\{1,\dots,d\}.$$ 

I now verify that $\mean_n (g-\ev g)$ with $g\in\mathcal{G}$ satisfies the conditions of the uniform law of large numbers of \citet[Theorem 8.2, p.\ 39]{pollard1991b}. By the Jensen inequality, the convergence occurs if it occurs for every entry of the matrix $\mean_n (g-\ev g)$. Hence, suppose for simplicity that $d=1$; otherwise argue separately for each entry. The moment condition (i) on the envelope in \citeauthor{pollard1991b}'s theorem holds immediately by Assumption \ref{as:smooth}. The other condition involves the packing numbers or, equivalently, the covering numbers of $\mathcal{G}$; see \citet[p.\ 98]{vandervaartwellner1996} for definitions and the equivalence result. Because the uniform entropy numbers of $\mathcal{G}$ relative to $L_r$ behave like those of VC subgraph classes for all $r\geq 1$, \citeauthor{pollard1991b}'s condition (ii) is easily satisfied in view of the discussion in \citet[p.\ 125]{vandervaartwellner1996}. The uniform law of large numbers now follows if $\{g : g\in\mathcal{G}\}$ is suitably measurable to support the symmetrization argument used in \citeauthor{pollard1991b}'s proof. A sufficient condition is that $\{g : g\in\mathcal{G}\}$ is almost measurable Suslin, which follows from
\begin{align*}
&\mean_n (\pi_{jh}\circ g_{\beta_1,\tau_1,\beta_2,\tau_2}-\pi_{jh}\circ g_{\beta_1',\tau_1',\beta_2',\tau_2'})^2  \\
&\qquad \lesssim \frac{1}{n} \sum_{i=1}^n\sum_{k=1}^{c_i}\sum_{l=1}^{c_i} \bigl((\tau_1-\tau_1')^2 + 1\{ |Y_{ik} - X_{ik}^\top \beta_1| < |X_{ik}| |\beta_1-\beta_1'| \}\\
&\hspace{10em} +(\tau_2-\tau_2')^2 + 1\{ |Y_{il} - X_{il}^\top \beta_2| < |X_{il}| |\beta_2-\beta_2'| \} \bigr) \pi_{jh}(X_{ik}X_{il}^\top)^2
\end{align*}
and the argument used in the proof of Lemma \ref{l:suslin}, mutatis mutandis.
%Since $\mean_n g_{\beta(\tau),\tau,\beta(\tau'),\tau'}\pto^{\prob} \Sigma(\tau,\tau')$ for all $\tau,\tau'\in\Tau$ by the law of large numbers for independent random variables. $$\sigma_{\Tau}(\tau_1,\tau_2)(\tau_1',\tau_2') := \sigma\bigl(\beta(\tau_1),\tau_1,\beta(\tau_2),\tau_2\bigr)\bigl(\beta(\tau_1'),\tau_1',\beta(\tau_2'),\tau_2'\bigr)$$
%
%Because $\sigma_{\Tau}(\tau_1,\tau_2)(\tau_1',\tau_2')\lesssim |(\tau_1-\tau_1',\tau_2-\tau_2')|^{\varepsilon/[2(1+\varepsilon)]}$ for all $\tau_1,\tau_2,\tau_1',\tau_2'\in\Tau$, the pseudometric space $(\Tau\times\Tau,\sigma_{\Tau})$ is totally bounded. 
%
\end{proof}

\begin{proof}[Proof of Theorem \ref{th:bootse}]
\begin{inparaenum}
\item[\eqref{th:bootse_pw}] Recall that $\{\mathbb{Z}(\tau):\tau\in\Tau\}$ is the $d$-dimensional Gaussian process described in Theorems \ref{th:clt} and \ref{th:bootclt}.  Write $Z_n^*(\tau) := \sqrt{n}( \hat{\beta}^*_n(\tau) - \hat{\beta}_n(\tau))$. Because $\ev^* Z_n^*(\tau)Z_n^*(\tau')^\top$ converges in probability to $\ev \mathbb{Z}(\tau)\mathbb{Z}(\tau')^\top$ if and only if each coordinate converges, assume for simplicity that $d=1$; otherwise, treat each coordinate individually. Then, by Theorem \ref{th:bootclt}, $(Z_n^*(\tau),Z_n^*(\tau')) \leadsto (\mathbb{Z}(\tau),\mathbb{Z}(\tau'))$ in probability in $\mathbb{R}^{2}$ at every $\tau,\tau'\in\Tau$. By arguing along subsequences, conclude from the portmanteau lemma \citep[see, e.g.,][Lemma 2.2, p.\ 6]{vandervaart1998} that in the following we can use the fact that $\ev^* f(Z_n^*(\tau),Z_n^*(\tau'))\pto^{\prob}\ev f(\mathbb{Z}(\tau),\mathbb{Z}(\tau'))$ for every continuous and bounded function $f$.

%\pub{\addlines}

Without loss of generality, assume $Z_n^*(\tau)Z_n^*(\tau')$ is nonnegative; if not, split into positive and negative parts and argue separately. Now, for a given $\Delta > 0$, 
\begin{align*}
|\ev^* Z_n^*(\tau)Z_n^*(\tau') - \ev^* \mathbb{Z}(\tau)\mathbb{Z}(\tau')| &\leq \ev^* Z_n^*(\tau)Z_n^*(\tau') - \ev\min\{ Z_n^*(\tau)Z_n^*(\tau'),\Delta\}\\
&\quad + |\ev^* \min\{ Z_n^*(\tau)Z_n^*(\tau'),\Delta\} - \ev\min\{ \mathbb{Z}(\tau)\mathbb{Z}(\tau'),\Delta\}|\\
&\quad + |\ev\min\{ \mathbb{Z}(\tau)\mathbb{Z}(\tau'),\Delta\} - \ev \mathbb{Z}(\tau)\mathbb{Z}(\tau')|.
\end{align*}
Because $(z,z')\mapsto\min\{zz',\Delta\}$ is continuous and bounded for $zz' \geq 0$, the second term on the right converges to zero in probability as $n\to\infty$. The first term on the right does not exceed $\ev^* Z_n^*(\tau)Z_n^*(\tau') 1\{ Z_n^*(\tau)Z_n^*(\tau') > \Delta \}$. For $\varepsilon > 0$, the law of iterated expectations and the Cauchy-Schwarz inequality give
$$\ev \ev^* Z_n^*(\tau)Z_n^*(\tau') 1\{ Z_n^*(\tau)Z_n^*(\tau') > \Delta \} \leq \sup_{n\geq 1}\sqrt{\ev Z_n^*(\tau)^{2(1+\varepsilon)}\ev Z_n^*(\tau')^{2(1+\varepsilon)}} \Delta^{-\varepsilon}.$$
Note that the expectation on the right operates on both the data and the bootstrap distribution. As long as $2(1+\varepsilon) < q$, the right-hand side is finite by Lemma \ref{l:uniint} and converges to zero as $\Delta \to \infty$. A similar argument applies to the third term. The Markov inequality completes the proof.

\item[\eqref{th:bootse_uni}] Apply mean-value expansions to $\sqrt{n} M_n'(\beta(\tau), \tau) = 0$ to deduce\begin{align*}
\sqrt{n}\bigl(\hat{\beta}^*_n(\tau) - \hat{\beta}_{n}(\tau)\bigr) 
&= \sqrt{n}\bigl(\hat{\beta}^*_n(\tau) - \beta(\tau)\bigr) - \sqrt{n}\bigl(\hat{\beta}_n(\tau) - \beta(\tau)\bigr)\\
&= I_n^{-1}(\hat{\beta}^*_n(\tau),\tau)\sqrt{n} M_n'(\hat{\beta}^*_n(\tau), \tau) - I_n^{-1}(\hat{\beta}_n(\tau),\tau)\sqrt{n} M_n'(\hat{\beta}_n(\tau), \tau).
\end{align*}
After adding and subtracting with $J_n^{-1}(\tau)$, this becomes
\begin{align*}
&J_n^{-1}(\tau)\sqrt{n}\bigl( M_n'(\hat{\beta}^*_n(\tau), \tau) - M_n'(\hat{\beta}_n(\tau), \tau)\bigr) \\
&\quad+ \bigl(I_n^{-1}(\hat{\beta}^*_n(\tau),\tau)-J_n^{-1}(\tau)\bigr)\sqrt{n} M_n'(\hat{\beta}^*_n(\tau), \tau)\\ 
&\quad+\bigl(I_n^{-1}(\hat{\beta}_n(\tau),\tau)-J_n^{-1}(\tau)\bigr)\sqrt{n} M_n'(\hat{\beta}_n(\tau), \tau).
%&\quad+\Bigl(J_n^{-1}(\tau) - J^{-1}(\tau)\Bigr)\Bigl( M_n'(\hat{\beta}^*_n(\tau), \tau) - M_n'(\hat{\beta}_n(\tau), \tau)\Bigr)\\
\end{align*}

Denote the first term by $G_n(\tau)$ and the remaining two terms by $R_n(\tau)$. Then the distance between the bootstrap and population covariance functions can be written as
\begin{align*}
| \hat{V}_n^*(\tau,\tau') - V(\tau,\tau') |_{\Tau\times\Tau} = \bigl| \ev^* \bigl(G_n(\tau) + R_n(\tau)\bigr) \bigl(G_n(\tau') + R_n(\tau')\bigr)^\top - V(\tau,\tau') \bigr|_{\Tau\times\Tau}.
\end{align*}
Monotonicity of the expectation operator and sub-multiplicativity of the Frobenius norm yield $|\ev^* G_n(\tau)R_n(\tau')^\top|_{\Tau\times\Tau} \leq \ev^* |G_n(\tau)R_n(\tau')^\top|_{\Tau\times\Tau} \leq \ev^* |G_n(\tau)|_\Tau|R_n(\tau)|_\Tau$. This remains true when $G_n$ is replaced by $R_n$. Hence, the preceding display is at most
\begin{align*}
\bigl|\ev^*G_n(\tau)G_n(\tau')^\top - V(\tau,\tau') \bigr|_{\Tau\times\Tau} + 2\ev^* |G_n(\tau)|_\Tau|R_n(\tau)|_\Tau + \ev^* |R_n(\tau)|^2_\Tau.
\end{align*}

I will now argue that the second and third term converge to zero in probability. By the Lo\`eve $c_r$ inequality and Lemma \ref{l:jbound}, $\ev^*|R_n(\tau)|^2_\Tau$ is bounded above by $2$ times
\begin{align*}
\ev^*\bigl|\sqrt{n}\bigl(\hat{\beta}^*_n(\tau) - \beta(\tau)\bigr)\bigr|^2_{\Tau} |M_n'(\hat{\beta}^*_n(\tau), \tau)|^2_{\Tau} + \bigl|\sqrt{n}\bigl(\hat{\beta}_n(\tau) - \beta(\tau)\bigr)\bigr|^2_{\Tau} |M_n'(\hat{\beta}_n(\tau), \tau)|^2_{\Tau}.
\end{align*}
Theorem 3.3 of \citet{koenkerbassett1978} yields $$|\mean_n z_{\hat{\beta}^*_n(\tau),\tau} - \mathbb{W}_n(\tau)|\lesssim \max_{i\leq n, k\leq c_i} |X_{ik}| + \sqrt{n}|\mathbb{W}_n(\tau)|$$ uniformly in $\tau\in\Tau$ and therefore
\begin{align*}
|\sqrt{n}M_n'(\hat{\beta}^*_n(\tau), \tau)|^4_{\Tau} 
&= |\mathbb{G}_n z_{\hat{\beta}^*_n(\tau),\tau} -  \sqrt{n}\mean_n z_{\hat{\beta}^*_n(\tau),\tau} + \mathbb{W}_n(\tau)  - \mathbb{W}_n(\tau)\bigr|^4_{\Tau}\\
&\lesssim |\mathbb{G}_n z_{\beta,\tau}|^4_{\Beta\times\Tau} +  \Bigl|n^{-1/2}\max_{i\leq n, k\leq c_i} |X_{ik}| + |\mathbb{W}_n(\tau)|\Bigr|^4_{\Tau}  + |\mathbb{W}_n(\tau)|^4_{\Tau}\\
&\lesssim |\mathbb{G}_n z|^4_{\mathcal{Z}} +  n^{-2}\max_{i\leq n, k\leq c_i} |X_{ik}|^4 + |\mathbb{W}_n(\tau)|^4_{\Tau}.
\end{align*}
The first term on the far right of the display satisfies $\ev|\mathbb{G}_n z|^4_{\mathcal{Z}}< \infty$ uniformly in $n$ by Lemma \ref{l:maxineq}\eqref{l:maxineq_iv}, the second satisfies $n^{-2}\ev\max_{i\leq n, k\leq c_i} |X_{ik}|^4\leq n^{-1}\sup_{i,k}\ev|X_{ik}|^4$ by Pisier's inequality \citep[see, e.g.,][Problem 2.2.8, p.\ 105]{vandervaartwellner1996}, and the third satisfies $\ev|\mathbb{W}_n(\tau)|^4_{\Tau}< \infty$ uniformly in $n$ by Lemma \ref{l:maxineq}\eqref{l:maxineq_ii}. Conclude that $\ev|M_n'(\hat{\beta}^*_n(\tau), \tau)|^4_{\Tau} = O(n^{-2})$. Repeat the argument above with $\mathbb{W}_n(\tau)\equiv 0$ and $\hat{\beta}_n(\tau)$ instead of $\hat{\beta}^*_n(\tau)$ to also establish $\ev|M_n'(\hat{\beta}_n(\tau), \tau)|^4_{\Tau}=O(n^{-2})$. The Markov inequality, Lemma 1.2.6 of \citet[p.\ 11]{vandervaartwellner1996}, the Cauchy-Schwarz inequality, and Lemma \ref{l:uniint} now imply $\ev^* |R_n(\tau)|^2_\Tau\pto^{\prob}0$. 

Further, decompose $J_n(\tau)G_n(\tau)$ into
\begin{align}\label{eq:gdecomp}
\mathbb{G}_n z_{\hat{\beta}^*_n(\tau),\tau} - \mathbb{G}_n z_{\hat{\beta}_n(\tau),\tau} + \sqrt{n}\mean_n z_{\hat{\beta}_n(\tau),\tau} - \bigl(\sqrt{n}\mean_n z_{\hat{\beta}^*_n(\tau),\tau} - \mathbb{W}_n(\tau)\bigr) -\mathbb{W}_n(\tau).
\end{align}
The same arguments as before show $\ev^* |J_n(\tau)G_n(\tau)|^2_\Tau = O_{\prob}(1)$ and thus also $\ev^* |G_n(\tau)|^2_\Tau\leq |J_n^{-1}(\tau)|^2_\Tau \ev^* |J_n(\tau)G_n(\tau)|^2_\Tau = O_{\prob}(1)$ by Lemma \ref{l:junif}. Finally, apply the Cauchy-Schwarz inequality to conclude
\begin{align*}
| \hat{V}_n^*(\tau,\tau') - V(\tau,\tau') |_{\Tau\times\Tau} = \bigl|\ev^*G_n(\tau)G_n(\tau')^\top - V(\tau,\tau') \bigr|_{\Tau\times\Tau} + o_{\prob}(1).
\end{align*}

In view of \eqref{eq:gdecomp} and the arguments above, to show that the right-hand side of the preceding display is within $o_{\prob}(1)$ of $|J_n^{-1}(\tau)\ev^*\mathbb{W}_n(\tau)\mathbb{W}_n(\tau')^\top J_n^{-1}(\tau) - V(\tau,\tau')|_{\Tau\times\Tau}$, it suffices to establish that the $\ev^*$-expectation of $$ \xi_n := \bigl|\mathbb{G}_n z_{\hat{\beta}^*_n(\tau),\tau} - \mathbb{G}_n z_{\hat{\beta}_n(\tau),\tau}\bigr|^2_{\Tau} $$ converges to zero in probability. I will show below that this already follows if the display has a $\prob$-probability limit of zero. Indeed, for any $\eta > 0$, $\prob(\xi_n > \eta)$ does not exceed
\begin{align}\label{eq:equiapprox}
%&\prob \Bigl( \bigl|\mathbb{G}_n z_{\hat{\beta}^*_n(\tau),\tau}- \mathbb{G}_n z_{\hat{\beta}_n(\tau),\tau}\bigr|_{\Tau} > \eta \Bigr)\\
%&\qquad\leq 
\prob \biggl( \sup_{\varrho( z_{\beta,\tau}, z_{\beta',\tau'} )<\delta} |\ep_n z_{\beta,\tau}-\mathbb{G}_n z_{\beta',\tau'} |^2 > \eta \biggr) + \prob \Bigl( \bigl|\varrho\bigl(z_{\hat{\beta}^*_n(\tau),\tau}, z_{\hat{\beta}_n(\tau),\tau} \bigr)\bigr|_{\Tau} \geq \delta\Bigr).
\end{align}
\end{inparaenum}
The limit superior of the first term on the right is arbitrarily small by Lemma \ref{l:sequi}\eqref{l:sequi:asy}. To see that the second term is eventually small as well, use Lemma \ref{l:sequi}\eqref{l:sequi:totbound} to establish $$\ev\bigl|\varrho\bigl(z_{\hat{\beta}^*_n(\tau),\tau}, z_{\hat{\beta}_n(\tau),\tau} \bigr)\bigr|_{\Tau}^{2q/(q-2)}\lesssim \ev\bigl|\hat{\beta}^*_n(\tau)-\hat{\beta}_n(\tau)\bigr|_{\Tau} \leq n^{-1/2} \sup_{n\geq 1}\ev\bigl|\sqrt{n}\bigl(\hat{\beta}^*_n(\tau)-\hat{\beta}_n(\tau)\bigr)\bigr|_{\Tau}.$$
The expression on the right converges to zero by Lemma \ref{l:uniint}. The Markov inequality yields $|\hat{\beta}_n(\tau) - \beta(\tau)|_{\Tau}\pto^{\prob}0$ and the desired result. 

It now follows that $\xi_n\pto^{\prob} 0$ and therefore also $\xi_n\leadsto 0$ by Lemma 1.10.2 of \citet{vandervaartwellner1996}. Hence, $\ev \min\{\xi_n,\Delta\}\to 0$ due to boundedness and continuity of $z\mapsto \min\{z,\Delta\}$ with nonnegative $z$ and $\Delta$. %By the version of the Fubini theorem given in Lemma 1.2.6 of \citet[p.\ 11]{vandervaartwellner1996}, $\ev \ev^* \min\{\xi_n,\Delta\}\leq \ev \min\{\xi_n,\Delta\}$. 
Conclude from the Markov inequality that the second term on the right of
$$ \ev^*\xi_n = (\ev^*\xi_n - \ev^*\min\{\xi_n,\Delta\}) + \ev^*\min\{\xi_n,\Delta\}.$$
converges to zero in probability. The first term is bounded above by $\ev^*\xi_n1\{\xi_n > \Delta \}$. For a small enough $\varepsilon > 0$, the right-hand side of $\ev \xi_n1\{\xi_n > \Delta \} \leq \ev \xi_n^{1+\varepsilon}\Delta^{-\varepsilon}$ is finite by Lemma \ref{l:maxineq}\eqref{l:maxineq_iv} uniformly in $n$ and converges to zero as $\Delta \to \infty$. Deduce from the Markov inequality that the preceding display has a probability limit of zero. Combine the results above and Lemma \ref{l:junif} to obtain
\begin{align*}
| \hat{V}_n^*(\tau,\tau') - V(\tau,\tau') |_{\Tau\times\Tau} = \bigl|J^{-1}(\tau)\ev^*\mathbb{W}_n(\tau)\mathbb{W}_n(\tau')^\top J^{-1}(\tau') - V(\tau,\tau')\bigr|_{\Tau\times\Tau} + o_{\prob}(1)
\end{align*}

Because $V(\tau,\tau') = J^{-1}(\tau)\Sigma(\tau,\tau')J^{-1}(\tau')$, it now suffices to show that $\ev^*\mathbb{W}_n(\tau)\mathbb{W}_n(\tau')^\top$ and $\Sigma(\tau,\tau')$ are uniformly close in probablity as $n\to\infty$. From the definition of the bootstrap weights we have $\ev^*\mathbb{W}_n(\tau)\mathbb{W}_n(\tau')^\top = \mean_n g_{\hat{\beta}_n(\tau),\tau,\hat{\beta}_n(\tau'),\tau'}.$ Decompose the difference between $\mean_n g_{\hat{\beta}_n(\tau),\tau,\hat{\beta}_n(\tau'),\tau'}$ and $\Sigma_n(\tau,\tau') = \mean_n \ev g_{\beta(\tau),\tau,\beta(\tau'),\tau'}$ into
\begin{align*}
\mean_n (g_{\hat{\beta}_n(\tau),\tau,\hat{\beta}_n(\tau'),\tau'} - \ev g_{\hat{\beta}_n(\tau),\tau,\hat{\beta}_n(\tau'),\tau'} ) &+ \mean_n (g_{\beta(\tau),\tau,\beta(\tau'),\tau'} - \ev g_{\beta(\tau),\tau,\beta(\tau'),\tau'} )\\
& \qquad- \mean_n\ev ( g_{\hat{\beta}_n(\tau),\tau,\hat{\beta}_n(\tau'),\tau'} - g_{\beta(\tau),\tau,\beta(\tau'),\tau'} ).
\end{align*}
The first two terms converge to zero by the uniform law of large numbers in Lemma \ref{l:sigmaequi}. An argument similar to the one given in Lemma \ref{l:sequi}\eqref{l:sequi:totbound} yields
\begin{align*}
&\sup_{\tau,\tau'}\sup_{|\beta-\beta(\tau)|<\delta}\sup_{|\beta'-\beta(\tau')|< \delta} | \mean_n\ev(g_{\beta,\tau,\beta',\tau'} - g_{\beta(\tau),\tau,\beta(\tau'),\tau'}) |\\
&\qquad \lesssim  \frac{1}{n} \sum_{i=1}^n\sum_{k=1}^{c_i}\sum_{l=1}^{c_i} \sup_{\tau,\tau'}\ev\bigl(1\{ |Y_{ik} - X_{ik}^\top \beta(\tau)| < |X_{ik}| \delta \}\\ 
&\hspace{13em}+ 1\{ |Y_{il} - X_{il}^\top \beta(\tau')| < |X_{il}|\delta  \} \bigr) |X_{ik}^\top X_{il}| \lesssim \delta^{(q-2)/q}.
\end{align*}
%In view of that lemma, stochastic equicontinuity of $g_{\beta_1,\tau_1,\beta_2,\tau_2}\mapsto \mean_n g_{\beta_1,\tau_1,\beta_2,\tau_2}$ and 
%$$\bigl|\sigma\bigl(\hat{\beta}_n(\tau),\tau,\hat{\beta}_n(\tau'),\tau'\bigr)\bigl(\beta(\tau),\tau,\beta(\tau'),\tau'\bigr)\bigr|_{\Tau\times\Tau}^{2q/(q-4)} \lesssim \bigl|\hat{\beta}_n(\tau) - \beta(\tau)\bigr|_{\Tau}\overset{\prob}{\to}0$$ 
Because $\prob(|\hat{\beta}_n(\tau) - \beta(\tau)|_\Tau\geq \delta)\to 0$ for every $\delta > 0$, let first $n\to\infty$ and then $\delta \to 0$ to obtain $\mean_n g_{\hat{\beta}_n(\tau),\tau,\hat{\beta}_n(\tau'),\tau'} = \Sigma_n(\tau,\tau') + o_{\prob}(1)$ uniformly in $\Tau\times\Tau$.

The proof is complete if $|\Sigma_n(\tau,\tau') - \Sigma(\tau,\tau')|_{\Tau\times\Tau}$ converges to zero as $n\to\infty$. To this end, note that the metric space $(\Tau\times\Tau,|\cdot|)$ is totally bounded because $T$ is bounded. A straightforward computation using Assumption \ref{as:smooth} and \eqref{eq:betalip} gives $|\Sigma_n(\tau_1,\tau_2)-\Sigma_n(\tau_1',\tau_2')| \lesssim |(\tau_1-\tau_1', \tau_2-\tau_2')|^{1/2}$. Thus $\Sigma_n(\tau)$ is asymptotically uniformly equicontinuous. The pointwise convergence given in Assumption \ref{as:gstat} is therefore uniform by the Arzel\`a-Ascoli theorem \citep[Theorem 21.7, p.\ 335]{davidson1994}. 
\end{proof}

\begin{proof}[Proof of Theorem \ref{th:clt}] Because $\sqrt{n} M_n'(\beta(\tau), \tau) = 0$, a mean-value expansion gives
\begin{align*}
\sqrt{n} M_n'(\hat{\beta}_n(\tau), \tau) = I_n(\hat{\beta}_n(\tau),\tau)\sqrt{n}\bigl(\hat{\beta}_n(\tau) - \beta(\tau)\bigr).
\end{align*}
In view of Lemma \ref{l:jbound}, rearrange to write $\sqrt{n}(\hat{\beta}_n(\tau) - \beta(\tau))$ as 
\begin{align*}
I_n^{-1}(\hat{\beta}_n(\tau),\tau)\Bigl( \ep_n \bigl(z_{\hat{\beta}_n(\tau), \tau} - z_{\beta(\tau), \tau} \bigr) - \sqrt{n}\mean_n z_{\hat{\beta}_n(\tau), \tau}  + \ep_n z_{\beta(\tau), \tau} \Bigr).
\end{align*}
Computations similar to, but substantially simpler than the ones given in the proof of Theorem \ref{th:bootse}\eqref{th:bootse_uni} show that the preceding display is, uniformly in $\tau\in\Tau$, $o_{\prob}(1)$ away from
$J^{-1}(\tau) \ep_n z_{\beta(\tau), \tau}$. For every finite set of points $\tau_1, \tau_2,\hdots\in\Tau$, the convergence of the marginal vectors $$(\ep_n z_{\beta(\tau_1), \tau_1}^\top, \ep_n z_{\beta(\tau_2), \tau_2}^\top,\dots)^\top$$ follows from the Lindeberg-Feller central limit theorem and Assumptions \ref{as:data}-\ref{as:gstat}. 

Take $\varrho_\Tau(\tau,\tau') := \varrho(z_{\beta(\tau),\tau}, z_{\beta(\tau'),\tau'})$. To see that $(\Tau,\varrho_\Tau)$ is a totally bounded pseudometric space, note that for every $\varepsilon > 0$ the number of intervals of length $\varepsilon$ needed to cover $\Tau$ does not exceed $\diam(\Tau)/\varepsilon$. Hence, in view of \ref{l:sequi}\eqref{l:sequi:totbound} and \eqref{eq:betalip}, there exists some absolute constant $\Delta$ such that for every given radius $\delta$, we can pick $\varepsilon = \delta^{2q/(q-2)}/\Delta$. The number of $\varrho_\Tau$-balls of radius $\delta$ needed to cover $\Tau$ is then at most $\diam(\Tau)\Delta/\delta^{2q/(q-2)}$. 

The $\varrho_\Tau$-stochastic equicontinuity of $z_{\beta(\tau), \tau}\mapsto\ep_n z_{\beta(\tau), \tau}$ is implied by the $\varrho$-stochastic equicontinuity of $z_{\beta, \tau}\mapsto\ep_n z_{\beta, \tau}$. The theorem now follows from Theorem 10.2  of \citet[p. 51]{pollard1991b} and the continuous mapping theorem in metric spaces \citep[see, e.g.][Theorem 18.11, p. 259]{vandervaart1998}.
\end{proof}

\begin{proof}[Proof of Theorem \ref{th:bootclt}] 
Essentially the same computations as in the proof of Theorem \ref{th:bootse}\eqref{th:bootse_uni} (although without the requirement that $q>4$) yield
$$\sqrt{n}\bigl(\hat{\beta}^*_n(\tau) - \hat{\beta}_{n}(\tau)\bigr) = -J^{-1}(\tau) \mathbb{W}_n(\tau) + o_{\prob}(1).$$
Here the minus sign on the right is a consequence of the definition of the perturbed QR problem in \eqref{eq:bootmin} and will of course have no impact on the asymptotic distribution.

Theorem 3 of \citet{kosorok2003} implies that the $\varrho$-stochastic equicontinuity of $\ep_n z_{\beta,\tau}$ carries over to its multiplier process $\mathbb{W}_n(\beta,\tau)$ if the conditions of Theorem 10.6 of \citet[p. 53]{pollard1991b} hold. In fact, inspection of the proof shows that if only $\varrho$-stochastic equicontinuity is needed, then it suffices already to verify \citeauthor{pollard1991b}'s conditions (i) and (iii)-(v). (This observation was also made by \citet[p.\ 2284]{andrews1994} for \citeauthor{pollard1991b}'s Theorem 10.6.) Further, as pointed out by \citeauthor{andrews1994}, these conditions can be verified for any pseudometric that satisfies \citeauthor{pollard1991b}'s condition (v) because \citeauthor{pollard1991b}'s total boundedness result is not needed. The pseudometric $\varrho$ used here has property (v) by construction.

Following \citet[p.\ 2284]{andrews1994}, the ``manageability'' condition (i) is implied by the finite uniform entropy property that was established in the proof of Lemma \ref{l:sequi}\eqref{l:sequi:asy}. The remaining moment conditions (iii) and (iv) are implied by Assumption \ref{as:smooth}. By a standard stochastic equicontinuity argument as in \eqref{eq:equiapprox} we can therefore restate, uniformly in $\tau\in\Tau$, the preceding display as 
$$\sqrt{n}\bigl(\hat{\beta}^*_n(\tau) - \hat{\beta}_{n}(\tau)\bigr) = -J^{-1}(\tau)\mathbb{W}_n(\beta(\tau),\tau)  + o_{\prob}(1).$$

Finally, apply the multiplier central limit theorem of \citet[Theorem 3]{kosorok2003} to $\mathbb{W}_n(\beta(\tau),\tau)$. His conclusions again do not depend on the the specific choice of pseudometric because they are used for a total boundedness result that is not required here. (Recall from the proof of Theorem \ref{th:clt} that $(\Tau,\varrho_\Tau)$ is totally bounded.) The only condition that still needs to be verified is that $\{ z_{\beta(\tau),\tau} : \tau\in\Tau\}$ is almost measurable Suslin, which holds by Lemma \ref{l:suslin}. The continuous mapping theorem in metric spaces completes the proof.
\end{proof}

%\pub{\addlines}

\begin{proof}[Proof of Theorem \ref{c:bootse}]
Let $\Omega_\infty(\tau) = R(\tau)V(\tau,\tau)R(\tau)^\top$. Because $a^\top \Sigma(\tau,\tau) a$ is bounded below by zero and $J(\tau)$ and $R(\tau)$ have full rank, Assumption \ref{as:posdef} yields $\inf_{\tau\in\Tau} a^\top \Omega_\infty(\tau) a > 0 $ for all non-zero $a$. Conclude from a singular value decomposition that $\inf_{\tau\in\Tau} \lambda_{\min}(\Omega_\infty(\tau)) > 0$ and therefore also $\inf_{\tau\in\Tau} \lambda_{\min}(\smash{\Omega^{1/2}_\infty}(\tau)) > 0$.  Since eigenvalues are Lipschitz continuous on the space of symmetric matrices, apply Theorem \ref{th:bootse} to deduce $$ \Bigl| \inf_{\tau\in\Tau} \lambda_{\min}(\hat{\Omega}_n^*(\tau)) - \inf_{\tau\in\Tau} \lambda_{\min}(\Omega_\infty(\tau))\Bigr|\leq \sup_{\tau\in\Tau} \bigl| \lambda_{\min}(\hat{\Omega}_n^*(\tau)) - \lambda_{\min}(\Omega_\infty(\tau)) \bigr| \pto^{\prob}0.$$ Hence, $\inf_{\tau\in\Tau} \lambda_{\min}(\hat{\Omega}^{*1/2}_n(\tau)) > 0$ with probability approaching one as $n\to\infty$. On that event, we can write $|\hat{\Omega}_n^{*-1/2}(\tau)|^2_\Tau\leq d/\inf_{\tau\in\Tau} \lambda_{\min}(\hat{\Omega}^{*1/2}_n(\tau)) \pto^{\prob} d/\inf_{\tau\in\Tau} \lambda_{\min}({\Omega}^{1/2}_\infty(\tau))$. But then $|\hat{\Omega}_n^{*-1/2}(\tau)|_\Tau$ must be bounded in probability and the right-hand side of
\begin{align*}
|\hat{\Omega}_n^{*-1/2}(\tau) - \Omega_\infty^{-1/2}(\tau)|_\Tau \leq |\hat{\Omega}_n^{*-1/2}(\tau)|_\Tau |\Omega_\infty^{-1/2}(\tau)|_\Tau |\hat{\Omega}^{*1/2}_n(\tau) - \Omega^{1/2}_\infty(\tau)|_\Tau
\end{align*}
converges to zero in probability if $|\hat{\Omega}^{*1/2}_n(\tau) - \Omega^{1/2}_\infty(\tau)|_\Tau$ does. By Proposition 3.2 of \citet{vonhemmeando1980} \citep[see also][Theorem 6.2, p.\ 135]{higham2008} this difference satisfies $$|\hat{\Omega}^{*1/2}_n(\tau) - \Omega^{1/2}_\infty(\tau)|_\Tau \leq \frac{|\hat{\Omega}_n^*(\tau) - \Omega_\infty(\tau)|_\Tau }{\inf_{\tau\in\Tau} \lambda_{\min}(\hat{\Omega}^{*1/2}_n(\tau)) + \inf_{\tau\in\Tau} \lambda_{\min}(\Omega^{1/2}_\infty(\tau))}$$ and therefore has a probability limit of zero by Theorem \ref{th:bootse}.

Let $\|\cdot\|$ be any norm on $\mathbb{R}^d$ and abbreviate $\sup_{\tau\in\Tau}\|\cdot\|$ by $\|\cdot\|_\Tau$. Apply first Theorem \ref{th:bootclt} and the continuous mapping theorem unconditionally, then Proposition 10.7 of \citet[pp.\ 189-190]{kosorok2008}, Theorem \ref{th:bootclt} conditional on the data, and Lipschitz continuity to obtain
\begin{align*}%\label{eq:bootweak}
K^*_n(\hat{\Omega}_n^*,T) = \bigl\Vert\Omega_\infty^{-1/2}(\tau)R(\tau)\sqrt{n}\bigl(\hat{\beta}^*_n(\tau)-\hat{\beta}_n(\tau)\bigr)\bigr\Vert_\Tau + o_{\prob}(1) \leadsto \Vert \Omega_\infty^{-1/2}(\tau)R(\tau)\mathbb{Z}(\tau) \Vert_\Tau
\end{align*}
in probability conditional on the data. Here the bootstrap convergence occurs with respect to the bounded Lipschitz metric as in Theorem \ref{th:bootclt}, uniformly on $\mathrm{BL}_1(\mathbb{R})$.
Similarly, under the null hypothesis, rewrite $\mathit{K}_n(\hat{\Omega}_n^*,T)$ and then apply Theorem \ref{th:clt} and the continuous mapping theorem to establish
%$$\sup_{\tau\in\Tau} n\Bigl(R(\tau)\bigl(\hat{\beta}_n(\tau)-{\beta}(\tau)\bigr)\Bigr)^\top \hat{\Omega}_n^*^{-1}(\tau)\Bigl(R(\tau)\bigl(\hat{\beta}_n(\tau)-{\beta}(\tau)\bigr)\Bigr).$$ 
%By Theorem \ref{th:clt} and the continuous mapping theorem, this expression converges in distribution to 
\begin{align*}
K_n(\hat{\Omega}_n^*,T) = \bigl\Vert\hat{\Omega}_n^{*-1/2}(\tau)R(\tau)\sqrt{n}\bigl(\hat{\beta}_n(\tau)-{\beta}(\tau)\bigr)\bigr\Vert_\Tau \leadsto \Vert\Omega_\infty^{-1/2}(\tau)R(\tau)\mathbb{Z}(\tau)\Vert_\Tau =: K(\Omega_\infty,T).
\end{align*}
%$$\sup_{\tau\in\Tau} \bigl(R(\tau)\mathbb{Z}(\tau)\bigr)^\top \Omega_\infty^{-1}(\tau)R(\tau)\mathbb{Z}(\tau).$$

For $x\in\ell^\infty(\Tau)$, the map $x\mapsto \Vert x\Vert_\Tau$ constitutes a continuous, convex functional. Theorem 11.1 of \citet[p.\ 75]{davydovetal1998} then implies that the distribution function of $K(\Omega_\infty,T)$ is continuous and strictly increasing on $(q_0,\infty)$, where $q_0 = \inf\{ q : \prob(K(\Omega_\infty,T) \leq q) > 0 \}$. Because $\mathbb{Z}(\tau)$ is a Gaussian process with non-degenerate variance, we have $\prob(K(\Omega_\infty,T) \leq 0) \leq \prob(K(\Omega_\infty,\{\tau\}) \leq 0)=0$ for arbitrary $\tau\in\Tau$. Furthermore, because all norms on $\mathbb{R}^d$ are equivalent, there exists a $c>0$ such that
\begin{align*}
\prob(K(\Omega_\infty,T) < q) \geq \prob\bigl( |\pi_j \circ \Omega_\infty^{-1/2}(\tau)R(\tau)\mathbb{Z}(\tau) |_{\{1,\dots, d\}\times \Tau} < q/c\bigr).
\end{align*}
%(For the maximum norm, dividing by $d$ is unnecessary.) 
The supremum on the right is the supremum of the absolute value of mean-zero Gaussian variables. By \citet[Corollary]{lifshits1982} and \citet[Theorem 11.1]{davydovetal1998}, this supremum has a continuous, strictly increasing distribution function on $(0,\infty)$. The right-hand side of the preceding display is therefore strictly positive for every $q>0$. Hence, zero is in the support of $K(\Omega_\infty,T)$ and we must have $q_0 = 0$. Conclude that the distribution function of $K(\Omega_\infty,T)$ is in fact continuous and strictly increasing on $(0,\infty)$.

I will now argue that the quantiles of $K_n(\hat{\Omega}_n^*,T)$ and $K(\Omega_\infty,T)$ are eventually close in probability.
Define the maps 
\begin{align*}
h_n(q) &= \rho_{1-\alpha}\bigl(K^*_n(\hat{\Omega}_n^*,T) - q\bigr) - \rho_{1-\alpha}\bigl(K^*_n(\hat{\Omega}_n^*,T)\bigr)\quad\text{and}\\
h(q) &= \rho_{1-\alpha}\bigl(K(\Omega_\infty,T) - q\bigr) - \rho_{1-\alpha}\bigl(K(\Omega_\infty,T)\bigr).
\end{align*}
Despite the fact that $h_n$ may not be a measurable function of the bootstrap weights, $q\mapsto\ev^* h_n(q)$ and $q\mapsto\ev h(q)$ are clearly convex. Furthermore, $\ev h$ takes on its unique minimum at $q_{1-\alpha}(\Omega_\infty,\Tau) = \argmin_{q\in\mathbb{R}} \ev h(q)$ by the properties of the distribution function established above. In addition, both $h_n$ and $h$ are Lipschitz continuous. Proposition 10.7 of \citet{kosorok2008} and the definition of conditional weak convergence in probability then yield $|\ev^* h_n(q) - \ev h(q)|_Q\pto^{\prob}0$ for every compact set $Q\subset\mathbb{R}$. Because $q_{n,1-\alpha}(\hat{\Omega}_n^*,\Tau) = \argmin_{q\in\mathbb{R}} \ev^* h_n(q)$ and $\ev h$ has a unique minimum, Lemma 2 of \citet{hjortpollard1993} gives 
%
%Conclude from continuity of the distribution function and the properties of quantile functions \citep[Lemma 21.2, p.\ 305]{vandervaart1998} that 
$q_{n,1-\alpha}(\hat{\Omega}_n^*,\Tau)\pto^{\prob} q_{1-\alpha}(\Omega_\infty,\Tau)$. Thus, 
$$ \mathit{K}_n(\hat{\Omega}_n^*,T) - q_{n,1-\alpha}(\hat{\Omega}_n^*,\Tau) \leadsto \Vert \Omega_\infty^{-1/2}(\tau)R(\tau)\mathbb{Z}(\tau) \Vert_\Tau - q_{1-\alpha}(\Omega_\infty,\Tau),$$
where the distribution of the right-hand side is again continuous. The first result now follows because by the definition of weak convergence
$$ \prob\bigl( \mathit{K}_n(\hat{\Omega}_n^*,T) > q_{n,1-\alpha}(\hat{\Omega}_n^*,\Tau) \bigr) \to \prob\bigl( \| \Omega_\infty^{-1/2}(\tau)R(\tau)\mathbb{Z}(\tau) \|_\Tau > q_{1-\alpha}(\Omega_\infty,\Tau) \bigr) = \alpha.$$

%\pub{\addlines}

Under the alternative, use Theorem \ref{th:clt} and then the fact that $\Omega_\infty$ and $R$ have full rank to find an $\varepsilon > 0$ such that $$\mathit{K}_n(\hat{\Omega}_n^*,T)/n \pto^{\prob} \bigl\|\Omega_\infty^{-1/2}(\tau)R(\tau)\bigl(\beta(\tau) - r(\tau)\bigr)\bigr\|_\Tau =: K_\infty > \varepsilon.$$
Hence, the first term on the right-hand side of 
\begin{align*}
&\prob\bigl(\mathit{K}_n(\hat{\Omega}_n^*,\Tau) \leq q_{n,1-\alpha}(\hat{\Omega}_n^*,\Tau)\bigr)\\ 
&\qquad\leq \prob\bigl(|\mathit{K}_n(\hat{\Omega}_n^*,\Tau)/n - K_\infty| \geq \varepsilon\bigr) + \prob\bigl(q_{n,1-\alpha}(\hat{\Omega}_n^*,\Tau) > n(K_\infty - \varepsilon)\bigr)
\end{align*}
converges to zero. To see that the second term converges to zero as well, note that $q_{n,1-\alpha}(\hat{\Omega}_n^*,\Tau)$ is bounded in probability by the arguments given in the previous paragraph. Hence, $q_{n,1-\alpha}(\hat{\Omega}_n^*,\Tau)/n \pto^{\prob}0$ and the desired conclusion follows because $K_\infty - \varepsilon > 0$.
%it is, by the properties of quantile functions \citep[Lemma 21.1(i), p.\ 304]{vandervaart1998}, equal to
%$$\prob\Bigl(\prob^*\bigl(K_n^*(\hat{\Omega}_n^*,\Tau)/n >  c_\infty - \varepsilon\bigr) > \alpha\Bigr) \leq \alpha^{-1} \prob\bigl(K_n^*(\hat{\Omega}_n^*,\Tau)/n >  c_\infty - \varepsilon\bigr).$$  Here the inequality follows from the Markov inequality and the Fubini theorem given in Lemma 1.2.6 of \citet[p.\ 11]{vandervaartwellner1996}.
%Because $c_\infty > \varepsilon$, it is therefore enough to show that $K_n^*(\hat{\Omega}_n^*,\Tau)/n = o_{\prob}(1)$. Since $|\hat{\beta}^*_n(\tau)-\hat{\beta}_n(\tau)|_\Tau = o_{\prob}(1)$ by Theorem \ref{th:bootclt}, submultiplicativity of the Frobenius norm immediately implies the desired result. 
All of the above remains valid with $I_d$ in place of $\hat{\Omega}_n^*$ and $\Omega_\infty$.
\end{proof}

\begin{proof}[Proof of Corollary \ref{c:bootband}]
Denote the diagonal matrix of a square matrix $A$ by $\diag A$. For every given $\Delta$, we can find a matrix $R$ such that 
$$ K_n^*(\hat{V}_n^*,\Tau,\Delta) = \sup_{\tau\in\Tau} \bigl|\diag(\hat{V}_n^*(\tau, \tau))^{-1/2}R\sqrt{n}\bigl(\hat{\beta}^*_n(\tau)-\hat{\beta}_n(\tau)\bigr)\bigr|_{\max}.$$
This statistic is of the same form as the statistic $K^*_n(\Omega,\Tau)$ used in Theorem \ref{th:bootse} with $\Omega = \diag \hat{V}_n^*$. If we also define
$$ K_n(\hat{V}_n^*,\Tau,\Delta) = \sup_{\tau\in\Tau} \bigl|\diag(\hat{V}_n^*(\tau, \tau))^{-1/2}R\sqrt{n}\bigl(\hat{\beta}_n(\tau)-\beta(\tau)\bigr)\bigr|_{\max}, $$
then we can view $K_n(\hat{V}_n^*,\Tau,\Delta)$ as $K_n(\diag\hat{V}_n^*,\Tau)$ under the null hypothesis and the event inside the displayed probability in the corollary is equivalent to $\{K_n(\hat{V}_n^*,\Tau,\Delta)\leq q_{n,1-\alpha}(\hat{V}_n^*,\Tau,\Delta)\}$. Hence, $\prob(K_n(\hat{V}_n^*,\Tau,\Delta)\leq q_{n,1-\alpha}(\hat{V}_n^*,\Tau,\Delta))\to 1-\alpha$ follows from Theorem \ref{c:bootse} if its proof also applies to the weight matrix $\diag \hat{V}_n^*$. Because $\inf_{\tau\in\Tau} a^\top V(\tau, \tau) a > 0 $ for all non-zero $a$ implies $\inf_{\tau\in\Tau} a^\top \diag V(\tau, \tau) a > 0 $ for all non-zero $a$, I only have to show that $|\diag \hat{V}_n^*(\tau,\tau) - \diag V(\tau, \tau)|_\Tau \pto^{\prob} 0$. But this is implied by $|\hat{V}_n^*(\tau, \tau) - V(\tau, \tau)|_\Tau \pto^{\prob} 0$ because for any square matrices $A$, $B$ of identical dimension, $|\diag A - \diag B| = |\diag(A-B)|\leq |A-B|$.
\end{proof}

\begin{proof}[Proof of Corollary \ref{c:bootgrid}]
Inspection of the proof of Theorem \ref{c:bootse} reveals that the desired conclusion holds if (i) $K_n(\hat{\Omega}_n^*,\Tau_n)$ converges in distribution to $K(\Omega_\infty,\Tau)$ and (ii) $K^*_n(\hat{\Omega}_n^*,T_n)$ and $K^*_n(\hat{\Omega}_n^*,T)$ are $o_{\prob}(1)$ away from one another. For (i), note that $\Tau_n$ is a subset of the pseudometric space $(\Tau,\varrho_\Tau)$. %For an arbitrary $\tau\in\Tau$, construct a sequence $\tau_{j_n}$ as follows: for every $n$, pick $\tau_{j_n}\in\Tau_n$ such that $\tau_{j_n}\leq \tau < \tau_{j_n+1}$. But then $0\leq \tau-\tau_{j_n} < \gamma_n$. 
Because $\tau_n\to\tau$, conclude from Lemma \ref{l:sequi}\eqref{l:sequi:totbound} that $\varrho_\Tau(\tau_n,\tau) \to 0$. %Hence, every $\tau\in\Tau$ is the limit of a sequence $\tau_{j_n}\in\Tau_n$. 
It now follows from Exercise 7.5.5 of \citet[p.\ 125, with his $\Tau_0$ equal to $\Tau$]{kosorok2008} and the extended continuous mapping theorem that $K_n(\hat{\Omega}_n^*,\Tau_n)\leadsto K(\Omega_\infty,\Tau)$.

For (ii), we obtain $K^*_n(\hat{\Omega}_n^*,T_n)\leadsto K(\Omega_\infty,\Tau)$ unconditionally using the same argument. Therefore, again by the extended continuous mapping theorem, we in fact have the joint convergence $$\bigl(K^*_n(\hat{\Omega}_n^*,T_n),K^*_n(\hat{\Omega}_n^*,T)\bigr) \leadsto \bigl(K(\Omega_\infty,\Tau), K(\Omega_\infty,\Tau)\bigr)$$ unconditionally, which immediately gives $ K^*_n(\hat{\Omega}_n^*,T_n) - K^*_n(\hat{\Omega}_n^*,T) \pto^{\prob} 0$.
\end{proof}

\phantomsection
\addcontentsline{toc}{section}{References}
\pub{\bibliography{qspec}}

\begin{thebibliography}{}

\bibitem[\protect\citeauthoryear{Alexander}{Alexander}{1985}]{alexander1985}
Alexander, K.~M. (1985).
\newblock Rates of growth for weighted empirical processes.
\newblock In {\em Proceedings of the {B}erkeley Conference in Honor of {J}erzy
  {N}eyman and {J}ack {K}iefer}, Volume~2, pp.\  475--493. University of
  California Press, Berkeley.

\bibitem[\protect\citeauthoryear{Andrews}{Andrews}{1994}]{andrews1994}
Andrews, D. W.~K. (1994).
\newblock Empirical process methods in econometrics.
\newblock In R.~F. Engle and D.~L. McFadden (Eds.), {\em Handbook of
  Econometrics}, Volume~IV, Chapter~37, pp.\  2248--2294. Elsevier.

\bibitem[\protect\citeauthoryear{Andrews and Buchinsky}{Andrews and
  Buchinsky}{2000}]{andrewsbuchinsky2000}
Andrews, D. W.~K. and M.~Buchinsky (2000).
\newblock A three-step method for choosing the number of bootstrap repetitions.
\newblock {\em Econometrica\/}~{\em 68}, 21--51.

\bibitem[\protect\citeauthoryear{Angrist, Chernozhukov, and
  Fern{\'a}ndez-Val}{Angrist et~al.}{2006}]{angristetal2006}
Angrist, J., V.~Chernozhukov, and I.~Fern{\'a}ndez-Val (2006).
\newblock Quantile regression under misspecification, with an application to
  the {U.S.} wage structure.
\newblock {\em Econometrica\/}~{\em 74}, 539--563.

\bibitem[\protect\citeauthoryear{Belloni, Chernozhukov, and
  Fern{\'a}ndez-Val}{Belloni et~al.}{2011}]{bellonietal2011}
Belloni, A., V.~Chernozhukov, and I.~Fern{\'a}ndez-Val (2011).
\newblock Conditional quantile processes based on series or many regressors.
\newblock Preprint, \texttt{arXiv:1105.6154}.

\bibitem[\protect\citeauthoryear{Bertrand, Duflo, and Mullainathan}{Bertrand
  et~al.}{2004}]{bertrandetal2004}
Bertrand, M., E.~Duflo, and S.~Mullainathan (2004).
\newblock How much should we trust differences-in-differences estimates?
\newblock {\em Quarterly Journal of Economics\/}~{\em 119}, 249--275.

\bibitem[\protect\citeauthoryear{Bester, Conley, and Hansen}{Bester
  et~al.}{2011}]{besteretal2011}
Bester, C.~A., T.~G. Conley, and C.~B. Hansen (2011).
\newblock Inference with dependent data using cluster covariance estimators.
\newblock {\em Journal of Econometrics\/}~{\em 165}, 137--151.

\bibitem[\protect\citeauthoryear{Cameron, Gelbach, and Miller}{Cameron
  et~al.}{2008}]{cameronetal2008}
Cameron, A.~C., J.~B. Gelbach, and D.~L. Miller (2008).
\newblock Bootstrap-based improvements for inference with clustered errors.
\newblock {\em Review of Economics and Statistics\/}~{\em 90}, 414--427.

\bibitem[\protect\citeauthoryear{Cameron and Miller}{Cameron and
  Miller}{2015}]{cameronmiller2014}
Cameron, A.~C. and D.~L. Miller (2015).
\newblock A practitoner's guide to cluster robust inference.
\newblock {\em Journal of Human Resources\/}, forthcoming.

\bibitem[\protect\citeauthoryear{Chen, Wei, and Parzen}{Chen
  et~al.}{2003}]{chenetal2003}
Chen, L., L.-J. Wei, and M.~I. Parzen (2003).
\newblock Quantile regression for correlated observations.
\newblock In {\em Proceedings of the Second Seattle Symposium in Biostatistics:
  Analysis of Correlated Data}. Lecture Notes in Statistics. Springer, New
  York.

\bibitem[\protect\citeauthoryear{Chetverikov, Larsen, and Palmer}{Chetverikov
  et~al.}{2013}]{chetverikovetal2013}
Chetverikov, D., B.~Larsen, and C.~Palmer (2013).
\newblock {IV} quantile regression for group-level treatments.
\newblock Unpublished manuscript.

\bibitem[\protect\citeauthoryear{Davidson}{Davidson}{1994}]{davidson1994}
Davidson, J. (1994).
\newblock {\em Stochastic Limit Theory}.
\newblock Oxford University Press, Oxford.

\bibitem[\protect\citeauthoryear{Davidson}{Davidson}{2012}]{davidson2012}
Davidson, R. (2012).
\newblock Statistical inference in the presence of heavy tails.
\newblock {\em Econometrics Journal\/}~{\em 15}, C31--C53.

\bibitem[\protect\citeauthoryear{Davidson and MacKinnon}{Davidson and
  MacKinnon}{2000}]{davidsonmackinnon2000}
Davidson, R. and J.~G. MacKinnon (2000).
\newblock Bootstrap tests: how many bootstraps?
\newblock {\em Econometric Reviews\/}~{\em 19}, 55--68.

\bibitem[\protect\citeauthoryear{Davydov, Lifshits, and Smorodina}{Davydov
  et~al.}{1998}]{davydovetal1998}
Davydov, Y.~A., M.~A. Lifshits, and N.~V. Smorodina (1998).
\newblock {\em Local Properties of Distributions of Stochastic Functionals},
  Volume 173 of {\em Translations of Mathematical Monographs}.
\newblock American Mathematical Society, Providence, RI.

\bibitem[\protect\citeauthoryear{Donald and Lang}{Donald and
  Lang}{2007}]{donaldlang2007}
Donald, S.~G. and K.~Lang (2007).
\newblock Inference with difference-in-differences and other panel data.
\newblock {\em Review of Economics and Statistics\/}~{\em 89}, 221--233.

\bibitem[\protect\citeauthoryear{Feng, He, and Hu}{Feng
  et~al.}{2011}]{fengetal2010}
Feng, X., X.~He, and J.~Hu (2011).
\newblock Wild bootstrap for quantile regression.
\newblock {\em Biometrika\/}~{\em 94}, 995--999.

\bibitem[\protect\citeauthoryear{Ghosh, Parr, Singh, and Babu}{Ghosh
  et~al.}{1984}]{ghoshetal1984}
Ghosh, M., W.~C. Parr, K.~Singh, and G.~J. Babu (1984).
\newblock A note on bootstrapping the sample median.
\newblock {\em Annals of Statistics\/}~{\em 12}, 1130--1135.

\bibitem[\protect\citeauthoryear{Gon\c{c}alves and White}{Gon\c{c}alves and
  White}{2005}]{goncalveswhite2005}
Gon\c{c}alves, S. and H.~White (2005).
\newblock Bootstrap standard error estimates for linear regression.
\newblock {\em Journal of the American Statistical Association\/}~{\em 100},
  970--979.

\bibitem[\protect\citeauthoryear{Graham}{Graham}{2008}]{graham2008}
Graham, B.~S. (2008).
\newblock Identifying social interactions through conditional variance
  restrictions.
\newblock {\em Econometrica\/}~{\em 76}, 643--660.

\bibitem[\protect\citeauthoryear{Gutenbrunner, Jur{\^e}ckov{\'a}, Koenker, and
  Portnoy}{Gutenbrunner et~al.}{1993}]{gutenbrunneretal1993}
Gutenbrunner, C., J.~Jur{\^e}ckov{\'a}, R.~Koenker, and S.~Portnoy (1993).
\newblock Tests of linear hypotheses based on regression rank scores.
\newblock {\em Journal of Nonparametric Statistics\/}~{\em 2}, 307--333.

\bibitem[\protect\citeauthoryear{Hahn}{Hahn}{1995}]{hahn1995}
Hahn, J. (1995).
\newblock Bootstrapping quantile regression estimators.
\newblock {\em Econometric Theory\/}~{\em 11}, 105--121.

\bibitem[\protect\citeauthoryear{Hall}{Hall}{1992}]{hall1992}
Hall, P. (1992).
\newblock {\em The Bootstrap and Edgeworth Expansion}.
\newblock Springer, New York.

\bibitem[\protect\citeauthoryear{Hendricks and Koenker}{Hendricks and
  Koenker}{1992}]{hendrickskoenker1992}
Hendricks, W. and R.~Koenker (1992).
\newblock Hierarchical spline models for conditional quantiles and the demand
  for electricity.
\newblock {\em Journal of the American Statistical Association\/}~{\em 87},
  58--68.

\bibitem[\protect\citeauthoryear{Higham}{Higham}{2008}]{higham2008}
Higham, N.~J. (2008).
\newblock {\em Functions of Matrices: Theory and Computation}.
\newblock Society for Industrial \& Applied Mathematics, Philadelphia, PA.

\bibitem[\protect\citeauthoryear{Hjort and Pollard}{Hjort and
  Pollard}{1993}]{hjortpollard1993}
Hjort, N. and D.~Pollard (1993).
\newblock Asymptotics for minimisers of convex processes.
\newblock Statistical Research Report, University of Oslo.

\bibitem[\protect\citeauthoryear{Karlin, Cameron, and Williams}{Karlin
  et~al.}{1981}]{karlinetal1981}
Karlin, S., E.~C. Cameron, and P.~T. Williams (1981).
\newblock Sibling and parent-offspring correlation estimation with variable
  family size.
\newblock {\em Proceedings of the National Academy of Sciences\/}~{\em 78},
  2664--2668.

\bibitem[\protect\citeauthoryear{Kato}{Kato}{2011}]{kato2011}
Kato, K. (2011).
\newblock A note on moment convergence of bootstrap {M}-estimators.
\newblock {\em Statistics \& Decisions\/}~{\em 28}, 51--61.

\bibitem[\protect\citeauthoryear{Kloek}{Kloek}{1981}]{kloek1981}
Kloek, T. (1981).
\newblock {OLS} estimation in a model where a microvariable is explained by
  aggregates and contemporaneous disturbances are equicorrelated.
\newblock {\em Econometrica\/}~{\em 49}, 205--207.

\bibitem[\protect\citeauthoryear{Knight}{Knight}{1998}]{knight1998}
Knight, K. (1998).
\newblock Limiting distributions of {\emph{l}\textsubscript{1}} regression
  estimators under general conditions.
\newblock {\em Annals of Statistics\/}~{\em 26}, 756--770.

\bibitem[\protect\citeauthoryear{Koenker}{Koenker}{2004}]{koenker2004}
Koenker, R. (2004).
\newblock Quantile regression for longitudinal data.
\newblock {\em Journal of Multivariate Analysis\/}~{\em 91}, 74--89.

\bibitem[\protect\citeauthoryear{Koenker}{Koenker}{2005}]{koenker2005}
Koenker, R. (2005).
\newblock {\em Quantile Regression}.
\newblock Cambridge University Press, New York.

\bibitem[\protect\citeauthoryear{Koenker}{Koenker}{2013}]{koenker2013}
Koenker, R. (2013).
\newblock {\em quantreg: Quantile Regression}.
\newblock R package version 5.05.

\bibitem[\protect\citeauthoryear{Koenker and Bassett}{Koenker and
  Bassett}{1978}]{koenkerbassett1978}
Koenker, R. and G.~Bassett (1978).
\newblock Regression quantiles.
\newblock {\em Econometrica\/}~{\em 46}, 33--50.

\bibitem[\protect\citeauthoryear{Koenker and Machado}{Koenker and
  Machado}{1999}]{koenkermachado1999}
Koenker, R. and J.~A. Machado (1999).
\newblock Goodness of fit and related inference processes for quantile
  regression.
\newblock {\em Journal of the American Statistical Association\/}~{\em 94},
  1296--1310.

\bibitem[\protect\citeauthoryear{Kosorok}{Kosorok}{2003}]{kosorok2003}
Kosorok, M.~R. (2003).
\newblock Bootstraps of sums of independent but not identically distributed
  stochastic processes.
\newblock {\em Journal of Multivariate Analysis\/}~{\em 84}, 299--318.

\bibitem[\protect\citeauthoryear{Kosorok}{Kosorok}{2008}]{kosorok2008}
Kosorok, M.~R. (2008).
\newblock {\em Introduction to Empirical Processes and Semiparametric
  Inference}.
\newblock Springer Series in Statistics. Springer, New York.

\bibitem[\protect\citeauthoryear{Krueger}{Krueger}{1999}]{krueger1999}
Krueger, A.~B. (1999).
\newblock Experimental estimates of education production functions.
\newblock {\em Quarterly Journal of Economics\/}~{\em 114}, 497--532.

\bibitem[\protect\citeauthoryear{Liang and Zeger}{Liang and
  Zeger}{1986}]{liangzeger1986}
Liang, K.-Y. and S.~L. Zeger (1986).
\newblock Longitudinal data analysis using generalized linear models.
\newblock {\em Biometrika\/}~{\em 73}, 13--22.

\bibitem[\protect\citeauthoryear{Lifshits}{Lifshits}{1982}]{lifshits1982}
Lifshits, M.~A. (1982).
\newblock On the absolute continuity of distributions of functionals of random
  processes.
\newblock {\em Theory of Probability and Its Applications\/}~{\em 27},
  600--607.

\bibitem[\protect\citeauthoryear{Liu}{Liu}{1988}]{liu1988}
Liu, R.~Y. (1988).
\newblock Bootstrap procedures under some non-{I.I.D.}~models.
\newblock {\em Annals of Statistics\/}~{\em 16}, 1696--1708.

\bibitem[\protect\citeauthoryear{MacKinnon and Webb}{MacKinnon and
  Webb}{2015}]{mackinnonwebb2014}
MacKinnon, J.~G. and M.~D. Webb (2015).
\newblock Wild bootstrap inference for wildly different cluster sizes.
\newblock Department of Economics, Queen's University Working Paper 2-2015.

\bibitem[\protect\citeauthoryear{Mammen}{Mammen}{1992}]{mammen1992}
Mammen, E. (1992).
\newblock {\em When Does Bootstrap Work? Asymptotic Results and Simulations},
  Volume~77 of {\em Lecture Notes in Statistics}.
\newblock Springer, New York.

\bibitem[\protect\citeauthoryear{Moulton}{Moulton}{1990}]{moulton1990}
Moulton, B.~R. (1990).
\newblock An illustration of a pitfall in estimating the effects of aggregate
  variables on micro units.
\newblock {\em Review of Economics \& Statistics\/}~{\em 72}, 334--338.

\bibitem[\protect\citeauthoryear{Parente and {Santos Silva}}{Parente and
  {Santos Silva}}{2015}]{parentesantossilva2013}
Parente, P.~M. and J.~M. {Santos Silva} (2015).
\newblock Quantile regression with clustered data.
\newblock {\em Journal of Econometric Methods\/}, forthcoming.

\bibitem[\protect\citeauthoryear{Parzen, Wei, and Ying}{Parzen
  et~al.}{1994}]{parzenetal1994}
Parzen, M.~I., L.-J. Wei, and Z.~Ying (1994).
\newblock A resampling method based on pivotal estimating functions.
\newblock {\em Biometrika\/}~{\em 81}, 341--350.

\bibitem[\protect\citeauthoryear{Pollard}{Pollard}{1982}]{pollard1982}
Pollard, D. (1982).
\newblock A central limit theorem for empirical processes.
\newblock {\em Journal of the Australian Mathematical Mathematical Society
  (Series A)\/}~{\em 33}, 235--248.

\bibitem[\protect\citeauthoryear{Pollard}{Pollard}{1990}]{pollard1991b}
Pollard, D. (1990).
\newblock {\em Empirical Processes: Theory and Applications}, Volume~2 of {\em
  NSF-CBMS Regional Conference Series in Probability and Statistics}.
\newblock Institute for Mathematical Statistics.

\bibitem[\protect\citeauthoryear{Portnoy}{Portnoy}{2014}]{portnoy2014}
Portnoy, S. (2014).
\newblock The jackknife's edge: Inference for censored regression quantiles.
\newblock {\em Computational Statistics and Data Analysis\/}~{\em 72},
  273--281.

\bibitem[\protect\citeauthoryear{Powell}{Powell}{1986}]{powell1986}
Powell, J.~L. (1986).
\newblock Censored regression quantiles.
\newblock {\em Journal of Econometrics\/}~{\em 32}, 143--155.

\bibitem[\protect\citeauthoryear{van~der Vaart}{van~der
  Vaart}{1998}]{vandervaart1998}
van~der Vaart, A.~W. (1998).
\newblock {\em Asymptotic Statistics}.
\newblock Cambridge University Press.

\bibitem[\protect\citeauthoryear{van~der Vaart and Wellner}{van~der Vaart and
  Wellner}{1996}]{vandervaartwellner1996}
van~der Vaart, A.~W. and J.~A. Wellner (1996).
\newblock {\em Weak Convergence and Empirical Processes: With Applications to
  Statistics}.
\newblock Springer Series in Statistics. Springer, New York.

\bibitem[\protect\citeauthoryear{van Hemmen and Ando}{van Hemmen and
  Ando}{1980}]{vonhemmeando1980}
van Hemmen, J.~L. and T.~Ando (1980).
\newblock An inequality for trace ideals.
\newblock {\em Communications in Mathematical Physics\/}~{\em 76}, 143--148.

\bibitem[\protect\citeauthoryear{Wang}{Wang}{2009}]{wang2009}
Wang, H. (2009).
\newblock Inference on quantile regression for heteroscedastic mixed models.
\newblock {\em Statistica Sinica\/}~{\em 19}, 1247--1261.

\bibitem[\protect\citeauthoryear{Wang and He}{Wang and He}{2007}]{wanghe2007}
Wang, H. and X.~He (2007).
\newblock Detecting differential expressions in genechip microarray studies: a
  quantile approach.
\newblock {\em Journal of American Statistical Association\/}~{\em 102},
  104--112.

\bibitem[\protect\citeauthoryear{Webb}{Webb}{2014}]{webb2013}
Webb, M.~D. (2014).
\newblock Reworking wild bootstrap based inference for clustered errors.
\newblock Department of Economics, University of Calgary Working Paper.

\bibitem[\protect\citeauthoryear{Word, Johnston, Bain, Fulton, Achilles, Lintz,
  Folger, and Breda}{Word et~al.}{1990}]{wordetal1990}
Word, E., J.~Johnston, H.~P. Bain, B.~D. Fulton, C.~M. Achilles, M.~N. Lintz,
  J.~Folger, and C.~Breda (1990).
\newblock The state of {T}ennessee's student/teacher achievement ratio ({STAR})
  project: Technical report 1985-1990.
\newblock Report, Tennessee State University, Center of Excellence for Research
  in Basic Skills.

\bibitem[\protect\citeauthoryear{Wu}{Wu}{1986}]{wu1986}
Wu, C. F.~J. (1986).
\newblock Jackknife, bootstrap and other resampling methods in regression
  analysis.
\newblock {\em Annals of Statistics\/}~{\em 14}, 1261--1295.

\bibitem[\protect\citeauthoryear{Yoon and Galvao}{Yoon and
  Galvao}{2013}]{yoongalvao2013}
Yoon, J. and A.~Galvao (2013).
\newblock Robust inference for panel quantile regression models with individual
  effects and serial correlation.
\newblock Unpublished manuscript.

\end{thebibliography}
\pre{
\putbib[qspec]
\end{bibunit}
}

\end{document}